\documentclass[11pt, letter, reqno]{amsart}
\usepackage[
    top=1in, bottom=0.9in, inner=0.9in, outer=0.9in,
    marginparwidth=2cm 
    ]{geometry}

\usepackage{amsmath,amsthm,amsxtra,amssymb,mathtools,latexsym,hyperref,xcolor,appendix,listings} 
\usepackage{tikz,overpic,graphicx,eepic,subfigure,caption}
\usepackage{enumitem} 
\setlist[enumerate]{label=\rm{(\roman*)}}
\usepackage{algorithm,algorithmicx,multirow,algpseudocode,xargs}
\usepackage{verbatim,float,extarrows,bbm,cleveref}
\usepackage[matrix,arrow,curve,cmtip]{xy} 
\usepackage{ifthen} 

\usepackage{tikz-cd}
\usetikzlibrary{positioning}

\allowdisplaybreaks[3]

\makeatletter
\@namedef{subjclassname@2020}{\textup{2020} Mathematics Subject Classification}
\makeatother

\theoremstyle{plain}
    \newtheorem{theorem}{Theorem}[section]
    \newtheorem*{theorem*}{Theorem}
    \newtheorem{lemma}[theorem]{Lemma}

    \newtheorem*{conj*}{Conjecture}
    \newtheorem{prop}[theorem]{Proposition}
    \newtheorem{cor}[theorem]{Corollary}

\theoremstyle{definition}
    \newtheorem{definition}[theorem]{Definition}
    \newtheorem{remark}[theorem]{Remark}  
    
    \newtheorem{example}[theorem]{Example}

    \newtheorem*{asss}{The Assumptions}

\theoremstyle{remark}

\numberwithin{equation}{section}
\numberwithin{theorem}{section}
\numberwithin{table}{section}
\numberwithin{figure}{section}

\makeatletter
\def\th@plain{%
	\thm@notefont{}
	\itshape 
}
\def\th@definition{%
	\thm@notefont{}
	\normalfont 
}
\makeatother

\makeatletter
\newcommand{\biggg}{\bBigg@{3}}

\newcommand{\Biggg}{\bBigg@{3.5}}

\newcommand{\bigggg}{\bBigg@{4}}

\newcommand{\Bigggg}{\bBigg@{4.5}}

\makeatother

\newcommand{\n}{\mathbb{N}}

\newcommand{\cx}{\mathbb{C}}
\newcommand{\ccx}{\widehat{\mathbb{C}}}
\newcommand{\real}{\mathbb{R}}

\providecommand{\defn}[1]{\emph{#1}}  

\newcommandx{\floor}[2][1=,]{\mathopen#1\lcloor #2 \mathclose#1\rfloor}
\newcommandx{\ceil}[2][1=,]{\mathopen#1\lceil #2 \mathclose#1\rceil}
\newcommandx{\abs}[2][1=,]{\mathopen#1\lvert #2 \mathclose#1\rvert}

\newcommand{\holder}{H\"{o}lder }

\newcommand{\holderexp}{\alpha} 
\newcommandx{\holderspace}[4][1=, 2=\holderexp, 3=X, 4=d, usedefault]{ C^{0, #2}\mathopen#1(#3, #4 \mathclose#1) }
\newcommandx{\holderspacesphere}[2][1=\holderexp, 2=d]{ C^{0, #1}(S^2, #2) }

\newcommandx{\listings}[2][2=n]{ #1_{1}, \, \dots, \, #1_{#2} }
\newcommandx{\collection}[3][1=, 3=n]{#1\{ #2_{1}, \, \dots, \, #2_{#3} #1\}} 
\newcommandx{\queue}[3][1=, 3=n]{ \mathopen#1( #2_{1}, \, \dots, \, #2_{#3} \mathclose#1)} 
\newcommandx{\sequen}[3][1=, 3=n]{\mathopen{}#1\{ #2 \mathclose{} #1\}_{#3 \in \n}}
\newcommandx{\parentheses}[2][1=,]{\mathopen#1(#2 \mathclose#1)}
\newcommandx{\set}[2][1=,]{\mathopen#1\{#2 \mathclose#1\}}
\newcommand{\juxtapose}[2]{ #1, \, #2 }

\newcommand{\define}{\coloneqq}
\newcommand{\mapping}{\rightarrow}
\newcommandx{\functional}[3][1=,]{ \mathopen#1 \langle #2, #3 \mathclose#1 \rangle} 
\newcommandx{\Functional}[2]{\bigl\langle #1, #2 \bigr\rangle}

\newcommand{\describe}{:}  

\newcommand{\crit}[1]{\operatorname{crit}{#1}}  
\newcommand{\post}[1]{\operatorname{post}{#1}}  
\newcommand{\supp}[1]{\operatorname{supp}{#1}}
\newcommandx{\card}[2][1=,]{\operatorname{card}\mathopen#1(#2 \mathclose#1)}
\newcommandx{\inte}[2][1=,]{\operatorname{inte}\mathopen#1(#2 \mathclose#1)}
\newcommandx{\diam}[3][1=,]{\operatorname{diam}_{#2}\mathopen#1(#3 \mathclose#1) }
\newcommandx{\diamn}[3][1=,]{\operatorname{diam}_{#2} #3   }
\newcommandx{\mesh}[3][1=,]{\operatorname{mesh}_{#2}\mathopen#1(#3 \mathclose#1) }
\newcommandx{\sgn}[2][1=,]{\operatorname{sgn}\mathopen#1(#2 \mathclose#1) }
\newcommand{\id}[1]{\operatorname{id}_{#1}}
\newcommandx{\interior}[2][1=,]{\operatorname{int}_\circ\mathopen#1(#2 \mathclose#1)} 
\newcommandx{\interiorn}[2][1=,]{\operatorname{int}_\circ #2  } 
\newcommand{\indicator}[1]{\mathbbm{1}_{#1}}
\newcommandx{\myexp}[2][1=,]{\operatorname{exp}\mathopen#1(#2 \mathclose#1) }
\newcommand{\tangent}[3]{ {#1}^{*}_{#2, #3} }

\newcommandx{\norm}[2][1=,]{\mathopen#1\|#2 \mathclose#1\|}
\newcommandx{\uniformnorm}[2][1=,]{\mathopen #1\| #2 \mathclose #1\|_{\infty}}
\newcommandx{\sumnorm}[2][1=,]{\mathopen #1\|#2 \mathclose #1\|_{\operatorname{sum}}}
\newcommandx{\normcontinuous}[4][1=, 4=]{\mathopen #1\| #2 \mathclose #1\|_{C \mathopen #4( #3 \mathclose #4)}}
\newcommandx{\holderseminorm}[3][1=, 3=\holderexp]{ \abs[#1]{#2}_{#3} }
\newcommandx{\holdernorm}[6][1=, 3=\holderexp, 5=d, 6=]{\mathopen #1\| #2 \mathclose #1\|_{C^{0, #3} \mathopen #6(#4, #5 \mathclose #6)} }

\newcommand{\probmea}[1]{\mathcal{P}(#1)}

\newcommandx{\ergmea}[2][1=S^2, 2=f]{ \mathcal{M}_{\operatorname{erg}}(#1, #2) }
\newcommandx{\invmea}[2][1=S^2, 2=f]{ \mathcal{M}(#1, #2) }

\newcommand{\black}{b} 
\newcommand{\white}{w} 
\newcommand{\blacktile}{X^{0}_{\black}}
\newcommand{\whitetile}{X^{0}_{\white}}

\newcommand{\Tile}[1]{\mathbf{X}^{#1}(f, \mathcal{C})}

\newcommand{\Edge}[1]{\mathbf{E}^{#1}(f, \mathcal{C})}

\newcommand{\cflower}[2]{\overline{W}^{#1}(#2)} 

\newcommandx{\ccFTile}[4][1=F,]{\mathbf{\mathfrak{X}}^{#2}_{#3 #4}(#1, \mathcal{C})} 

\newcommandx{\neighbortile}[5][1=F]{\mathbf{\mathfrak{X}}^{#2}_{#3 #4}(#1, \mathcal{C}, #5)}

\newcommandx{\ccndegF}[5][1=,]{\deg_{#2 #3}\mathopen#1(F^{#4}, #5 \mathclose#1)}
\newcommandx{\Deg}[3][1=,]{ \operatorname{Deg}\mathopen#1(F^{#2}, #3 \mathclose#1) }

\newcommand{\potential}{\phi} 
\newcommand{\normpotential}{\overline{\phi}} 

\newcommandx{\holdertilespace}[3][1=\holderexp, 3=d]{ C^{0, #1}\bigl(X^{0}_{#2}, #3 \bigr) }
\newcommandx{\splholderspace}[2][1=\holderexp, 2=d]{ C^{0, #1}\bigl(\blacktile, #2 \bigr) \times C^{0, #1}\bigl(\whitetile, #2 \bigr) }

\newcommand{\ruelleopt}[1][\phi]{\mathcal{L}_{#1}}  
\newcommand{\normopt}{\mathcal{L}_{\normpotential}}  

\newcommand{\moduconst}{\tau}
\newcommand{\modufun}{\eta}
\newcommandx{\moduspace}[4][1=\moduconst, 2=\modufun, 3=X, 4=d, usedefault]{ C^{#1}_{#2} ( #3, #4) }
\newcommandx{\modutilespace}[4][1=\moduconst, 2=\modufun, 4=d]{ C^{#1}_{#2} \bigl( X^{0}_{#3}, #4 \bigr) }
\newcommandx{\splmoduspace}[3][1=\moduconst, 2=\modufun, 3=d]{ C^{#1}_{#2} \bigl( X^{0}_{\black}, #3 \bigr) \times C^{#1}_{#2} \bigl( X^{0}_{\white}, #3 \bigr) }

\newcommandx{\deltameasure}[2][1=n]{V_{#1}(#2)}
\newcommandx{\birkhoffmeasure}[1][1=n]{\Sigma_{#1}}
\newcommandx{\periodorbit}[2][1=n, 2=f]{\operatorname{Per}_{#1}(#2)}

\newcommandx{\mutifun}[2][1=\ell, 2=\varphi]{\queue{#2}[#1]}
\newcommandx{\mutiavg}[2][1=\ell, 2=\alpha]{\queue{#2}[#1]}

\renewcommand{\leq}{\leqslant}

\renewcommand{\geq}{\geqslant}

\renewcommand{\:}{\colon}

\newcommand{\PPP}{\mathcal{P}}

\newcommand{\N}{\mathbb{N}}

\newcommand{\Q}{\mathbb{Q}}
\newcommand{\R}{\mathbb{R}}

\newcommand{\cA}{\mathcal{A}}
\newcommand{\cB}{\mathcal{B}}

\newcommand{\cE}{\mathcal{E}}

\newcommand{\cI}{\mathcal{I}}

\newcommand{\cK}{\mathcal{K}}
\newcommand{\cL}{\mathcal{L}}
\newcommand{\cM}{\mathcal{M}}

\newcommand{\cO}{\mathcal{O}}
\newcommand{\cP}{\mathcal{P}}
\newcommand{\cQ}{\mathcal{Q}}

\newcommand{\cS}{\mathcal{S}}

\newcommand{\cV}{\mathcal{V}}
\newcommand{\cW}{\mathcal{W}}

\renewcommand{\th}{\widetilde{h}}

\providecommand{\Absbig}[1]{\bigl\lvert#1\bigr\rvert}
\providecommand{\Absbigg}[1]{\biggl\lvert#1\biggr\rvert}

\providecommand{\norm}[1]{\|#1\|}

\renewcommand{\=}{\coloneqq}

\newcommand{\Sing}{\operatorname{Sing}}

\begin{document}


\title{Computable Thermodynamic Formalism}
    \author{Ilia~Binder \and Qiandu~He \and Zhiqiang~Li \and Xianghui~Shi}


\address{Ilia~Binder, Department of Mathematics, University of Toronto, Bahen Centre, 40 St. George St., Toronto, Ontario, M5S 2E4, CANADA}
\email{ilia@math.toronto.edu}

\address{Qiandu~He, School of Mathematical Sciences, Peking University, Beijing 100871, CHINA}
\email{heqiandu@stu.pku.edu.cn}
    
\address{Zhiqiang~Li, School of Mathematical Sciences \& Beijing International Center for Mathematical Research, Peking University, Beijing 100871, China}
\email{zli@math.pku.edu.cn}

\address{Xianghui~Shi, Beijing International Center for Mathematical Research, Peking University, Beijing 100871, China}
\email{xhshi@pku.edu.cn}




\begin{abstract}
	We investigate the theory of thermodynamic formalism from the perspective of computable analysis, with a special focus on the computability of equilibrium states. 
	Specifically, we develop two complementary general approaches to verify the computability of equilibrium states for nonuniformly expanding computable dynamical systems. 
	The first approach applies to dynamical systems whose topological pressure functions admit effective approximations and whose measure-theoretic entropy functions are upper semicontinuous.
	As a concrete application, we establish the computability of the equilibrium states for Misiurewicz--Thurston rational maps with \holder continuous potentials. 
	The second approach exploits prescribed Jacobians of equilibrium states through a local analysis and applies to settings where the measure-theoretic entropy functions may lack upper semicontinuity.
\end{abstract}

\maketitle
\tableofcontents

\section{Introduction}
\label{sec:Introduction}



The study of computability in dynamical systems lies at the intersection of mathematics, physics, and computer science, and has become increasingly vital for understanding complex physical phenomena through computational perspectives. 
A fundamental insight driving this field is that while chaotic systems exhibit sensitive dependence on initial conditions, 
their typical statistical behaviors are often described by computable objects. 
This interplay between chaos and computability underscores a fundamental principle: the more expansive and chaotic a system's dynamics, the more tractable its typical behavior becomes. 

This perspective builds on the well-established principle that higher dynamical complexity, particularly through uniform expansion or hyperbolicity, facilitates the analysis of associated dynamical invariants. 
In such uniformly expanding systems, the dynamical mechanisms driving long-term behavior become sufficiently regular to facilitate effective computation of key quantities such as entropy, pressure, and key invariant measures. 
Thus, the central challenge, and our primary focus, lies in extending these computational methodologies to systems where uniform expansion fails.

Thermodynamic formalism emerges as the natural framework for studying chaotic dynamics from a statistical viewpoint.
This theory, which draws inspiration from statistical mechanics, was pioneered by Ruelle, Sinai, Bowen, and others in the 1970s (\cite{dobruschin1968description, sinai1972gibbs,bowen1975equilibrium,walters1982introduction}). 
Since its birth, thermodynamic formalism has been extensively applied in various classical contexts and has remained at the frontier and core of research in dynamical systems.
It focuses on \emph{equilibrium states}, which are invariant measures that maximize pressure functionals encoding both entropy\footnote{
The concepts of entropy in dynamical systems have their roots in the early works on the notions of entropy by Boltzmann and Gibbs (statistical mechanics, 1875), von Neumann (quantum mechanics, 1932), and Shannon (information theory, 1948). 
These notions of entropy are all designed to describe the complexity of their respective systems or objects. 
In recent years, there have been exciting developments and diverse applications of entropy and complexity theory, see e.g.\ Braverman's report at the International Congress of Mathematicians in 2022 \cite{braverman2022communication}. 
}
and integrals of potentials.
Crucially, equilibrium states characterize the typical behavior of the dynamics and thus possess central mathematical and physical significance.
In many settings, equilibrium states describe the weighted distribution of random backward orbits (see e.g.\ \cite{hawkins2003maximal,li2018equilibrium}), iterated preimages, and periodic orbits. 

In complex dynamics, Brolin--Lyubich measures \cite{brolin1965invariant,lyubich1982maximum} are measures of maximal entropy for rational maps. 
A uniform algorithm to compute such measures was developed in \cite{binder2011computability}.
This complements the discovery of polynomials with computable coefficients but non-computable Julia sets, as explored in the pioneering works of Braverman and Yampolsky \cite{braverman2006noncomputable, braverman2009computability}, which can be traced back to a question posed by Milnor (see \cite[Section~1]{braverman2006noncomputable}).
For further research on algorithmic aspects of Julia sets, we refer the reader to recent works of Rojas and Yampolsky \cite{rojas2021real} and Dudko and Yampolsky \cite{dudko2021computational} and references therein.

The computability of Brolin--Lyubich measures presents an apparent paradox. 
Intuitively, one might expect a measure to contain more information than its support.
However, computable analysis reveals the existence of a computable invariant probability measure (the Brolin--Lyubich measure) whose support is non-computable.
Indeed, this paradox can be reconciled by interpreting these results as reflecting distinct computability properties of the system from geometric and statistical perspectives. 
From this statistical viewpoint, questions regarding the computability of equilibrium states via thermodynamic formalism gain critical significance.

In this article, we study thermodynamic formalism from the point of view of computable analysis, with a special focus on the computability of equilibrium states in dynamical systems. 
Classical hyperbolic systems admit well-developed techniques and are generally regarded as well-understood. 
By contrast, nonuniformly hyperbolic systems resist conventional approaches.
The relaxation of uniform expansion requirements therefore represents a fundamental challenge in dynamical systems. Our central contribution lies in demonstrating that statistical computability can coexist with dynamical complexity, even in nonuniformly hyperbolic regimes.

We develop two complementary approaches that establish fundamental links between the computability of thermodynamic quantities and equilibrium states.
Each approach offers distinct advantages and ranges of applicability, designed to cover a broad class of dynamical systems. 

The first approach utilizes the set of tangent functionals to the topological pressure function to link the computability of the topological pressure to that of the equilibrium state. 
This method is suitable for dynamical systems where the measure-theoretic entropy map is upper semicontinuous and the topological pressure function can be effectively approximated.
For dynamical systems whose measure-theoretic entropy functions may lack upper semicontinuity, we implement the second approach, which establishes the computability of equilibrium states through local (rather than global) considerations. 
The novelties of our second approach include the following two aspects: 
(i) on the thermodynamic formalism side, using general transfer operators to study equilibrium states of full measure away from singular points with prescribed Jacobians, and 
(ii) on the computable analysis side, the verification of Jacobians away from singular points, and the construction of a recursively compact set of measures to exclude atomic measures supported on singular points.

To demonstrate applications of these approaches, we focus on expanding Thurston maps as primary case studies, whose ergodic theory has been actively studied.
A \emph{Thurston map} is a branched covering map on a topological $2$-sphere that is not a homeomorphism and satisfies the postcritically-finite condition (i.e., every critical point has a finite forward orbit).
Such maps play a central role in the study of complex dynamics, and the most important examples are given by postcritically-finite rational maps on the Riemann sphere\footnote{
There has been active research on algorithmic aspects of Thurston maps. 
For example, Bonnet, Braverman, and Yampolsky \cite{sylvain2012thurston} devised an algorithm to determine whether a Thurston map is Thurston equivalent to a rational map; 
Hubbard and Schleicher \cite{hubbard49spider}, in the setting of quadratic rational maps, provided an algorithm that, given a convenient description of the Thurston map, outputs the coefficients of the rational map; 
Selinger, Rafi, and Yampolsky \cite{selinger2015constructive,rafi2020centralizers} investigated the decidability of Thurston equivalence.
}.
Inspired by Sullivan's dictionary and their interest in Cannon's Conjecture \cite{cannon1994combinatorial}, Bonk and Meyer \cite{bonk2010expanding,bonk2017expanding}, as well as Ha{\"i}ssinsky and Pilgrim \cite{haissinsky2009coarse}, studied a subclass of Thurston maps, called \emph{expanding Thurston maps}, by imposing some additional condition of weak expansion (see Definition~\ref{def:expanding_Thurston_maps}).

Based on these works, ergodic theory for expanding Thurston maps has been actively investigated in \cite{bonk2010expanding,bonk2017expanding, haissinsky2009coarse, li2015weak,li2017ergodic,li2018equilibrium,li2025ground,shi2025thermodynamic,shi2024uniqueness,shi2024entropy}.
In particular, the third-named author \cite{li2018equilibrium} developed the thermodynamic formalism and investigated the existence, uniqueness, and other ergodic properties of equilibrium states for expanding Thurston maps\footnote{cf.~the monograph \cite{li2017ergodic}.}. 
Notably, expanding Thurston maps exhibit weak expansion; they are neither expansive nor $h$-expansive. Furthermore, those with at least one periodic critical point are not even asymptotically $h$-expansive \cite{li2015weak}.
Recent advances by the third-named and fourth-named authors \cite{shi2024entropy} further demonstrated that the measure-theoretic entropy function is upper semicontinuous if and only if the expanding Thurston map has no periodic critical points.
Leveraging these results, we apply our methods to investigate the computability of measures of maximal entropy and equilibrium states for expanding Thurston maps, demonstrating the applicability of our two distinct approaches.

Our approaches to investigating the computability of equilibrium states extend way beyond the setting of expanding Thurston maps. See the discussions below.

\subsection{Main results}
\label{sub:Main results}

As mentioned earlier, two complementary approaches link the computability of thermodynamic quantities to that of equilibrium states.

The first approach (Theorem~\ref{theorem A}) applies to systems with upper semicontinuous measure-theoretic entropy functions. 
It shows that certain computability properties of the topological pressures guarantee the computability of the equilibrium states, with a concrete application to Misiurewicz--Thurston rational maps in Theorem~\ref{Application A}.
The second approach, established in Theorem~\ref{theorem B}, applies to systems without relying on the upper semicontinuity of the entropy, where the computability of equilibrium states arises from the computability of their prescribed Jacobians. 
As an application, we establish the computability of the measures of maximal entropy for computable expanding Thurston maps with computable critical points in Theorem~\ref{Application B}.

We adopt the conventions and terminology for computable analysis from \cite{weihrauch2000computable} and refer the reader to Section~\ref{sec:Preliminaries} for a more detailed introduction. 
Below, we introduce the uniformly computable systems that form the setting for our approaches.

We say that the quintuple $(X,\,\rho,\,\cS,\,\{X_n\}_{n\in\N},\,\{T_n\}_{n\in\N})$ is a \emph{uniformly computable system} if the following properties are satisfied:
\begin{enumerate}
    \smallskip
    \item $(X,\,\rho,\,\cS)$ is a computable metric space in which $X$ is recursively compact.
    \smallskip
    \item $X_n\subseteq X$ is open and $T_n\colon X\mapping X$ is a Borel-measurable function for each $n\in\N$.
    \smallskip 
    \item In $(X,\,\rho,\,\cS)$, $\{X_n\}_{n\in\N}$ is a sequence of uniformly lower semi-computable open sets and $\{T_n\}_{n\in\N}$ is a sequence of uniformly computable functions with respect to $\{X_n\}_{n\in\N}$.
\end{enumerate}

Here the computable structure of metric spaces is given in Definition~\ref{def:computablemetricspace}, the lower semi-computable openness of a set is an effective version of openness given in Definition~\ref{def:lower semicomputability and uniform version of open set}, and the recursive compactness of a set is an effective version of compactness (with an additional algorithmic covering procedure); see Definition~\ref{def:recursively compact}. Moreover, the computability of functions is given in Definition~\ref{def:Algorithm about computable functions}. 

In the first approach, we exploit the set of tangent functionals to the topological pressure function $P(T, \cdot)$ (see Section~\ref{sec:Approach I}) at the potential $\phi$ and demonstrate that the computability of the equilibrium state follows from certain computability properties of the topological pressure. We note that this set of tangent functionals is naturally identified with a subset of the set $\PPP(X)$ of Borel probability measures (see Remark~\ref{rm:identification}).
We denote by $\mathcal{E}(T,\potential)$ the set of equilibrium state(s) for a map $T$ and a potential $\potential$.
We refer the reader to Subsection~\ref{sub:Computable Analysis} for a detailed discussion on the computability of measures.

\begin{theorem}\label{theorem A}
    Let $(X,\,\rho,\,\cS,\,\{X_n\}_{n\in\N},\,\{T_n\}_{n\in\N})$ be a uniformly computable system with $X_n=X$ for all $n\in\N$, and $\{\phi_n\}_{n\in\N}$ be a sequence of uniformly computable functions $\phi_n\colon X\mapping\R$. Suppose $T_n$ has finite topological entropy, and the measure-theoretic entropy map $\nu\mapsto h_{\nu}(T_n)$ is upper semicontinuous on $\cM(X,T_n)$
    for each $n\in\N$. Assume that the following statements are true:
	\begin{enumerate}
        \smallskip
		\item There exists a sequence $\{\psi_{n,i}\}_{(n,i)\in\N^2}$ of uniformly computable functions $\psi_{n,i}\colon X\rightarrow\R$ such that the closure $\overline{D}_n$ of $D_n\=\{\psi_{n,i}:i\in\N\}$ in $C(X)$ contains a neighborhood of $\potential_n$ for each $n\in\N$ and $\{P(T_n,\psi_{n,i})\}_{(n,i)\in\N^2}$ is a sequence of uniformly upper semi-computable real numbers.
		\smallskip
		\item $\{P(T_n,\potential_n)\}_{n\in\N}$ is a sequence of uniformly lower semi-computable real numbers.
        \smallskip
        \item There exists a unique equilibrium state $\mu_n$ for $T_n$ and $\phi_n$ for each $n\in\N$.
	\end{enumerate}
    Then $\{\mu_n\}_{n\in\N}$ is a sequence of uniformly computable measures.
\end{theorem}

Here the space $\cM(X,T_n)$ of $T_n$-invariant Borel probability measures is equipped with the weak$^*$ topology, the computability properties of real numbers are recalled in Definitions~\ref{def:computability of real}~and~\ref{def:semicomputabilityoverR}, and the computable structure on the space $\PPP(X)$ of Borel probability measures is specified in Proposition~\ref{prop:recursively compactness of measure space}.

Applying Theorem~\ref{theorem A}, we establish the computability of the equilibrium states for computable Misiurewicz--Thurston rational maps (i.e., postcritically-finite rational maps without periodic critical points which are computable) and computable \holder continuous potentials. 
The existence and uniqueness of the equilibrium state for an expanding Thurston map follow from Theorem~\ref{thm:properties of equilibrium state}.

\begin{theorem}\label{Application A}
    Let $f \colon \widehat{\cx} \rightarrow \widehat{\cx}$ be a computable Misiurewicz--Thurston rational map, $\sigma$ be the chordal metric, and $\alpha\in (0,1]$. Assume that $\{\phi_n\}_{n\in\N}$ is a sequence of uniformly computable functions on $\widehat{\cx}$, and $\{Q_n\}_{n\in\N}$ is a sequence of uniformly computable real numbers. Suppose $\phi_n\in C^{0,\alpha}\bigl(\ccx,\sigma\bigr)$, $\cE(f,\phi_n)=\{\mu_n\}$, and $\abs{\phi_n}_{\alpha,\sigma}\leq Q_n$ for each $n\in\N$. Then $\{\mu_n\}_{n\in\N}$ is a sequence of uniformly computable measures.
\end{theorem}

The computable structure $\bigl(\ccx,\,\sigma,\,\cS\bigl(\ccx\bigr)\bigr)$ for $\ccx$ is given in Proposition~\ref{prop:hatCiscompactmetricspace}. Moreover, $\abs{\potential}_{\alpha,\sigma}$ denotes the H\"older constant of $\potential\in C^{0,\alpha}\bigl(\ccx,\sigma\bigr)$ (see (\ref{eq:Holder_constant})).

It is worth noting that Theorem~\ref{theorem A} can be applied to establish the computability of the equilibrium states for a wide range of dynamical systems with a unique equilibrium state and an upper semicontinuous measure-theoretic entropy map, such as rational maps with H\"older continuous hyperbolic potentials (see e.g.\ \cite{Denker1991ergodic}). 
Due to space limitations, we focus on the current examples and postpone further investigations to future work.  

For many dynamical systems, the measure-theoretic entropy map may not be upper semicontinuous, or verifying this property may be difficult. 
To overcome this issue, we establish a second approach by considering the prescribed Jacobians for equilibrium states.

For a compact metric space $(X,\rho)$, a Borel-measurable transformation $T\colon X \rightarrow X$, we say that $A\subseteq X$ is \emph{admissible} (for $T$) if $A,\,T(A)\in\cB(X)$ and $T|_A$ is injective. Given a Borel subset $Y\subseteq X$, and a Borel function $J\colon X \rightarrow [0,+\infty)$, define
    \begin{align}
		\cM(X,T;Y)&\=\bigl\{\mu\in\cP(X):\mu\bigl(T^{-1}(A)\cap Y\bigr)\leq\mu(A)\text{ for each Borel }A\subseteq X\bigr\},\quad\text{ and}\label{eq:defofcMonsubset}\\ 
        \cM(X,T;Y,J)&\= \biggl\{\mu\in\cP(X):\mu(T(A))\geq\int_A\!J\,\mathrm{d}\mu\text{ for each admissible set }A\subseteq Y\text{ for }T\biggr\}.\label{eq:defofcMforJonsubset}
	\end{align}

\begin{theorem}\label{theorem B}
	Let $(X,\,\rho,\,\cS,\,\{X_n\}_{n\in\N},\,\{T_n\}_{n\in\N})$ be a uniformly computable system, and $Y_n$ be an open subset of $X_n$ for each $n\in\N$. Assume that there exist two recursively enumerable sets $K,\,L$ with $L\subseteq\N\times K$, and a sequence $\{Y_{n,k}\}_{(n,k)\in L}$ of uniformly lower semi-computable open sets in $(X,\,\rho,\,\cS)$ with the properties that $Y_{n,k}$ is admissible for $T_n$, and $Y_n=\bigcup_{(n,k)\in L_n}Y_{n,k}$, where $L_n\coloneqq\{(n,k)\in L:k\in K\}$ for each $n\in\N$. 
    
    Assume that $\{J_n\}_{n\in\N}$ is a sequence of uniformly lower semi-computable functions $J_n\colon X \rightarrow[0,+\infty)$ with respect to $\{Y_n\}_{n\in\N}$ satisfying that $J_n$ is nonnegative on $Y_n$ and Borel for each $n\in\N$. Suppose there exists a sequence $\{\cK_n\}_{n\in\N}$ of uniformly recursively compact sets in $(\cP(X),\,W_{\rho},\,\cQ_{\cS})$ such that
    \begin{equation}\label{eq:propertyofKn}
        \cM(X,T_n;Y_n,J_n)\cap \mathcal{K}_n=\{\mu_n\}\quad\text{ for each }n\in\N.
    \end{equation}
    
	Then $\{\mu_n\}_{n\in\n}$ is a sequence of uniformly computable measures.
\end{theorem}

By Theorem~\ref{thm:prescribed Jacobian of equilibrium state}, in some settings (cf.\ Definition~\ref{def:admissible}), the set defined in (\ref{eq:defofcMforJonsubset}) describes the set of equilibrium states. Moreover, as applications of Theorem~\ref{theorem B}, in Theorems~\ref{thm:additionalassumptionone}~and~\ref{thm:additionalassumptiontwo}, we construct two corresponding families of uniformly recursively compact sets $\{\cK_n\}_{n\in\N}$ satisfying the additional assumptions for two classes of dynamical systems.

Theorem~\ref{theorem B} extends the methodologies developed in \cite[Theorem A]{binder2011computability} and \cite[Theorem 1.1]{binder2025computability}, introducing new techniques to handle both nonuniform expansion and the possible failure of upper semicontinuity of entropy.
Compared to the previous results, our approach significantly broadens applicability by relaxing the computability requirement on prescribed Jacobians to only lower semi-computability.


To demonstrate the applicability of Theorem~\ref{theorem B} while minimizing unnecessary technicalities, we employ it to prove the computability of the measures of maximal entropy for computable expanding Thurston maps. 
In fact, one can establish the computability of equilibrium states for these maps with H\"older continuous potentials using the cone method adapted for computability as in \cite{binder2025computability}.

\begin{theorem}\label{Application B}
    Let $\sigma$ be the chordal metric and $f\colon \ccx\rightarrow\ccx$ be an expanding Thurston map that is computable in the computable metric space $\bigl(\ccx,\,\sigma,\,\cS\bigl(\ccx\bigr)\bigr)$. 
    Assume that all critical points of $f$ are computable. 
    Then the measure of maximal entropy of $f$ is a computable measure.
\end{theorem}

We provide examples of such expanding Thurston maps and apply Theorem~\ref{Application B} to them at the end of Subsection~\ref{subsec:The measures of maximal entropy for computable expanding Thurston maps}. 
Even though the result in Theorem~\ref{Application B} is not uniform and considers only computable expanding Thurston maps on $\ccx$, Theorem~\ref{theorem B} can indeed be leveraged to extend the above result to the \emph{uniform} computability of the equilibrium states of computable expanding Thurston maps on a general topological $2$-sphere $S^2$. 
However, such applications would require a discussion on the computability of \emph{visual metrics} (see e.g.~\cite[Chapter~8]{bonk2017expanding}), which would take us too far astray from the core topic of this article.

The conditions in Theorems~\ref{theorem A} and~\ref{theorem B} illustrate some methodological distinctions. Specifically, Theorem~\ref{theorem A} provides an efficient pathway to proving computability when equilibrium states are unique and measure-theoretic entropy functions are regular. 
In contrast, Theorem~\ref{theorem B} becomes useful for systems with nonunique equilibrium states or those (potentially) lacking upper semicontinuous entropy, though it demands more sophisticated analysis.
Note that computable functions on some subsets are always continuous on the corresponding subsets. 
Theorem~\ref{theorem B} can be applied to some Borel-measurable dynamical systems (not necessarily continuous on the whole space), for example, Pomeau--Manneville maps with geometric potentials.



In fact, Theorem~\ref{theorem B} can be applied and further extended to study the computability of equilibrium states with \emph{time complexity} in the following settings:

\begin{enumerate}
\smallskip
    \item[(i)] Pomeau--Manneville maps with geometric potentials. 
\smallskip
    \item[(ii)] Holomorphic endomorphisms of $\mathbb{P}^k$ with $\log^q$-continuous hyperbolic potential.
\smallskip
    \item[(iii)] Nonuniformly expanding local diffeomorphism on smooth manifolds and H\"older continuous potentials with not very large oscillation (see \cite{binder2025nonuniformly}).
\end{enumerate}
We leave these investigations for future work.


\subsection{Strategy and organization}

In thermodynamic formalism, identifying equilibrium states often hinges on verifying the identity $P(T, \phi) = h_{\mu}(T) + \int \! \phi \,\mathrm{d} \mu$.
From the perspective of computable analysis, computing the topological pressure $P(T, \phi)$ and the integral $\int \! \phi \,\mathrm{d} \mu$ is straightforward. 
The core difficulty lies in evaluating the measure-theoretic entropy function $\mu \mapsto h_{\mu}(T)$, which lacks universal computational methods.

Our approaches reflect two distinct strategies for addressing this challenge. 
The first approach applies to dynamical systems where the measure-theoretic entropy $h_{\mu}(T)$ coincides with its upper semicontinuous regularization $\overline{h}_{\mu}(T)$. 
Under this hypothesis, we use convex analysis to characterize the set of equilibrium states as the set of tangent functionals to the topological pressure function (see Section~\ref{sec:Approach I}).
The second approach avoids relying on semicontinuity of entropy by instead investigating Jacobians for equilibrium states, a technique motivated by Rokhlin's formula (cf.~Proposition~\ref{prop:rokhlin'sformula}). 

The proofs of the two approaches (Theorems~\ref{theorem A} and~\ref{theorem B}) follow the same general strategy: we prove the recursive compactness of a subset $\mathcal{K}\subseteq\mathcal{E}(T, \potential)$. 
Then the assumption that $\mathcal{K} = \{\mu\}$ implies that the equilibrium state $\mu$ is computable (cf.~Proposition~\ref{thecomplementofrecursivelycompactset}~(i)). 
The strategy for identifying a compact subset $\cK\subseteq\cE(T,\phi)$ can be summarized as follows: for some dynamical systems, the measure-theoretic entropy has the following upper and lower bounds:
\begin{equation*}
    \overline{h}_{\mu}(T)\geq h_{\mu}(T)\geq\int\!\log(J_{\mu})\,\mathrm{d}\mu,
\end{equation*}
where $J_{\mu}$ is a Jacobian of $T$. Hence, we obtain that
\begin{equation*}
    \{\mu\in\cM(X,T):\overline{h}_{\mu}(T)+\functional{\mu}{\phi}=P(T,\phi)\}\supseteq \cE(T,\phi)\supseteq\{\mu\in\cM(X,T):\functional{\mu}{\phi+\log(J_{\mu})}=P(T,\phi)\}.
\end{equation*}
For simplicity, we denote the former (resp.\ latter) set by $\cE_1(T,\phi)$ (resp.\ $\cE_2(T,\phi)$). 
Indeed, by results in convex analysis (cf.\ Lemma~\ref{lem:representationoninvariantmeasure} and Proposition~\ref{prop:relationbetweentangentspaceandsetofequilibriumstate}), the set $\mathcal{E}_1(T,\potential)$ coincides with the set $C(X)^*_{\potential,P_T}$. 
Moreover, by our investigations in ergodic theory (cf.\ Theorem~\ref{thm:prescribed Jacobian of equilibrium state}), the set $\cE_2(T,\phi)$ can be described by $\cM(X,T;Y,J)$ (defined in (\ref{eq:defofcMforJonsubset})), where $Y\subseteq X$ is a Borel set and $J\colon X\mapping\R$ is a Borel function satisfying some assumptions (indeed, $\cE_2(T,\phi)\cap\cP(X;Y)$ coincides with $\cM(X,T;Y,J)\cap\cM(X,T)\cap\cP(X;Y)$). 
It is worth noting that the set $\mathcal{E}(T,\potential)$ of equilibrium states may not be weak$^*$ compact. 
However, the sets $C(X)^*_{\potential,P_T}$ and $\cM(X,T;Y,J)$ are always weak$^*$ compact. 
Hence, we demonstrate the recursive compactness of these sets instead of investigating $\cE(T,\phi)$ directly.

We now outline the proofs of the recursive compactness properties of $C(X)^*_{\potential,P_T}$ and $\cM(X,T;Y,J)$. 
For the set $C(X)^*_{\potential,P_T}$, by some results in convex analysis, it is the set of measures $\mu\in\cM(X,T)$ such that $\phi$ is the minimizing point of the operator defined by $h\mapsto P(T,h)-\functional{\mu}{h}$. 
Hence, the recursive compactness of the set $C(X)^*_{\potential,P_T}$ can be derived from some computability properties of the topological pressure function. 
Moreover, due to the convexity of such operator, the conditions can be further reduced to the computability properties of the topological pressure function near the potential $\phi$. 
For the set $\cM(X,T;Y,J)$, a local version of a method from \cite{binder2011computability,binder2025computability} is applied to check if a Jacobian for $T$ with respect to $\mu$ is greater than or equal to the given function $J$ in the ``good'' subset $Y$, which is a union of open and admissible sets. 
It is worth noting that we improve this method and relax the requirements for the domains of computability of dynamical systems.

Moreover, in the proof of Theorem~\ref{thm:prescribed Jacobian of equilibrium state}, the existence of singular points prevents the direct application of Ruelle operators to characterize the prescribed Jacobians of equilibrium states. 
To address this, we instead use transfer operators to provide an equivalent description of the Jacobians of invariant measures in Theorem~\ref{thm:description of Jacobian}. 
Combined with the Variational Principle, this yields Theorem~\ref{thm:prescribed Jacobian of equilibrium state}, which allows us to identify equilibrium states by verifying if a Jacobian for $T$ is greater than the prescribed function $J$ in the ``good'' subset $Y$.

Finally, we summarize the two approaches as follows. In the global approach via convex analysis, we assume the upper semicontinuity of the pressure function to ensure that $\cE_1(T,\phi) = \cE(T,\phi)$. 
In the local approach via transfer operator and Jacobian, the main challenge is to construct a recursively compact set $\cK \subseteq \PPP(X)$ that excludes measures with positive mass on a ``bad'' region. 
Since a uniform construction of such a set $\cK$ is not feasible for all systems, we hypothesize its existence in Theorem~\ref{theorem B} to obtain a general result, and provide explicit constructions for specific families in Theorems~\ref{thm:additionalassumptionone},~\ref{thm:additionalassumptiontwo},~and~\ref{Application B}.
More precisely, we consider the case where the ``bad'' regions are indeed sets of finitely many points. In Theorem~\ref{thm:additionalassumptionone} we use a sequence of open sets containing these ``bad'' regions to eliminate the mass supported on them, in Theorem~\ref{thm:additionalassumptiontwo} we investigate the dynamical systems that are uniformly contracting near all the periodic singular points, and in Theorem~\ref{Application B} study expanding Thurston maps, a family of maps which are uniformly expanding near all the periodic singular points.

\smallskip

We now describe the structure of this article.
After fixing some notation in Section~\ref{sec:Notation}, we review some notions and results in computable analysis and ergodic theory in Section~\ref{sec:Preliminaries}.

Section~\ref{sec:Approach I} focuses on the first (global) approach. 
We prove Theorem~\ref{theorem A} by establishing the recursive compactness of the set of tangent functionals to the topological pressure function at the potential with upper semicontinuous entropy. 

Section~\ref{sec:Approach II} is devoted to the second (local) approach as stated in Theorem~\ref{theorem B}, which uses the prescribed Jacobian with respect to an equilibrium state.
Through employing the transfer operators to establish Theorem~\ref{thm:prescribed Jacobian of equilibrium state} (in Subsection~\ref{subsec:jacobian and transfer operator}), which characterizes an equilibrium state in terms of its Jacobian, we complete the proof of Theorem~\ref{theorem B} in Subsection~\ref{subsec:proof of theorem B}. 
Theorems~\ref{thm:additionalassumptionone} and~\ref{thm:additionalassumptiontwo}, presented as consequences of Theorem~\ref{theorem B}, are stated in Subsection~\ref{sub:applications}.

Section~\ref{sec:Computable Analysis on thermodynamic formalism} examines the computability of the equilibrium states for expanding Thurston maps. 
We first provide the definitions and properties of these maps in Subsection~\ref{sub:Expanding Thurston_maps}.  
Then in Subsection~\ref{subsec:Misiurewicz--Thurston rational maps}, we apply Theorem~\ref{theorem A} to demonstrate the computability of the equilibrium state for a Misiurewicz--Thurston rational map and a H\"older continuous potential, thereby establishing Theorem~\ref{Application A}. 
Finally, Subsection~\ref{subsec:The measures of maximal entropy for computable expanding Thurston maps} addresses a broader class of expanding Thurston maps whose measure-theoretic entropy maps may lack upper semicontinuity. 
Here, we use Theorem~\ref{theorem B} to study the measures of maximal entropy and establish Theorem~\ref{Application B}.

\medskip

\noindent\textbf{Acknowledgments.}
Q.~H., Z.~L., and X.~S.\ were partially supported by Beijing Natural Science Foundation (JQ25001 and 1214021) and National Natural Science Foundation of China (12471083, 12101017, 12090010, and 12090015). I.~B.\ was partially supported by an NSERC Discovery grant. 
Q.~H.\ was also supported by Peking University Funding (7101303303 and 6201001846).

\section{Notation}
\label{sec:Notation}

The chordal metric $\sigma$ on the Riemann sphere $\ccx$ is defined as follows: $\sigma (z,w)  \=  \frac{2\abs{z - w}}{\sqrt{1 + \abs{z}^2} \sqrt{1 + \abs{w}^2}}$ for all $\juxtapose{z}{w} \in \cx$, and $\sigma(\infty,z) = \sigma(z,\infty)  \=  \frac{2}{\sqrt{1 + \abs{z}^2}}$ for all $z \in \cx$. 
Let $S^2$ denote an oriented topological $2$-sphere.
We use $\n$ to denote the set of integers greater than or equal to $1$ and $\N^*\=\bigcup_{k\in\N} \N^k$. We write $\n_0  \=  \{0\} \cup \n$ and $\N_0^*\=\{0\}\cup\N^*$. We denote by $\Q^+$ (resp.\ $\R^+$) the set of all positive rational (resp.\ real) numbers.
The symbol $\log$ denotes the natural logarithm. 
For $x \in \real$, we define $\lfloor x \rfloor$ as the greatest integer $\leqslant x$, $\lceil x \rceil$ as the smallest integer $\geqslant x$, and $x^+ = (x)^{+} \=\max\{x,\,0\}$. For a function $f\colon X\mapping\R$ on a set $X$, we define $f^+=(f)^+\colon X\mapping\R$ by $f^+(x)\coloneqq (f(x))^+$ for each $x\in X$.
The cardinality of a set $A$ is denoted by $\operatorname{card} A$.

Consider a map $f \colon X \rightarrow X$ on a set $X$. 
We write $f^n$ for the $n$-th iterate of $f$, and $f^{-n} \=  (f^n)^{-1}$, for each $n \in \n$.
We set $f^0  \=  \id{X}$, the identity map on $X$. 
For a real-valued function $\phi \colon X \rightarrow \real$, we write $S_n \phi(x) = S^f_n \phi(x)  \=  \sum_{m=0}^{n-1} \phi ( f^m(x) )$ for $x\in X$ and $n\in \n_0$. 
We omit the superscript $f$ when the map $f$ is clear from the context. When $n = 0$, by definition $S_0 \phi = 0$.

Let $(X,d)$ be a metric space. We denote by $\cB(X)$ the $\sigma$-algebra of all Borel subsets of $X$. For each subset $Y \subseteq X$, we denote the diameter of $Y$ by $\diamn{d}{Y}  \=  \sup\{d(x, y) \describe \juxtapose{x}{y} \in Y\}$, the interior of $Y$ by $\interiorn{Y}$, and the characteristic function of $Y$ by $\indicator{Y}$. 

For each $r\in\R$ and each $x\in X$, we denote the open (resp.\ closed) ball of radius $r$ centered at $x$ by $B_{d}(x,r) \= \{ y \in X : d(x,y) < r \}$. 
For each $r\in\R$ and each nonempty set $K\subseteq X$, we define $d(x,K)  \=  \inf_{y\in K}d(x,y)$, and $B_d(K,r) \= \{x\in X :  d(x,K)<r\}$. 
We often omit the metric $d$ in the subscript when it is clear from the context. 

For a compact metric space $(X,d)$, we denote by $C(X)$ the space of continuous functions from $X$ to $\real$, and by $\cM(X)$ (resp.\ $\PPP(X)$) the set of finite signed Borel measures (resp.\ Borel probability measures) on $X$.
Let $g\colon X\rightarrow X$ be a Borel-measurable transformation. We denote by $\cM(X,g)$ the set of $g$-invariant Borel probability measures on $X$.
Moreover, for each Borel subset $C\in\cB(X)$, $\PPP(X;C)$ denotes the set $\{\mu\in\PPP(X) : \mu(C)=1\}$. By the Riesz representation theorem, 
we can identify the dual of $C(X)$ with the space $\cM(X)$. For $\mu \in \cM(X)$, we use $\norm{\mu}$ to denote the total variation norm of $\mu$, $\supp{\mu}$ to denote the support of $\mu$, and 
\begin{equation*} 
    \functional{\mu}{u}  \=  \int \! u \,\mathrm{d}\mu
\end{equation*}
for each $\mu$-integrable Borel function $u$ on $X$. If we do not specify otherwise, we equip $C(X)$ with the uniform norm $\normcontinuous{\,\cdot\,}{X}  \=  \uniformnorm{\,\cdot\,}$, and equip $\cM(X)$, $\PPP(X)$, and $\cM(X, g)$ with the weak$^*$ topology. 

The space of real-valued \holder continuous functions with an exponent $\holderexp \in (0,1]$ on a compact metric space $(X, d)$ is denoted as $\holderspace$. 
For each $\phi \in \holderspace$, 
\begin{equation}  \label{eq:Holder_constant}
    \abs{\phi}_{\alpha,d}  \=  \sup\{ \abs{\phi(x) - \phi(y)} / d(x, y)^{\holderexp} \describe \juxtapose{x}{y} \in X, \, x \ne y \}.
\end{equation}

For a complete separable metric space $(X,d)$, we recall the Wasserstein--Kantorovich metric $W_{d}$ on $\PPP(X)$ given by
\begin{equation}   \label{eq:WK_metric}
	W_{d}(\mu,\nu) \= \sup\bigl\{\abs{\functional{\mu}{f}-\functional{\nu}{f}} :  f \in C^{0,1}(X,d), \, \abs{f}_{1,d}\leq 1\bigr\}.
\end{equation}
Note that for Borel probability measures in $\probmea{X}$, the convergence in $W_{d}$ is equivalent to the convergence in the weak$^*$ topology (see e.g.~\cite[Corollary~6.13]{villani2009optimal}).


\section{Preliminaries}
\label{sec:Preliminaries}

\subsection{Computable analysis}
\label{sub:Computable Analysis}

We recall fundamental notions and results from recursion theory and computable analysis.\footnote{Our notion of algorithm is consistent with \emph{Type-2 machines} defined in \cite[Definition~2.1.1]{weihrauch2000computable}.} We present, in order, definitions and results concerning the computability of real numbers, computable structures on metric spaces, computability of open sets, functions, compact sets, and probability measures.


    \medskip
    \noindent\textbf{Computability over the reals.} We begin by reviewing basic notations and concepts from classical recursion theory; for an introduction, see e.g.~\cite[Chapter~3]{Bridges1994computability}.

\begin{definition}[Effective enumeration and recursively enumerable set]\label{def:effectiveenumeration}
    Let $S\subseteq\N^*$ be a nonempty set. An \emph{effective enumeration} of $S$ is a sequence $\{x_i\}_{i\in\N}$ with $S = \{x_i:i\in\N\}$ such that there exists an algorithm that, for each $i\in\N$, upon input $i$, outputs $x_i$. 
    
    Moreover, a set $I\subseteq\N^*$ is said to be a \emph{recursively enumerable set}\footnote{We emphasize that recursively enumerable sets in this article are subsets of $\N^*$.} if $I=\emptyset$ or there exists an effective enumeration of $I$.
\end{definition}

For brevity, the symbol $I$ denotes a nonempty recursively enumerable set throughout this subsection.

Note that $\N^k$, for $k\in\N$, and $\N^*$ are all recursively enumerable sets by Definition~\ref{def:effectiveenumeration}. 

\begin{definition}[Partial recursive and recursive function] \label{def:computable function}
    Let $\{i_n\}_{n\in\N}$ be an effective enumeration of $I$. We say that $f\colon I\rightarrow\N_0^*$ is \emph{partial recursive} if there exists an algorithm that, for each $n\in\N$, on input $n$, outputs $f(i_n)$ if $f(i_n)\in\N^*$, and runs forever otherwise, namely, if $f(i_n)=0$. We say that $f\colon I\rightarrow\N_0^*$ is \emph{recursive} if $f$ is a partial recursive function with $f(I)\subseteq\N^*$.
\end{definition}

We now define the computability of real numbers.

\begin{definition}[Computable real number]\label{def:computability of real}
    A real number $x$ is called \emph{computable} if there exist three recursive functions $f\colon\N\mapping\N$, $g\colon\N\mapping\N$, and $h\colon\N\mapping\N$ such that $\Absbig{(-1)^{h(n)}f(n)/g(n)-x}<2^{-n}$ for all $i\in I$ and $n\in\N$. 
    
    Let $\{x_i\}_{i\in I}$ be a sequence of real numbers. We say that $\{x_i\}_{i\in I}$ is a \emph{sequence of uniformly computable real numbers} if there exist three recursive functions $f\colon\N\times I\rightarrow\N$, $g\colon\N\times I\rightarrow\N$, and $h\colon\N\times I\rightarrow\N$ such that $\Absbig{(-1)^{h(n,i)}  f(n,i)/g(n,i)-x_i}<2^{-n}$ for all $i\in I$ and $n\in\N$.
\end{definition}

Clearly, $x\in\R$ is computable if and only if $\{x_i\}_{i\in\N}$ defined by $x_i\=x$ for all $i\in\N$ is uniformly computable. For analogous concepts in the sequel, we will define the uniform sequence version and regard the individual case as the special case of constant sequences.

    \medskip
    \noindent\textbf{Computable metric spaces.} 
\begin{definition}[Computable metric space]\label{def:computablemetricspace}
    A \emph{computable metric space} is a triple $(X,\,\rho,\,\mathcal{S})$ satisfying that
    \begin{enumerate}
        \smallskip

        \item $(X,\rho)$ is a separable metric space,
        
        \smallskip

        \item $\mathcal{S}= \{s_n\}_{n\in\N}$ forms a countable dense subset $\{s_n : n \in\N\}$ of $X$, and
        
        \smallskip
        
        \item $\{\rho(s_m,s_n)\}_{(m,n)\in\N^2}$ is a sequence of uniformly computable real numbers.
    \end{enumerate}
    The points in $\mathcal{S}$ are called \emph{ideal}. Since $\N^3$ is recursively enumerable, the collection $\cB \= \{B(s_i,m/n) : i,\,m,\,n\in\N\}$ can be enumerated as $\{B_l\}_{l\in\N}$ satisfying the following: there exists an algorithm that, for each $l\in\N$, upon input $l$, outputs $i,\,m,\,n\in\N$ with $B_l=B(s_i,m/n)$. We call the elements in $\cB$ \emph{ideal balls} and such an enumeration of $\cB$ an \emph{effective enumeration of ideal balls} in $(X,\,\rho,\,\cS)$.
\end{definition}

We then define the computability of points in a computable metric space.

\begin{definition}[Computable point] \label{def:computability and uniform computability of point}
    Let $(X,\,\rho,\,\mathcal{S})$ be a computable metric space with $\cS=\{s_i\}_{i\in\N}$, and $\{x_i\}_{i\in I}$ be a sequence of points in $X$. Then $\{x_i\}_{i\in I}$ is called \emph{uniformly computable} (in $(X,\,\rho,\,\mathcal{S})$) if there exists a recursive function $f\colon\N\times I\rightarrow\N$ such that $\rho\bigl(s_{f(n,i)},x_i\bigr)<2^{-n}$ for all $n\in\N$ and $i\in I$. Moreover, a point $x$ in $X$ is \emph{computable} (in $(X,\,\rho,\,\cS)$) if $\{x_i\}_{i\in \N}$ defined by $x_i\=x$ for all $i\in\N$ is uniformly computable.
\end{definition}

We now specify the computable structure on $\R$. 
Let $\cS_{\Q}=\{q_n\}_{n\in\N}$ be the enumeration of $\Q$ induced by an effective enumeration of $\N^3$ via the mapping $(a,b,c) \mapsto (-1)^c a/b$. 
Note that $\{d_{\R}(q_m,q_n)\}_{(m,n)\in\N^2}$ is a sequence of uniformly computable real numbers, where $d_{\R}$ is the Euclidean metric. Then the triple $\bigl(\R,\,d_{\R},\,\cS_{\Q}\bigr)$ forms a computable metric space according to Definition~\ref{def:computablemetricspace}.
A similar construction provides a computable structure for $\R^+$. 
In this article, we fix these as the standard computability structures on $\R$ and $\R^+$. It is clear that under these structures, Definitions~\ref{def:computability of real}~and~\ref{def:computability and uniform computability of point} are equivalent for the computability of real numbers.
That is, a sequence of reals is uniformly computable in one sense if and only if it is in the other.

We also consider a weaker notion of computability over $\R$ that leverages its natural ordered structure.
    \begin{definition}[Semi-computable real number]\label{def:semicomputabilityoverR}
    Let $\{x_i\}_{i\in I}$ be a sequence of real numbers. We say that $\{x_i\}_{i\in I}$ is \emph{uniformly lower} (resp.\ \emph{upper}) \emph{semi-computable} if there exist three recursive functions $f\colon\N\times I\rightarrow\N$, $g\colon\N\times I\rightarrow\N$, and $h\colon\N\times I\rightarrow\N$ such that for each $i\in I$, $\bigl\{(-1)^{h(n,i)}f(n,i)/g(n,i)\bigr\}_{n\in\N}$ is nondecreasing (resp.\ nonincreasing) and converges to $x_i$ as $n\to+\infty$. Moreover, a real number $x$ is called \emph{lower} (resp.\ \emph{upper}) \emph{semi-computable} if the sequence $\{x_i\}_{i\in \N}$ defined by $x_i\coloneqq x$ for each $i\in \N$ is uniformly lower (resp.\ upper) semi-computable.
    \end{definition}

    \medskip
    \noindent\textbf{Lower semi-computable open sets.} We define an effective version of open sets and collect some relevant results.

    Let $(X,\,\rho,\,\cS)$ be a computable metric space. Let $\cB$ be the set of ideal balls, and $\{B_l\}_{l\in\N}$ be an effective enumeration of ideal balls in $(X,\,\rho,\,\cS)$. 
    We define the set $\cB_0 \= \cB\cup\{\emptyset\}$ of \emph{extended ideal balls} and an enumeration $\{D_l\}_{l\in\N}$ of $\cB_0$ such that $D_1=\emptyset$ and $D_l=B_{l-1}$ for each integer $l\geq 2$. 
    We call such an enumeration an \emph{effective enumeration of extended ideal balls} in $(X,\,\rho,\,\cS)$. 

\begin{definition}[Lower semi-computable open set] \label{def:lower semicomputability and uniform version of open set}
    Let $(X,\,\rho,\,\mathcal{S})$ be a computable metric space, and $\{D_l\}_{l\in\N}$ be an effective enumeration of extended ideal balls. 
    Then a sequence $\{U_i\}_{i\in I}$ of open sets in $X$ is said to be \emph{uniformly lower semi-computable open} (in $(X,\,\rho,\,\cS)$) if there exists a recursive function $f\colon\N\times I\rightarrow\N$ such that $U_i=\bigcup_{n\in\N}D_{f(n,i)}$ for each $i\in I$. Moreover, an open set $U\subseteq X$ is called \emph{lower semi-computable open} (in $(X,\,\rho,\,\cS)$) if the sequence $\{U_i\}_{i\in \N}$ defined by $U_i\=U$ for $i\in \N$ is uniformly lower semi-computable open.
\end{definition}

The above definition of a lower semi-computable open set differs slightly from the ones in \cite[Definition~3.4]{binder2011computability} and \cite[Definition~2.4]{binder2014computable}. In our definition, we use extended ideal balls, which include the empty set $\emptyset$.



The term \emph{recursively open set} in the literature (e.g.~\cite[Subsection~2.2~and~Definition~2.4]{galatolo2011dynamics} and \cite[Subsection~3.3]{hoyrup2009computability}) is equivalent to the notion of \emph{lower semi-computable open set} defined above. 
A detailed discussion of this equivalence is provided in \cite[Subsection~3.3]{Qiandu2025Lecture}.

Note that we can algorithmically decide whether $s\in B$ for each ideal point $s\in\cS$ and each extended ideal ball $B\in\cB_0$. The following result then follows immediately from Definition~\ref{def:lower semicomputability and uniform version of open set} (see e.g.~\cite[Proposition~3.9]{Qiandu2025Lecture}).

\begin{prop}\label{prop:semidecideinopenset}
    Let $(X,\,\rho,\,\cS)$ be a computable metric space with $\cS=\{s_n\}_{n\in\N}$. Assume that $\{U_i\}_{i\in I}$ is uniformly lower semi-computable open. Then there exists a recursively enumerable set $E\subseteq\N\times I$ such that $\{s_n:(n,i)\in E_i\}=\{s_n:n\in\N\}\cap U_i$, where $E_i\=\{(n,i)\in E:n\in\N\}$ for each $i\in I$.
\end{prop}

The following two results are two classical results in computable analysis which both follow immediately from Definitions~\ref{def:effectiveenumeration}~and~\ref{def:lower semicomputability and uniform version of open set} (see e.g.~\cite[Propositions~3.10~\&~3.11]{Qiandu2025Lecture}).  

\begin{prop}\label{prop:lower semi-computable open}
    Let $(X,\,\rho,\,\mathcal{S})$ be a computable metric space. Assume that $H$ and $L$ are two nonempty recursively enumerable sets with $L\subseteq I\times H$, and that $\{U_{i,h}\}_{(i,h)\in L}$ is uniformly lower semi-computable open. Then $\{\bigcup\{U_{i,h}:(i,h)\in L_h\}\}_{h\in H}$ is uniformly lower semi-computable open, where $L_h\=\{(i,h)\in L:i\in I\}$ for each $h\in H$. In particular, if $\{U_i\}_{i\in I}$ is uniformly lower semi-computable open, then $\bigcup_{i\in I}U_i$ is lower semi-computable open.
\end{prop}

\begin{prop}		\label{prop:lowersemicomputableraiusball}
	Let $(X,\,\rho,\,\cS)$ be a computable metric space. Assume that $\{r_i\}_{i\in I}$ is a sequence of uniformly lower semi-computable real numbers and $\{x_i\}_{i\in I}$ is uniformly computable in $(X,\,\rho,\,\cS)$. Then $\{B(x_i,r_i)\}_{i\in I}$ is uniformly lower semi-computable open.
\end{prop}

    \medskip
    \noindent\textbf{Computability of functions.} We begin with the definition of oracles for points.

\begin{definition}[Oracle]  \label{def:oracle}
   	Let $(X,\,\rho,\,\mathcal{S})$ be a computable metric space with $\cS=\{s_i\}_{i\in\N}$, and $x\in X$. 
    We say that a function $\tau \colon \N \rightarrow \N$ is an \emph{oracle} for $x$ if $\rho( s_{\tau(n)},x)<2^{-n}$ for each $n\in\N.$
\end{definition}

With the above definition, computable functions can be defined as follows.

\begin{definition}[Computable function]  \label{def:Algorithm about computable functions}
    Let $(X,\,\rho,\,\mathcal{S})$ and $(X',\,\rho',\,\mathcal{S}')$ be computable metric spaces with $\cS = \{s_n\}_{n\in\N}$ and $\cS'=\{s'_n\}_{n\in\N}$. Assume that $\{i_n\}_{n\in\N}$ is an effective enumeration of $I$, and $C_i\subseteq X$ for each $i\in I$.
    Then a sequence $\{f_i\}_{i\in I}$ of functions $f_i\colon X \rightarrow X'$ is called a \emph{sequence of uniformly computable functions with respect to $\{C_i\}_{i\in I}$} if there exists an algorithm that, for all $l,\,n\in\N$, $x\in C_{i_n}$, and oracle $\tau$ for $x$, on input $l,\,n$, and $\tau$, outputs $m\in\N$ with $\rho'(s'_m,f_{i_n}(x))<2^{-l}$. We often omit the phrase ``with respect to $\{C_i\}_{i\in I}$'' when $C_i=X$ for all $i\in I$. Moreover, a function $f\colon X\rightarrow X^{\prime}$ is said to be a \emph{computable function on $C$} if $\{f_i\}_{i\in \N}$, defined by $f_i\=f$ for all $i\in \N$, is a sequence of uniformly computable functions with respect to $\{C_i\}_{i\in\N}$ defined by $C_i\=C$ for all $i\in\N$. We often omit the phrase ``with respect to $C$'' when $C=X$.
\end{definition}

Computable functions serve as an effective version of continuous functions. The following result provides examples of computable functions (see e.g.~\cite[Examples~4.3.3~and~4.3.13.5]{weihrauch2000computable}).

\begin{example}\label{exp:computablefunctionexample}
    The exponential function $\exp\:\R\rightarrow\R$ and the logarithmic function $\log\colon\R^+\rightarrow\R$ are computable functions.
\end{example}

We recall the following classical characterization of computable functions (cf.~\cite[Proposition~5.2.14]{rojas2021computable} and \cite[Proposition~3.6]{binder2011computability}; see also \cite[Theorem~3.17]{Qiandu2025Lecture}).

\begin{prop}\label{prop:computable function}
    Let $(X,\,\rho,\,\mathcal{S})$ and $(X^{\prime},\,\rho^{\prime},\,\cS^{\prime})$ be computable metric spaces. Suppose $\{B_n^{\prime}\}_{n\in\N}$ is an effective enumeration of ideal balls in $(X^{\prime},\,\rho^{\prime},\,\cS^{\prime})$. Given $f_i\colon X\rightarrow X^{\prime}$ and $C_i\subseteq X$ for each $i\in I$, the following statements are equivalent:
    \begin{enumerate}
        \smallskip
        \item The sequence $\{f_i\}_{i\in I}$ is a sequence of uniformly computable functions with respect to $\{C_i\}_{i\in I}$.
        \smallskip
        \item There exists a sequence $\{U_{n,i}\}_{(n,i)\in\N\times I}$ of uniformly lower semi-computable open sets in $(X,\,\rho,\,\mathcal{S})$ such that $f_i^{-1}(B_n^{\prime})\cap C_i=U_{n,i}\cap C_i$ for all $i\in I$ and $n\in\N$.
        \smallskip
        \item For each nonempty recursively enumerable set $M$ and each sequence $\{V_m^{\prime}\}_{m\in M}$ of uniformly lower semi-computable open sets, there exists a sequence $\{W_{m,i}\}_{(m,i)\in M\times I}$ of uniformly lower semi-computable open sets in $(X,\,\rho,\,\mathcal{S})$ such that $f_i^{-1}(V_m^{\prime})\cap C_i=W_{m,i}\cap C_i$ for all $m\in M$ and $i\in I$.
    \end{enumerate}
\end{prop}

We now define a notion of weaker computability property for functions.

\begin{definition}[Semi-computable function] \label{def:domain of lower or upper computability}
    Let $(X,\,\rho,\,\mathcal{S})$ be a computable metric space, $\{i_n\}_{n\in\N}$ be an effective enumeration of $I$, and $C_i\subseteq X$ for each $i\in I$. A sequence $\{f_i\}_{i\in I}$ of functions $f_i\:X\rightarrow\R$ is a \emph{sequence of uniformly upper} (resp.\ \emph{lower}) \emph{semi-computable functions with respect to $\{C_i\}_{i\in I}$} if there exists an algorithm that, for all $l,\,n\in\N$, $x\in C_{i_n}$, and oracle $\tau$ for $x$, on input $l,\,n$, and $\tau$, outputs $q_{l,n,\tau}\in\Q$ such that for each $n\in\N$, each $x\in C_{i_n}$, and each oracle $\tau$ for $x$, $\{q_{l,n,\tau}\}_{l\in\N}$ is nonincreasing (resp.\ nondecreasing) and converges to $f_{i_n}(x)$ as $l\to+\infty$. We often omit the phrase ``with respect to $\{C_i\}_{i\in I}$''  when $C_i=X$ for each $i\in I$. Moreover, a function $f\colon X\rightarrow\R$ is said to be an \emph{upper} (resp.\ a \emph{lower}) \emph{semi-computable function on $C$} if $\{f_i\}_{i\in \N}$ defined by $f_i\=f$ for each $i\in \N$, is a sequence of uniformly upper (resp.\ lower) semi-computable functions with respect to $\{C_i\}_{i\in\N}$ defined by $C_i\=C$ for all $i\in\N$. We often omit the phrase ``with respect to $C$'' when $C=X$. 
\end{definition}

The following proposition is an immediate consequence of Proposition~\ref{prop:computable function} (see e.g.~\cite[Theorem~3.19]{Qiandu2025Lecture}).

\begin{prop}\label{prop:semicomputability}
    Let $(X,\,\rho,\,\mathcal{S})$ be a computable metric space, and $\cS_{\Q}=\{q_n\}_{n\in\N}$. Given $f_i\colon X\rightarrow\R$ and $C_i\subseteq X$ for all $i\in I$, the following statements are equivalent:
    \begin{enumerate}
        \smallskip
        \item The sequence $\{f_i\}_{i\in I}$ is a sequence of uniformly upper (resp.\ lower) semi-computable functions with respect to $\{C_i\}_{i\in I}$.
        \smallskip
        \item There exists a sequence $\{U_{n,i}\}_{(n,i)\in\N\times I}$ of uniformly lower semi-computable open sets in $(X,\,\rho,\,\mathcal{S})$ such that $f_i^{-1}(Q_n)\cap C_i=U_{n,i}\cap C_i$ with $Q_n \= (-\infty,q_n)$ (resp.\ $Q_n \= (q_n,+\infty)$) for all $i\in I$ and $n\in\N$.
        \smallskip
        \item For each nonempty recursively enumerable set $L$ and each sequence $\{r_l\}_{l\in L}$ of uniformly computable real numbers, there exists a sequence $\{W_{l,i}\}_{(l,i)\in L\times I}$ of uniformly lower semi-computable open sets in $(X,\,\rho,\,\mathcal{S})$ such that $f_i^{-1}(R_l)\cap C_i=W_{l,i}\cap C_i$ with $R_l\coloneqq(-\infty,r_l)$ (resp.\ $R_l\coloneqq(r_l,+\infty)$) for all $l\in L$ and $i\in I$.
    \end{enumerate}
\end{prop}

The following result indicates that the characteristic function of a lower semi-computable open set is a lower semi-computable function (see e.g.~\cite[Proposition~3.33]{Qiandu2025Lecture}).

\begin{prop}		\label{prop:characterfunctionislowersemicomputable}
	Let $(X,\,\rho,\,\mathcal{S})$ be a computable metric space. Assume that $\{U_i\}_{i\in I}$ is uniformly lower semi-computable open. Then there exists a sequence $\{h_{n,i}\}_{(n,i)\in\N\times I}$ of uniformly computable functions $h_{n,i}\colon X\rightarrow\R$ such that for each $i\in I$, the following properties are satisfied:
	\begin{enumerate}
		\smallskip
		
		\item For each $x\in X$, $\{h_{n,i}(x)\}_{n\in\N}$ is nondecreasing and converges to $\mathbbm{1}_{U_i}(x)$ as $n\to+\infty$.
		
		\smallskip
		
		\item For each $n\in\N$, $h_{n,i}(x)\geq 0$ for each $x\in X$ and $h_{n,i}(x)=0$ for each $x\notin U_{i}$.
	\end{enumerate}  
\end{prop}

    \medskip
    \noindent\textbf{Recursively compact sets and recursively precompact metric spaces.} Here we recall the definitions of recursive compactness and recursive precompactness. For a more detailed discussion, see \cite[Section~2]{galatolo2011dynamics}.

\begin{definition}[Recursively compact set]\label{def:recursively compact}
    Let $(X,\,\rho,\,\mathcal{S})$ be a computable metric space with $\cS=\{s_i\}_{i\in\N}$, and $\{i_l\}_{l\in\N}$ be an effective enumeration of $I$. A sequence $\{K_i\}_{i\in I}$ of compact sets in $X$ is called \emph{uniformly recursively compact} (in $(X,\,\rho,\,\cS)$) if there exists an algorithm that, for each $n\in\N$, each sequence $\{m_n\}_{n=1}^p$ of integers, and each sequence $\{q_n\}_{n=1}^p$ of positive rational numbers, upon input, halts if and only if $K_{i_l}\subseteq\bigcup_{n=1}^pB(s_{m_n},q_n)$. 
    Moreover, a set $K\subseteq X$ is called \emph{recursively compact} (in $(X,\,\rho,\,\cS)$) if the sequence $\{K_i\}_{i\in \N}$ defined by $K_i\=K$ for each $i\in\N$, is uniformly recursively compact.
\end{definition}

Note that for each compact set $K$ and each function $f\colon\N\rightarrow\N$, $K\subseteq\bigcup_{n\in\N}D_{f(n)}$ if and only if $K\subseteq\bigcup_{n=1}^kD_{f(n)}$ for some $k\in\N$. This implies the following result.

\begin{prop}\label{prop:equivalentdefinitionofrecursivelycompact}
    Let $(X,\,\rho,\,\cS)$ be a computable metric space. Suppose $\{h_m\}_{m\in\N}$ (resp.\ $\{l_n\}_{n\in\N}$) is an effective enumeration of a nonempty recursively enumerable set $H$ (resp.\ $L$). Assume that $\{K_h\}_{h\in H}$ is uniformly recursively compact and $\{U_l\}_{l\in L}$ is uniformly lower semi-computable open. Then there exists an algorithm that, for all $m,\,n\in\N$, upon input, halts if and only if $K_{h_m}\subseteq U_{l_n}$.
\end{prop}

We collect some fundamental properties of recursively compact sets (cf.\ \cite[Propositions~1~\&~3]{galatolo2011dynamics}; see also \cite[Proposition~3.23]{Qiandu2025Lecture}).

\begin{prop}\label{thecomplementofrecursivelycompactset}
    Let $(X,\,\rho,\,\mathcal{S})$ be a computable metric space. Assume that $X$ is recursively compact, and $\{K_i\}_{i\in I}$ is uniformly recursively compact. Then the following statements are true:
    \begin{enumerate}
        \smallskip
        \item Let $x_i\in X$ for each $i\in I$. Then $\{x_i\}_{i\in I}$ is uniformly computable if and only if the sequence $\{\{x_i\}\}_{i\in I}$ of singletons is uniformly recursively compact.
        \smallskip
        \item $\{X\smallsetminus K_i\}_{i\in I}$ is uniformly lower semi-computable open.
        \smallskip
        \item If $\{U_i\}_{i\in I}$ is uniformly lower semi-computable open, then $\{K_i\smallsetminus U_i\}_{i\in I}$ is uniformly recursively compact.
        \smallskip
        \item If $\{f_i\}_{i\in I}$ is a sequence of uniformly lower (resp.\ upper) semi-computable functions $f_i\colon X\rightarrow\R$ with respect to $\{K_i\}_{i\in I}$, then $\{\inf_{x\in K_i}f_i(x)\}_{i\in I}$ (resp.\ $\{\sup_{x\in K_i}f_i(x) \}_{i\in I}$) is uniformly lower (resp.\ upper) semi-computable.
        \smallskip
        \item If $\{T_i\}_{i\in I}$ is a sequence of uniformly computable functions $T_i\colon X\rightarrow X$ with respect to $\{K_i\}_{i\in I}$, then $\{T_i(K_i)\}_{i\in I}$ is uniformly recursively compact.
    \end{enumerate}
\end{prop}

Next, we investigate whether the property of uniform computability for recursively compact sets is preserved under the union and intersection.

\begin{prop}    \label{prop:operatorsonrecursivelycompactsets}
    Let $(X,\,\rho,\,\mathcal{S})$ be a computable metric space. Suppose $X$ is recursively compact, $H$ and $L$ are two nonempty recursively enumerable sets with $L\subseteq I\times H$, and $\{K_{i,h}\}_{(i,h)\in L}$ is uniformly recursively compact. Denote $L_h\=\{(i,h)\in L:i\in I\}$ for each $h\in H$. Then the following statements are true:
    \begin{enumerate}
        \smallskip
        \item $\{\bigcap\{K_{i,h}:(i,h)\in L_h\}\}_{h\in H}$ is uniformly recursively compact. 
        \smallskip
        \item If the function $F\colon H\mapping\N$ defined by $F(h)\=\operatorname{card} L_h$ for $h\in H$ is recursive, then $\{\bigcup\{K_{i,h}:(i,h)\in L_h\}\}_{h\in H}$ is uniformly recursively compact.
    \end{enumerate}
\end{prop}

Proposition~\ref{prop:operatorsonrecursivelycompactsets}~(i) follows immediately from Proposition~\ref{prop:lower semi-computable open} and Proposition~\ref{thecomplementofrecursivelycompactset}~(ii) and~(iii). Moreover, Proposition~\ref{prop:operatorsonrecursivelycompactsets}~(ii) follows from Definition~\ref{def:recursively compact}. As a corollary of Proposition~\ref{prop:operatorsonrecursivelycompactsets}~(ii), we obtain the following result, which is important in the proof of Theorem~\ref{Application B}.

\begin{cor}\label{cor:preimageisloweropen}
    Let $(X,\,\rho,\,\cS)$ be a computable metric space. Assume that $X$ is recursively compact, $T\colon X\mapping X$ is a computable function, and $\{U_i\}_{i\in I}$ is uniformly lower semi-computable open. Then $\{V_{n,i}\}_{(n,i)\in\N\times I}$ is uniformly lower semi-computable open, where $V_{n,i}$ is defined inductively by setting $V_{1,i} \= U_i$ and $V_{n+1,i}\=T^{-1}(V_{n,i})\cap U_i$ for each $n\in\N$ and each $i\in I$.
\end{cor}

\begin{proof}
    Since $T$ is a computable function, by Definition~\ref{def:computable function}, we obtain that $\{T^n\}_{n\in\N_0}$ is a sequence of uniformly computable functions. Then by Proposition~\ref{prop:computable function}, $\{T^{-n}(U_i)\}_{(n,i)\in\N_0\times I}$ is a sequence of uniformly lower semi-computable open sets. Hence, since $X$ is recursively compact, by Proposition~\ref{thecomplementofrecursivelycompactset}~(iii), $\{X\smallsetminus T^{-n}(U_i)\}_{(n,i)\in\N_0\times I}$ is a sequence of uniformly recursively compact sets. 
    
    Define $L\subseteq\N_0\times\N\times I$ by $L\coloneqq\{(m,n,i)\in\N_0\times\N\times I:m< n\}$. Then $L$ is a recursively enumerable set and $\{X\smallsetminus T^{-m}(U_i)\}_{(m,n,i)\in L}$ is uniformly recursively compact. Note that $\mathrm{card}\{m\in\N_0:(m,n,i)\in L\}=n$ for all $n\in\N$ and $i\in I$. Then the function $F\colon\N\times I\mapping\N_0$ given by $F(n,i)\=\mathrm{card}\{m\in\N_0:(m,n,i)\in L\}$ is a recursive function. Thus, by Proposition~\ref{prop:operatorsonrecursivelycompactsets}~(ii), we obtain that $\{\bigcup\{X\smallsetminus T^{-m}(U_i):m\in\N_0,\,(m,n,i)\in L\}\}_{(n,i)\in\N\times I}$ is a sequence of uniformly recursively compact sets. Hence, since $X$ is a recursively compact set, by Proposition~\ref{thecomplementofrecursivelycompactset}~(ii), $\{X\smallsetminus\bigcup\{X\smallsetminus T^{-m}(U_i):m\in\N_0,\,(m,n,i)\in L\}\}_{(n,i)\in\N\times I}=\{\bigcap\{T^{-m}(U_i):m\in\N_0,\,(m,n,i)\in L\}\}_{(n,i)\in\N\times I}$ is uniformly lower semi-computable open. Since $V_{1,i}=U_i$ and $V_{n+1,i}=T^{-1}(V_{n,i})\cap U_i$ for all $n\in\N$ and $i\in I$, it follows by induction that $V_{n+1,i}=\bigcap_{k=0}^{n}T^{-k}(U_i)$ for each $n\in\N_0$. Therefore, $\{V_{n,i}\}_{(n,i)\in\N\times I}$ is uniformly lower semi-computable open.
\end{proof}

Moreover, given the recursive compactness of $X$, the computability of functions is preserved under a finite number of operations among additions and multiplications. We summarize this property in the following result (cf.~\cite[Corollary~4.3.4]{weihrauch2000computable}; see also \cite[Proposition~3.26]{Qiandu2025Lecture}).

\begin{prop}\label{prop:computabilityoffunctionspreservedunderoperators}
    Let $(X,\,\rho,\,\cS)$ be a computable metric space in which $X$ is recursively compact, and $H$ be a nonempty recursively enumerable set. Assume that $\{f_i\}_{i\in I}$ (resp.\ $\{g_h\}_{h\in H}$) is a sequence of uniformly computable functions $f_i\colon X\mapping\R$ (resp.\ $g_h\colon X\mapping\R$). Then $\{f_i+g_h\}_{(i,h)\in I\times H}$, $\{f_i\cdot g_h\}_{(i,h)\in I\times H}$ are two sequences of uniformly computable functions.
\end{prop}

Next, we recall the definition of recursively precompact metric space.

\begin{definition}[Recursively precompact metric space]		\label{def:recursivelyprecompact}
	Let $(X,\,\rho,\,\cS)$ be a computable metric space with $\cS=\{s_i\}_{i\in\N}$. Then $(X,\,\rho,\,\cS)$ is called \emph{recursively precompact} if there exists an algorithm that, for each $n\in\N$, on input $n$, outputs a finite subset $\{r_i:1\leq i\leq m\}$ of $\N$ such that $X=\bigcup_{i=1}^mB (s_{r_i},2^{-n})$.
\end{definition}

Finally, we record the following useful characterization of complete recursively precompact metric spaces (see e.g.~\cite[Proposition~4]{galatolo2011dynamics}).

\begin{prop}\label{prop:relationbetweencompactandprecompact}
	Let $(X,\,\rho,\,\cS)$ be a computable metric space. Then $X$ is recursively compact if and only if $(X,\rho)$ is complete and $(X,\,\rho,\,\cS)$ is recursively precompact.
\end{prop}

\medskip
\noindent\textbf{Computability of probability measures.} 
Building upon the theory of computable functions and recursively compact sets, we now discuss the computability of probability measures. 
We begin by reviewing the computable structure on the measure space $\PPP(X)$ introduced in \cite[Section~4]{hoyrup2009computability} (cf.~\cite[Proposition~4.1.3]{hoyrup2009computability}; see also \cite[Proposition~3.29]{Qiandu2025Lecture}).

\begin{prop}    \label{prop:recursively compactness of measure space}
    Let $(X,\,\rho,\,\mathcal{S})$ be a computable metric space with $\cS=\{s_n\}_{n\in\N}$. 
    Assume that $X$ is recursively compact in $(X,\,\rho,\,\cS)$. Then the following statements are true:
    \begin{enumerate}
        \smallskip
        \item There exists an enumeration $\mathcal{Q}_{\mathcal{S}}=\{\nu_k\}_{k\in\N}$ of the set of Borel probability measures that are supported on finitely many points in $\{s_n:n\in\N\}$ and assign rational values to them such that there exists an algorithm that, for each $k\in\N$, upon input $k$, outputs a sequence $\{n_l\}_{l=1}^p$ of integers and a sequence $\{q_l\}_{l=1}^p$ of positive rational numbers satisfying that $\sum_{l=1}^pq_l=1$ and $\nu_k=\sum_{l=1}^pq_l\delta_{s_{n_l}}$.
        \smallskip
        \item $(\PPP(X),\,W_{\rho},\,\mathcal{Q}_{\mathcal{S}})$ is also a computable metric space in which $\PPP(X)$ is recursively compact, where $W_{\rho}$ is the Wasserstein--Kantorovich metric on $\PPP(X)$ (see (\ref{eq:WK_metric})).
    \end{enumerate}
\end{prop}

Let $(X,\,\rho,\,\mathcal{S})$ be a computable metric space and assume that $X$ is recursively compact.
We endow the measure space $\cP(X)$ with the computable structure $(\cP(X),\,W_{\rho},\,\cQ_{\cS})$ given by Proposition~\ref{prop:recursively compactness of measure space}.

The computability of measures is then defined via Definition~\ref{def:computability and uniform computability of point}.
Specifically, a sequence $\{\mu_i\}_{i\in I}$ of measures in $\cP(X)$ is \emph{a sequence of uniformly computable measures} if it is uniformly computable in $(\PPP(X),\,W_{\rho},\,\mathcal{Q}_{\mathcal{S}})$, and a single measure $\mu\in\cP(X)$ is \emph{a computable measure} if the corresponding constant sequence consisting of $\mu$ is uniformly computable.

We now recall a key result on the computability of the integration function (cf.\ \cite[Corollary~4.3.2]{hoyrup2009computability}; see also \cite[Proposition~3.30]{Qiandu2025Lecture}).

\begin{prop}\label{prop:computability of integral operator}
    Let $(X,\,\rho,\,\mathcal{S})$ be a computable metric space. Assume that $X$ is recursively compact in $(X,\,\rho,\,\cS)$, and that $\{f_i\}_{i\in I}$ is a sequence of uniformly computable functions $f_i\colon X\rightarrow\R$. Then the sequence $\{\cI_i\}_{i\in I}$ of functions $\cI_i\colon\cP(X)\rightarrow\R$ defined by $\cI_i(\mu) \= \functional{\mu}{f_i}$ for $\mu \in \cP(X)$ is a sequence of uniformly computable functions.
\end{prop}

As immediate corollaries of Proposition~\ref{prop:computability of integral operator}, we have the following results.

\begin{cor}\label{cor:semicomputability of integral operator}
    Let $(X,\,\rho,\,\mathcal{S})$ be a computable metric space. Assume that $X$ is recursively compact in $(X,\,\rho,\,\cS)$, and that $\{f_i\}_{i\in I}$ is a sequence of uniformly upper (resp.\ lower) semi-computable functions $f_i\colon X\rightarrow\R$. Then the sequence $\{\cI_i\}_{i\in I}$ of functions $\cI_i\colon\cP(X)\rightarrow\R$ given by $\cI_i(\mu) \= \functional{\mu}{f_i}$ is a sequence of uniformly upper (resp.\ lower) semi-computable functions.
\end{cor}

\begin{cor}\label{cor:seminormrecursivelycompact}
    Let $(X,\,\rho,\,\cS)$ be a computable metric space in which $X$ is recursively compact. Assume that $H$ and $L$ are two nonempty recursively enumerable sets with $L\subseteq I\times H$, $\{U_{i,h}\}_{(i,h)\in L}$ is uniformly lower semi-computable open in $(X,\,\rho,\,\cS)$, and $\{r_{i,h}\}_{(i,h)\in L}$ is a sequence of uniformly computable real numbers. Define, for each $i\in I$, $L_i\=\{(i,h)\in L:h\in H\}$ and $\cK_i\=\{\mu\in\PPP(X):\mu(U_{i,h})\leq r_{i,h}\text{ for each }(i,h)\in L_i\}$. Then $\{\cK_i\}_{i\in I}$ is uniformly recursively compact in $(\cP(X),\,W_{\rho},\,\cQ_{\cS})$.
\end{cor}

Recall the definition of $\cM(X,T;Y)$ from (\ref{eq:defofcMonsubset}). The following result indicates the recursive compactness of the set $\cM(X,T;Y)$ (cf.~\cite[Lemma~4.12]{binder2025computability}; see also \cite[Theorem~3.36]{Qiandu2025Lecture}).

\begin{prop}\label{prop:recursivecompactofsetinvariant}
    Let $(X,\,\rho,\,\cS)$ be a computable metric space in which $X$ is recursively compact, and $\{U_i\}_{i\in I}$ is a sequence of uniformly lower semi-computable open sets. Assume that $\{T_i\}_{i\in I}$ is a sequence of uniformly computable functions $T_i\colon X\rightarrow X$ with respect to $\{U_i\}_{i\in I}$. Then $\{\cM(X,T_i;U_i)\}_{i\in I}$ is uniformly recursively compact in $(\cP(X),\,W_{\rho},\,\cQ_{\cS})$. In particular, if $\{T_i\}_{i\in I}$ is a sequence of uniformly computable functions, then $\{\cM(X,T_i)\}_{i\in I}$ is uniformly recursively compact.
\end{prop}

The following result is useful in the proof of the main result of this article (see e.g.~\cite[Proposition~3.35]{Qiandu2025Lecture}).

\begin{prop}\label{prop:test function}
	Let $(X,\,\rho,\,\mathcal{S})$ be a computable metric space, and $X$ be a recursively compact set in $(X,\,\rho,\,\cS)$. Then there exists a sequence $\{\tau_n\}_{n\in\N}$ of uniformly computable functions $\tau_n\colon X\mapping\R$ such that $\{\tau_n:n\in\N\}$ is dense in $C(X)$. Moreover, for all $\mu,\,\nu \in \cM(X)$, $\mu(A)\geq\nu(A)$ for each $A\in\cB(X)$ if and only if $\Functional{\mu}{\tau_n^+}\geq\Functional{\nu}{\tau_n^+}$ for each $n\in\N$.
\end{prop}

\subsection{Thermodynamic formalism}
\label{sub:thermodynamic formalism}

We review basic concepts from ergodic theory. For more detailed discussions, we refer the reader to \cite[Section~4]{walters1982introduction}.

Let $(X,\cB,\mu)$ be a probability space. A \emph{partition} $\xi=\{A_h:h\in H\}$ of $(X,\cB,\mu)$ is a disjoint collection of elements of $\cB$ whose union is $X$, where $H$ is a countable index set. For each pair of partitions $\xi = \{A_h:h \in H\}$ and $\eta = \{B_l:l \in L\}$ of $X$, their \emph{join} is the partition $\xi\vee\eta\=\{A_h\cap B_l \describe h\in H,\,l\in L\}$.

Assume that $T\colon X\mapping X$ is a measure-preserving transformation of $(X,\cB,\mu)$. Consider a partition $\xi = \{A_h \describe h\in H\}$ of $X$. For each $n\in\N$, $T^{-n}(\xi)$ denotes the partition $\bigl\{T^{-1}(A_h)\describe h\in H\bigr\}$, and $\xi^n_T$ denotes the join $\xi\vee T^{-1}(\xi)\vee\cdots\vee T^{-(n-1)}(\xi)$. The \emph{entropy} of $\xi$ is $H_{\mu}(\xi)\=-\sum_{h\in H} \mu(A_h)\log(\mu(A_h))\in[0, +\infty]$, where $0\log 0$ is defined to be zero. One can show that if $H_{\mu}(\xi)<+\infty$, then $\lim\limits_{n\to +\infty} H_{\mu}(\xi^n_T)/n$ exists (see e.g.\ \cite[Chapter~4]{walters1982introduction}). We denote this limit by $h_{\mu}(T,\xi)$ and call it the \emph{measure-theoretic entropy of $T$ relative to $\xi$}.
The \emph{measure-theoretic entropy} of $T$ for $\mu$ is defined as
\begin{equation}   \label{eq:def:measure-theoretic entropy}
h_{\mu}(T)\=  \sup\{h_{\mu}(T,\xi) \describe \xi \text{ is a partition of } X  \text{ with } H_{\mu}(\xi) < +\infty\}.   
\end{equation}

We now introduce thermodynamic formalism, a particular branch of ergodic theory.
The main objects of study are the topological pressure and equilibrium states (see e.g.~\cite{przytycki2010conformal, walters1982introduction}; for the general Borel-measurable setting used in Approach~II, see e.g.~\cite[Definition~1.1]{iommi2010natural}, \cite[Section~2.3]{demers2017equilibrium}, and \cite[Chapter~1.4]{dobbs2023free}).

Let $(X,\rho)$ be a compact metric space, $T\colon X\mapping X$ be a Borel-measurable transformation such that $\cM(X,T)\neq\emptyset$, and $\phi\colon X\mapping[-\infty,+\infty]$ be a Borel function. Then the \emph{topological pressure} of the potential $\phi$ with respect to the transformation $T$ is given by
\begin{equation}  \label{eq:Variational Principle for pressure}
    P(T, \phi)\coloneqq \sup \bigl\{h_\mu (T) + \functional{\mu}{\phi} \describe \mu \in \cM(X, T)\text{ and }\functional{\mu}{\phi}>-\infty\bigr\}.
\end{equation}
A measure $\mu\in\cM(X,T)$ that attains the supremum in \eqref{eq:Variational Principle for pressure} is called an \emph{equilibrium state} for the transformation $T$ and the potential $\phi$. Denote the set of all such measures by $\mathcal{E}(T,\phi)$. In particular, when the potential $\phi$ is the constant function $0$, we denote $h_{\mathrm{top}}(T)\coloneqq P(T,0)$ and say that a measure $\mu\in\cM(X,T)$ is a \emph{measure of maximal entropy} of $T$ if $\mu\in\mathcal{E}(T,0)$.


\section{Approach I: Global approach via convex analysis}
\label{sec:Approach I} 

In this section, we establish the computability of equilibrium states for certain dynamical systems with upper semicontinuous measure-theoretic entropy functions. We begin by recalling several definitions and results from functional analysis: a characterization of the set $\tangent{C(X)}{\potential}{P_T}$ in \eqref{eq:expressionoftangentspace} and its relation to the set $\mathcal{E}(T,\potential)$ of equilibrium states in \eqref{eq:relationbetweentangentspaceandsetofequilibriumstate}. 
We then apply these results to verify the recursive compactness of $\tangent{C(X)}{\potential}{P_T}$, thereby completing the proof of Theorem~\ref{theorem A}.

Let $(X,\rho)$ be a compact metric space, and $T\colon X\rightarrow X$ a continuous map with finite topological entropy.
The \emph{measure-theoretic entropy function} of $T$ is the function $\mu\mapsto h_{\mu}(T)$ (defined in \eqref{eq:def:measure-theoretic entropy}) on the space $\cM(X,T)$ of $T$-invariant Borel probability measures, where $\cM(X,T)$ is equipped with the weak$^*$ topology.
The \emph{topological pressure function} of $T$, denoted by $P_T$, is the function $\potential\mapsto P(T,\potential)$ (defined in \eqref{eq:Variational Principle for pressure}) on $C(X)$.
For each $\mu\in\cM(X,T)$, we denote by $\overline{h}_{\mu}(T)$ the \emph{upper semicontinuous regularization} of the measure-theoretic entropy function.
More precisely, $\overline{h}_{\mu}(T)$ is the supremum of all limit suprema $\limsup_{n\to\infty} h_{\mu_n}(T)$, where $\{\mu_n\}_{n\in\N}$ ranges over all sequences in $\cM(X,T)$ that converge to $\mu$ in the weak$^*$ topology. 

\begin{definition}\label{def:tangent space}
	Let $V$ be a real topological vector space, and $G\colon V\rightarrow\R$ be a convex continuous function.
	A continuous linear functional $F\colon V\rightarrow\R$ is \emph{tangent to $G$ at $x\in V$} if
	\begin{equation*}
		F(y) \leqslant G(x + y) - G(x)
	\end{equation*}
	for each $y\in V$.
	We denote the set of all such functionals by $V_{x,G}^*$.
\end{definition}

The following lemma is well-known (see e.g.\ \cite[Theorem~9.7~(iv) and~(v)]{walters1982introduction}).

\begin{lemma}	\label{lem:continuityandconvexityofpressurefunction}
    Let $(X,\rho)$ be a compact metric space and $T\colon X\rightarrow X$ be a continuous map with finite topological entropy.
    Then the topological pressure function $P_T\colon C(X)\rightarrow\R$ is convex and continuous.
\end{lemma} 

Lemma~\ref{lem:continuityandconvexityofpressurefunction} ensures that the set $C(X)_{\potential,P_T}^*$ is well defined for any continuous map $T$ on a compact metric space $X$ and any continuous function $\potential$.
We now record a characterization of this space from \cite[Theorem~3~(i)]{walters1992differentiability}.

\begin{lemma}	\label{lem:representationoninvariantmeasure}
    Let $(X,\rho)$ be a compact metric space, $T\colon X\rightarrow X$ be a continuous map with finite topological entropy, and $\potential\colon X\rightarrow\R$ be a continuous function.
    A functional $F\in C(X)^*$ belongs to $C(X)_{\potential,P_T}^*$ if and only if there exists a measure $\mu_F \in \cM(X,T)$ such that $F(f)=\functional{\mu_F}{f}$ for each $f\in C(X)$, and $\mu_F$ is the weak$^*$ limit of a sequence of measures $\{\mu_n\}_{n\in\N}$ in $\cM(X,T)$ satisfying $h_{\mu_n}(T)+\functional{\mu_n}{\potential}\to P(T,\potential)$ as $n\to+\infty$.
\end{lemma}

\begin{remark}  \label{rm:identification}
    Recall that $C(X)^*$ can be naturally identified with the space $\cM(X)$ of finite signed Borel measures on $X$.
    Then Lemma~\ref{lem:representationoninvariantmeasure} allows us to identify $C(X)_{\potential,P_T}^*$ as a subset of $\cM(X,T)$.
    We will adopt this identification in the remaining part of this article.
\end{remark}

The following proposition characterizes $C(X)_{\potential,P_T}^*$ and its relation to the set of equilibrium states $\mathcal{E}(T,\potential)$.

\begin{prop} \label{prop:relationbetweentangentspaceandsetofequilibriumstate}
	Let $(X,\rho)$ be a compact metric space, $T\colon X\rightarrow X$ be a continuous map with finite topological entropy, and $\potential\colon X\rightarrow\R$ be a continuous function.
	Then
	\begin{align}
		C(X)_{\potential,P_T}^*&=\bigl\{\mu\in\cM(X,T) : \overline{h}_{\mu}(T)+\functional{\mu}{\potential}=P(T,\potential)\bigr\}\text{ and }\label{eq:expressionoftangentspace}\\
		\mathcal{E}(T,\potential)&=C(X)_{\potential,P_T}^*\cap\bigl\{\mu\in\cM(X,T) : \overline{h}_{\mu}(T)\leqslant h_{\mu}(T)\bigr\}.\label{eq:relationbetweentangentspaceandsetofequilibriumstate}
	\end{align}			
\end{prop}

The characterization~\eqref{eq:expressionoftangentspace} follows from Lemma~\ref{lem:representationoninvariantmeasure}, and \eqref{eq:relationbetweentangentspaceandsetofequilibriumstate} is a consequence of~\cite[Theorem~5]{walters1992differentiability}.

The final ingredient of our argument is the following variational characterization of the regularized entropy, which can be proved by the same method as~\cite[Theorem~9.12]{walters1982introduction}.

\begin{lemma}	\label{lem:computationforuppersemicontinuityofentropy}
    Let $(X,\rho)$ be a compact metric space and $T\colon X\rightarrow X$ be a continuous map with finite topological entropy.
    Then $\overline{h}_{\mu}(T)=\inf\{P(T,\theta)-\functional{\mu}{\theta} : \theta\in C(X)\}$ for each $\mu\in\cM(X,T)$.
\end{lemma}

We now apply these results to establish the computability of $C(X)_{\potential,P_T}^*$.

\begin{theorem}\label{thm:recursivecompactoftangentspace}
    Let $(X,\,\rho,\,\cS,\,\{X_n\}_{n\in\N},\,\{T_n\}_{n\in\N},\,\{\phi_n\}_{n\in\N})$ be a uniformly computable system with $X_n=X$ for all $n\in\N$.
    Suppose $T_n$ has finite topological pressure for each $n\in\N$.
    Assume that the sequence $\{T_n\}_{n\in\N}$ of transformations and the sequence $\{\potential_n\}_{n\in\N}$ of functions satisfy properties~(i)~and~(ii) in Theorem~\ref{theorem A}.
    For each $n\in\N$, define 
    \begin{equation}\label{eq:defofGamman}
        \Gamma_n\coloneqq\bigl\{\mu\in\cM(X,T_n):\overline{h}_{\mu}(T_n)+\functional{\mu}{\phi_n}=P(T_n,\phi_n)\bigr\}.
    \end{equation}
    Then $\{\Gamma_n\}_{n\in\N}$ is uniformly recursively compact in $(\cP(X),\,W_{\rho},\,\cQ_{\cS})$.
\end{theorem}

\begin{proof}
    For each $n\in\N$, by (\ref{eq:expressionoftangentspace}) in Proposition~\ref{prop:relationbetweentangentspaceandsetofequilibriumstate}, $\Gamma_n$ is indeed the space of functionals which are tangent to $P_{T_n}$ at $\phi_n$.
    Recall that the sequence $\{\psi_{n,i}\}_{(n,i)\in\N^2}$ of functions and  the sequence $\{D_n\}_{n\in\N}$ defined by $D_n\=\{\psi_{n,i}:i\in\N\}$ are given in Theorem~\ref{theorem A}~property~(i).
    In what follows, we establish the following characterization of the spaces $\Gamma_n$.

    \smallskip
    \emph{Claim~1.} For each $n\in\N$, we have $\Gamma_n=\{\mu\in\cM(X,T_n):\inf\{P(T_n,\psi)-\functional{\mu}{\psi}:\psi\in D_n\}\geqslant P(T_n,\potential_n)-\functional{\mu}{\potential_n}\}$.
    \smallskip
    
    \emph{Proof of Claim~1.}
    Fix an arbitrary $n\in\N$. By \eqref{eq:defofGamman} and Lemma~\ref{lem:computationforuppersemicontinuityofentropy}, we have that
    \begin{equation*}		
		\Gamma_n=\{\mu\in\cM(X,T_n) : \inf\{P(T_n,\theta)-\functional{\mu}{\theta}:\theta\in C(X)\}=P(T_n,\potential_n)-\functional{\mu}{\potential_n}\}.
    \end{equation*}
    Hence, to show Claim~1, it suffices to show that for each $\mu\in\cM(X,T_n)$, the following two relations are equivalent:
    \begin{align}
		\inf \{P(T_n,\theta)-\functional{\mu}{\theta} : \theta\in C(X)\} 
        &=P(T_n,\potential_n)-\functional{\mu}{\potential_n},\label{eq:equationequivalent1} \\
		\inf\{P(T_n,\psi)-\functional{\mu}{\psi}:\psi\in D_n\}
        &\geqslant P(T_n,\potential_n)-\functional{\mu}{\potential_n}.\label{eq:equationequivalent2} 
    \end{align}
    
    First, assume $\mu\in\cM(X,T_n)$ satisfies \eqref{eq:equationequivalent1}. It follows from $D_n\subseteq C(X)$ that (\ref{eq:equationequivalent1}) implies (\ref{eq:equationequivalent2}).
    Conversely, assume $\mu\in\cM(X,T_n)$ satisfies \eqref{eq:equationequivalent2}.
    We proceed by contradiction. Suppose $\mu$ does not satisfy \eqref{eq:equationequivalent1}.
    Then there exists $\phi\in C(X)$ such that $P(T_n,\phi)-\functional{\mu}{\phi}<P(T_n,\potential_n)-\functional{\mu}{\potential_n}$. Since $\overline{D}_n$ (the closure of $D_n\coloneqq\{\psi_{n,i}:i\in\N\}$) contains a neighborhood of $\potential_n$ by property~(i) in Theorem~\ref{theorem A}, there exists $c\in(0,1]$ such that $c\phi+(1-c)\phi_n\in \overline{D}_n$.
    Thus, by the convexity and continuity of the pressure function $\phi \mapsto P(T_n, \phi)$, we obtain that
    \begin{align*}
        P(T_n,\potential_n)-\functional{\mu}{\potential_n}
        &>P(T_n,c\phi+(1-c)\phi_n)-\functional{\mu}{c\phi+(1-c)\phi_n}\\
        &\geqslant\inf\{P(T_n,\psi)-\functional{\mu}{\psi}:\psi\in\overline{D}_n\}=\inf\{P(T_n,\psi)-\functional{\mu}{\psi}:\psi\in D_n\},
    \end{align*}
    which contradicts the assumption that $\mu$ satisfies \eqref{eq:equationequivalent2}.
    Hence, $\mu$ must satisfy \eqref{eq:equationequivalent1}, completing the proof of Claim~1.

    \smallskip
    
	For each $n\in\N$, define a function $f_n\colon\PPP(X)\rightarrow\R$ by
	\begin{equation}		\label{eq:definitionoffunctionf}
		f_n(\nu)\define\inf\{P(T_n,\psi_{n,i})-\functional{\nu}{\psi_{n,i}}:i\in\N\}+\functional{\nu}{\potential_n}\quad\text{ for }\nu\in\PPP(X)\text{ and }n\in\N.
	\end{equation} 
	
    \smallskip
    \emph{Claim~2.} $\{f_n\}_{n\in\N}$ is a sequence of uniformly upper semi-computable functions.
    \smallskip

    \emph{Proof of Claim~2.}
    We construct a sequence $\{F_{m,n}\}_{(m,n)\in\N^2}$ of uniformly computable functions such that for each $n\in\N$, the sequence $\{F_{m,n}\}_{m\in\N}$ is nonincreasing and converges pointwise to $f_n$ as $m\to\infty$.
    By property~(i) of Theorem~\ref{theorem A}, $\{P(T_n,\psi_{n,i})\}_{(n,i)\in\N^2}$ is a sequence of uniformly upper semi-computable real numbers.
    Hence, by Definition~\ref{def:semicomputabilityoverR}, there exists an algorithm $\cA_p$ such that for all $m, \, n, \, i\in\N$, on input $m, \, n, \, i$, the algorithm $\cA_p$ outputs $p_{m,n,i}\in\cQ$ such that $\{p_{m,n,i}\}_{(m,n,i)\in\N^3}$ is nonincreasing in $m$ and converges to $P(T_n,\psi_{n,i})$ as $m\rightarrow+\infty$ for all $n,i\in\N$.
    Define integral functionals $\cI_{n,i}(\nu)\define\functional{\nu}{\potential_n-\psi_{n,i}}$ for all $n,i\in\N$ and $\nu\in\PPP(X)$.
    Then by Proposition~\ref{prop:computability of integral operator}, it follows from the uniform computability of $\{\psi_{n,i}\}_{(n,i) \in \N^2}$ and $\{\potential_n\}_{n \in \N}$ that $\{\cI_{n,i}\}_{(n,i) \in \N^2}$ is a sequence of uniformly computable functions.
    Define $F_{m,n}(\nu)\define\min\{p_{m,n,i} + \cI_{n,i}(\nu) \describe i \in \N \cap[1,m]\}$ for all $n,m\in\N$ and $\nu\in\PPP(X)$.
    By the uniform computability of $\{\cI_{n,i}\}_{(n,i)\in\N^2}$, there exists an algorithm $\cA_{\cI}$ such that for all $n,\,i,\,k\in\N$ and $\nu\in\cP(X)$, on input $n,\,i,\,k$ and an oracle for $\nu$, $\cA_{\cI}$ outputs $q_{i,k}$ with $\abs{q_{i,k}-\cI_{n,i}(\nu)}\leqslant 2^{-k}$.
    We can now design an algorithm $\cA_F$ to compute $\{F_{m,n}\}_{(m,n)\in\N^2}$.
    For each $m,\,n,\,k\in\N$ and each $\nu\in\cP(X)$, we apply $\cA_p$ and $\cA_{\cI}$ to compute $\{p_{m,n,i}\}_{i=1}^m$ and $\{q_{i,k}\}_{i=1}^m$ such that $\abs{q_{i,k}-\cI_{n,i}(\nu)}\leqslant 2^{-k}$ for each integer $1\leqslant i\leqslant m$.
    Then we compute $\min\{p_{m,n,i}+q_{i,k}:i\in\N\cap[1,m]\}$ as the output of $\cA_F$.
    Note that $\min\{p_{m,n,i}+q_{i,k}:i\in\N\cap[1,m]\}+2^{-k}\geqslant F_{m,n}(\nu)\geqslant\min\{p_{m,n,i}+q_{i,k}:i\in\N\cap[1,m]\}-2^{-k}$.
    It follows that $\cA_F$ demonstrates the uniform computability of the sequence $\{F_{m,n}\}_{(m,n)\in\N^2}$ of functions.
    
    Now fix an integer $n\in\N$ and a measure $\nu\in\PPP(X)$.
    Since $\{p_{m,n,i}\}_{(m,i)\in\N^2}$ is nonincreasing in $m$ for each $i\in\N$, the sequence $\{F_{m,n}(\nu)\}_{m\in\N}$ is nonincreasing; so its limit exists.
    Note that $p_{m,n,i}\geqslant P(T_n,\psi_{n,i})$ for all $m,\,i\in\N$.
    Then by construction, we have $F_{m,n}(\nu)\geqslant f_n(\nu)$ for each $m\in\N$.
    Hence, $\lim_{m\to+\infty}F_{m,n}(\nu)\geqslant f_n(\nu)$.
    On the other hand, since $\lim_{m\to+\infty}p_{m,n,i}=P(T_n,\psi_{n,i})$ for each $i\in\N$, it follows from \eqref{eq:definitionoffunctionf} that for each $\epsilon>0$, there exist integers $i, j$ with $1 \leq i \leq j$ such that $f_n(\nu)+\epsilon>P(T_n,\psi_{n,i})+\cI_{n,i}(\nu)+\epsilon/2>p_{j,n,i}+\cI_{n,i}(\nu)\geqslant F_{j,n}(\nu)\geqslant\lim_{m\to+\infty}F_{m,n}(\nu)$.
    Consequently, $f_n(\nu)\geqslant\lim_{m\to +\infty}F_{m,n}(\nu)$.
    This shows that $\lim_{m\to+\infty}F_{m,n}(\nu)=f_n(\nu)$. 
    
    Since $\{F_{m,n}\}_{(m,n)\in\N^2}$ is a sequence of uniformly computable functions and $\{F_{m,n}(\nu)\}_{m\in\N}$ is nonincreasing for all $\nu$, by Definition~\ref{def:domain of lower or upper computability}, $\{f_n\}_{n\in\N}$ is a sequence of uniformly upper semi-computable functions.
    This establishes Claim~2.
	
    \smallskip

    We now complete the proof by expressing $\Gamma_n$ as the complement of a uniformly lower semi-computable open set within the uniformly recursively compact set $\cM(X,T_n)$.
    By property~(ii) of Theorem~\ref{theorem A}, $\{P(T_n,\potential_n)\}_{n\in\N}$ is a sequence of uniformly lower semi-computable real numbers. By Definition~\ref{def:semicomputabilityoverR}, there exists a sequence $\{q_{m,n}\}_{(m,n)\in\N^2}$ of uniformly computable real numbers such that, for each $n\in\N$, the sequence $\{q_{m,n}\}_{m\in\N}$ is nondecreasing in $m$ and converges to $P(T_n,\potential_n)$ as $m\to +\infty$.
    Hence, it follows from the definition of $\{D_n\}_{n\in\N}$, Claim~1, and \eqref{eq:definitionoffunctionf} that
	\begin{equation}		\label{eq:tangentspaceisintersection}
		\Gamma_n=\cM(X,T_n)\cap f_n^{-1}([P(T_n,\potential_n),+\infty))=\cM(X,T_n)\smallsetminus\bigcup_{m\in\N}f_n^{-1}((-\infty,q_{m,n})).
	\end{equation}
    Since $\{q_{m,n}\}_{(m,n)\in\N^2}$ is a sequence of uniformly computable real numbers, by Claim~2 and the implication from (i) to (iii) in Proposition~\ref{prop:semicomputability}, $\bigl\{f_n^{-1}((-\infty,q_{m,n}))\bigr\}_{(m,n)\in\N^2}$ is uniformly lower semi-computable open in $(\cP(X),\,W_{\rho},\,\cQ_{\cS})$.
    Hence, by Proposition~\ref{prop:lower semi-computable open}, $\bigl\{\bigcup_{m\in\N}f_n^{-1}((-\infty,q_{m,n}))\bigr\}_{n\in\N}$ is uniformly lower semi-computable open.
    Moreover, since $\{T_n\}_{n\in\N}$ is a sequence of uniformly computable functions, by Proposition~\ref{prop:recursivecompactofsetinvariant}, $\{\cM(X,T_n)\}_{n\in\N}$ is uniformly recursively compact.
    Therefore, by Proposition~\ref{thecomplementofrecursivelycompactset}~(iii) and \eqref{eq:tangentspaceisintersection}, $\{\Gamma_n\}_{n\in\N}$ is uniformly recursively compact.
\end{proof}

Now we turn to prove Theorem~\ref{theorem A}.

\begin{proof}[\bf Proof of Theorem~\ref{theorem A}]
    Let $\Gamma_n$ be as defined in (\ref{eq:defofGamman}). Consider an arbitrary $n\in\N$. Then the measure-theoretic entropy map $\nu \mapsto h_{\nu}(T_n)$ is upper semicontinuous by hypothesis. This implies that $\overline{h}_{\nu}(T_n)=h_{\nu}(T_n)$ for all $\nu\in\cM(X,T_n)$.
    It follows from (\ref{eq:defofGamman}) that the set $\Gamma_n$ coincides with the set of equilibrium states $\mathcal{E}(T_n,\potential_n)$.
    Hence, by property~(iii), we have $\Gamma_n = \mathcal{E}(T_n,\potential_n) = \{\mu_n\}$.
    
    By Theorem~\ref{thm:recursivecompactoftangentspace}, the sequence $\{\{\mu_n\}\}_{n\in\N}$ of singletons is uniformly recursively compact in $(\cP(X),\,W_{\rho},\,\cQ_{\cS})$.
    Therefore, by Proposition~\ref{thecomplementofrecursivelycompactset}~(i), $\{\mu_n\}_{n\in\N}$ is uniformly computable.
\end{proof}

\section{Approach II: Local approach via transfer operator and Jacobian}
\label{sec:Approach II}

This section is dedicated to the proof of Theorem~\ref{theorem B} and its applications. 
Subsection~\ref{subsec:jacobian and transfer operator} recalls essential notions and establishes Theorem~\ref{thm:prescribed Jacobian of equilibrium state}, which states the relationship between the set $\cM(X,T;Y,J)$ and $\cE(T,\phi)$ under additional conditions.
The proof of Theorem~\ref{theorem B} follows in Subsection~\ref{subsec:proof of theorem B}. 
Finally, in Subsection~\ref{sub:applications}, we apply these results to demonstrate the computability of equilibrium states.


\subsection{Jacobian and the transfer operator}
\label{subsec:jacobian and transfer operator}

We begin by recalling the definition of Jacobians and establishing their existence, uniqueness, and Rokhlin's formula. We then define transfer operators and establish their key properties in our setting.
Using these tools, we provide an equivalent characterization of Jacobians with respect to invariant measures in Theorem~\ref{thm:description of Jacobian}.
Building on this, the subsection concludes with Theorem~\ref{thm:prescribed Jacobian of equilibrium state}.


\begin{definition}[Jacobian]    \label{def:Jacobian}
	Let $(X,\rho)$ be a compact metric space, and $T\colon X\rightarrow X$ be a Borel-measurable transformation. We say that $A\subseteq X$ is \emph{admissible} (for $T$) if $A,\,T(A)\in\cB(X)$, and $T|_{A}$ is injective. Suppose $J\colon X\rightarrow [0,+\infty)$ is a Borel function, $\mu\in\mathcal{P}(X)$, and $E\in\cB(X)$ with $\mu(E)=1$. Then $J$ is said to be a \emph{Jacobian on $E$} for $T$ with respect to $\mu$ if for all admissible sets $A \subseteq E$,
	\begin{equation*}    \label{eq:def:Jacobian}
		\mu(T(A)) = \int_{A} \! J \,\mathrm{d}\mu.
	\end{equation*}
	Moreover, we say that $J$ is a \emph{Jacobian} for $T$ with respect to $\mu$ if there exists $\widetilde{E}\in\cB(X)$ with $\mu\bigl(\widetilde{E}\bigr)=1$ such that $J$ is a Jacobian on $\widetilde{E}$ for $T$ with respect to $\mu$.  
\end{definition}

Recall that $\cP(X;Y)=\{\mu\in\cP(X):\mu(Y)=1\}$ for $Y\in\cB(X)$. We state below the hypotheses under which we will develop our theory in this section. 

\begin{definition}\label{def:admissible}
    We say that the sextuple $(X,\,\rho,\,T,\,Y,\,\{Y_k\}_{k\in\N},\,\mu)$ is \emph{admissible} if it has the following properties:
\begin{enumerate}
	\smallskip
	\item $(X,\rho)$ is a compact metric space.
	\smallskip
	\item $T\colon X\mapping X$ is a Borel-measurable transformation.
    \smallskip
	\item $\{Y_k\}_{k\in\N}$ is a sequence of pairwise disjoint admissible sets for $T$.
    \smallskip
    \item $Y=\bigcup_{k\in\N}Y_k$.
    \smallskip
    \item $\mu\in\cM(X,T)\cap\cP(X;Y)$.
\end{enumerate}
\end{definition}

The following proposition states the uniqueness of the Jacobian and provides a lower bound for the measure-theoretic entropy in terms of the Jacobian. 

\begin{prop}\label{prop:rokhlin'sformula}
	Let $(X,\,\rho,\,T,\,Y,\,\{Y_k\}_{k\in\N},\,\mu)$ be admissible. Assume that $J\colon X\mapping[0,+\infty)$ is a Jacobian for $T$ with respect to $\mu$. Then $J(x)\geq 1$ for $\mu$-a.e.\ $x\in X$, and $h_{\mu}(T)\geq\functional{\mu}{\log(J)}$. Moreover, for each Borel function $\widetilde{J}\colon X\mapping[0,+\infty)$, $\widetilde{J}$ is a Jacobian for $T$ with respect to $\mu$ if and only if $J(x)=\widetilde{J}(x)$ for $\mu$-a.e.\ $x\in X$.
\end{prop}

The lower bound given above is a classical result in ergodic theory known as the Rokhlin entropy formula. We refer the reader to \cite[Theorem~4.2]{sarig1999thermodynamic} for a version for topological Markov shifts and to \cite[Corollary~12.1]{coudène2016ergodic} for a version for finite admissible partitions; see \cite[Proposition~4.3]{Qiandu2025Lecture} for the proof in our context. The uniqueness of the Jacobian immediately follows from \cite[Theorem~2.7]{rohlin1949fundamental}, \cite[Definition~2.9.2~\&~Proposition~2.9.5]{przytycki2010conformal}.

Next, we construct a specific Jacobian, define a useful operator, and establish its key properties in the following proposition.
With these properties, this operator can be seen as a normalized transfer operator, which is essential for the proof of Theorem~\ref{thm:description of Jacobian}.

\begin{prop}\label{prop:transfer operator}
	Let $(X,\,\rho,\,T,\,Y,\,\{Y_k\}_{k\in\N},\,\mu)$ be admissible. Then the following statements are true: 
	\begin{enumerate}
		
		\smallskip
		
		\item For each $k\in\N$, there exists a nonnegative $\mu$-integrable Borel function $\Phi_k$ on $X$ such that 
		\begin{equation}\label{eq:inverseJacobian}
			\Phi_k(x)=0\quad\text{ for each }x\notin T(Y_k)\quad\text{ and}\quad\mu\bigl(T^{-1}(B)\cap Y_k\bigr)=\int_{B}\!\Phi_k\,\mathrm{d}\mu\quad\text{ for each }B\in\cB(X).
		\end{equation}
				
		\smallskip
		
		\item Define a function $\Psi\colon X\mapping\R$ by
		\begin{equation}		\label{eq:def of psi}
			\Psi(x)\=0\quad\text{ for each }x\in Y^c\quad\text{ and}\quad\Psi(x)\=\Phi_k(T(x))\quad\text{ for each }k\in\N\text{ and each }x\in Y_k.
		\end{equation}
		Write $\widetilde{Y}\=Y\smallsetminus\Psi^{-1}(0)$. Then there exists a Borel function $J_{\mu}\colon X\mapping [0,+\infty)$ that is a Jacobian on $\widetilde{Y}$ for $T$ with respect to $\mu$ and satisfies $J_{\mu}(x)\cdot\Psi(x)=1$ for $\mu$-a.e.\ $x\in X$.
		
		\smallskip
		
		\item Denote by $L^+(X)$ the space of all Borel functions from $X$ to $[0,+\infty]$. Then $\cL_{\mu}\colon L^+(X)\mapping L^+(X)$ given by
		\begin{equation}		\label{eq:transfer operator}
			\cL_{\mu}(u)(x) \= \sum_{y\in T^{-1}(x)\cap Y}u(y)\Psi(y),\quad\text{ for }u\in L^+(X)\text{ and }x\in X,
		\end{equation}
		satisfies the following: for all $u,v\in L^+(X)$, $c\geq 0$, and $\mu$-a.e.\ $x\in X$,
		\begin{align}
			\cL_{\mu}(\mathbbm{1})(x)&=1,\label{eq:invariance of transfer operator}\\
			\cL_{\mu}(u+v)(x)=\cL_{\mu}(u)(x)+\cL_{\mu}(v)&(x),\quad\cL_{\mu}(cu)(x)=c\cL_{\mu}(u)(x),\quad\text{ and}\label{eq:operatoronL^+X}\\
			\functional{\mu}{\cL_{\mu}(u)}&=\functional{\mu}{u}.\label{eq:integration of transfer opertor}
		\end{align}
	\end{enumerate}
\end{prop} 

We prove (i) and (ii) by using the Radon--Nikodym theorem to construct ${\Phi_k}$ and $J_{\mu}$, then showing their reciprocal relationship with $\Psi$ via change of variables. 
To establish (iii), we shall apply the monotone convergence theorem repeatedly. 
We prove (\ref{eq:invariance of transfer operator}) by testing $\int_A\!\cL_{\mu}(\mathbbm{1})\,\mathrm{d}\mu = \mu(A)$ for all Borel $A$. 
Finally, (\ref{eq:operatoronL^+X}) follows directly from (\ref{eq:transfer operator}), and the integration formula (\ref{eq:integration of transfer opertor}) is checked on characteristic functions.

\begin{proof}
	(i) Fix an arbitrary $k\in\N$. Since $(X,\,\rho,\,T,\,Y,\,\{Y_k\}_{k\in\N},\,\mu)$ is admissible, by Definition~\ref{def:admissible}~(ii), (iii), and~(v), $T$ is Borel measurable, $Y_k,\,T(Y_k)\in\cB(X)$, and $\mu\in\cM(X,T)\cap\cP(X;Y)$.
	Define $\mu_k(B)\coloneqq\mu(B)$ and $\widetilde{\mu}_k(B) \=  \mu\bigl((T|_{Y_k})^{-1}(B)\bigr)=\mu\bigl(T^{-1}(B)\cap Y_k\bigr)$ for each Borel subset $B\subseteq T(Y_k)$. Then $\mu_k$ and $\widetilde{\mu}_k$ are both $\sigma$-finite positive Borel measures on $T(Y_k)$.
	Since $\mu\in\cM(X,T)\cap\cP(X;Y)$, we obtain that $\widetilde{\mu}_k(B)=\mu\bigl(T^{-1}(B)\cap Y_k\bigr)\leq\mu\bigl(T^{-1}(B)\bigr)=\mu(B)=\mu_k(B)$ for each Borel subset $B\subseteq T(Y_k)$. 
	This implies that $\widetilde{\mu}_k$ is absolutely continuous with respect to $\mu_k$.
	By the Radon--Nikodym theorem, there exists a nonnegative $\mu_k$-integrable derivative $\frac{\mathrm{d}\widetilde{\mu}_k}{\mathrm{d}\mu_k}$. Define $\Phi_k:X\rightarrow\R$ by
	\begin{equation*}
		\Phi_k(x) \=
		\begin{cases}
			\frac{\mathrm{d}\widetilde{\mu}_k}{\mathrm{d}\mu_k}(x)&\text{if }x\in T(Y_k);\\
			0&\text{if }x\notin T(Y_k).
		\end{cases}
	\end{equation*}
	By construction, $\Phi_k$ is a nonnegative $\mu$-integrable Borel function. Moreover, by the definition of the Radon--Nikodym derivative,
	\begin{equation*}
		\mu\bigl(T^{-1}(B)\cap Y_k\bigr)=\mu\bigl(T^{-1}(B\cap T(Y_k))\cap Y_k\bigr)=\int_{B\cap T(Y_k)}\!\Phi_k\,\mathrm{d}\mu=\int_B\!\Phi_k\,\mathrm{d}\mu\quad\text{ for each }B\in\cB(X).
	\end{equation*}
	Hence, $\Phi_k$ satisfies (\ref{eq:inverseJacobian}), establishing Proposition~\ref{prop:transfer operator}~(i).
	
	\smallskip

	(ii)
	By (\ref{eq:def of psi}), $\Psi$ is a Borel function on $X$, and thus $\widetilde{Y}\in\cB(X)$. We first prove that $\mu\bigl(\widetilde{Y}\bigr)=1$. By Proposition~\ref{prop:transfer operator}~(i) and (\ref{eq:def of psi}), we have $\mu\bigl(\Psi^{-1}(0)\cap Y_k\bigr)=\mu\bigl(T^{-1}\bigl(\Phi_k^{-1}(0)\bigr)\cap Y_k\bigr)=\int_{\Phi_k^{-1}(0)}\!\Phi_k\,\mathrm{d}\mu=0$ for each $k\in\N$. Since $\{Y_k\}_{k\in\N}$ is a sequence of pairwise disjoint Borel subsets with $\mu(Y)=\mu\bigl(\bigcup_{k\in\N}Y_k\bigr)=1$ by Definition~\ref{def:admissible}~(iv)~and~(v), we obtain $\mu\bigl(\widetilde{Y}\bigr)=\mu(Y)-\mu\bigl(\Psi^{-1}(0)\cap Y\bigr)=1-\sum_{k\in\N}\mu\bigl(\Psi^{-1}(0)\cap Y_k\bigr)=1$.

	Now fix an arbitrary $k\in\N$ and write $\widetilde{Y}_k\=Y_k\smallsetminus\Psi^{-1}(0)$.
    Since $\Psi$ is a Borel function, we have $\widetilde{Y}_k\in\cB(X)$. Define $\nu_k(B)\coloneqq\mu(B)$ and $\tilde{\nu}_k(B)\=\mu(T(B))$ for each Borel subset $B\subseteq\widetilde{Y}_k$.
    By \cite[Corollary~15.2]{kechris2012classical}, it follows from Definition~\ref{def:admissible}~(iii) that $T$ is a Borel isomorphism of $Y_k$ with $T(Y_k)$. 
    Hence, by Definition~\ref{def:admissible}~(iii)~and~(v), $\nu_k$ and $\tilde{\nu}_k$ are both $\sigma$-finite positive Borel measures on $\widetilde{Y}_k$. We next prove that $\tilde{\nu}_k$ is absolutely continuous with respect to $\nu_k$. To this end, consider an arbitrary Borel set $A\subseteq\widetilde{Y}_k$ with $\mu(A)=0$; we show that $\mu(T(A))=0$. Indeed, since $T$ is injective on $Y_k$, we have $T^{-1}(T(A))\cap Y_k=A$. Thus, by Proposition~\ref{prop:transfer operator}~(i),
	\begin{equation}\label{eq:integrationequalstozero}
		0=\mu(A)=\mu\bigl(T^{-1}(T(A))\cap Y_k\bigr)=\int_{T(A)}\!\Phi_k\,\mathrm{d}\mu.
	\end{equation}
	Since $A\subseteq\widetilde{Y}_k=Y_k\smallsetminus\Psi^{-1}(0)$ and $\Psi(x)=\Phi_k(T(x))$ for each $x\in Y_k$ by (\ref{eq:def of psi}), we have $\Phi_k(x)>0$ for each $x\in T(A)$.
    Combined with (\ref{eq:integrationequalstozero}), this implies that $\mu(T(A))=0$. Thus, $\tilde{\nu}_k$ is absolutely continuous with respect to $\nu_k$. By the Radon--Nikodym theorem, there exists a nonnegative $\mu$-integrable derivative $\frac{\mathrm{d}\tilde{\nu}_k}{\mathrm{d}\nu_k}$.
	
	We define $J_{\mu}$ by
	\begin{equation}\label{eq:constructionofJmu}
		J_{\mu}(x)\=0\text{ for each } x \in \Psi^{-1}(0) \quad\text{ and }\quad J_{\mu}(x)\=\frac{\mathrm{d}\tilde{\nu}_k}{\mathrm{d}\nu_k}(x)\text{ for all }k\in\N\text{ and }x\in \widetilde{Y}_k.
	\end{equation}
	We now verify that $J_{\mu}$ is a Jacobian on $\widetilde{Y}$ for $T$ with respect to $\mu$. Since $\widetilde{Y}=\bigcup_{k\in\N}\widetilde{Y}_k$, by the definitions of $\{\nu_k\}_{k\in\N}$ and $\{\tilde{\nu}_k\}_{k\in\N}$, it follows from (\ref{eq:constructionofJmu}) that
	\begin{equation*}
		\begin{aligned}
			\mu(T(A))&=\sum_{k\in\N}\mu\bigl(T\bigl(A\cap\widetilde{Y}_k\bigr)\bigr)=\sum_{k\in\N}\tilde{\nu}_k\bigl(A\cap\widetilde{Y}_k\bigr)=\sum_{k\in\N}\int_{A\cap\widetilde{Y}_k}\!\frac{\mathrm{d}\tilde{\nu}_k}{\mathrm{d}\nu_k}\,\mathrm{d}\nu_k\\
			&=\sum_{k\in\N}\int_{A\cap\widetilde{Y}_k}\!\frac{\mathrm{d}\tilde{\nu}_k}{\mathrm{d}\nu_k}\,\mathrm{d}\mu=\sum_{k\in\N}\int_{A\cap\widetilde{Y}_k}\!J_{\mu}\,\mathrm{d}\mu=\int_{A\cap\widetilde{Y}}\!J_{\mu}\,\mathrm{d}\mu=\int_{A}\!J_{\mu}\,\mathrm{d}\mu,
		\end{aligned}
	\end{equation*}
	for each admissible set $A\subseteq\widetilde{Y}$ for $T$. Since $\mu\bigl(\widetilde{Y}\bigr)=1$, we conclude that $J_{\mu}$ is a Jacobian on $\widetilde{Y}$.

	Finally, we verify that $J_{\mu}(x)\cdot\Psi(x)=1$ for $\mu$-a.e.\ $x\in X$. By \cite[Corollary~15.2]{kechris2012classical}, it follows from Definition~\ref{def:admissible}~(iii) that $T$ is a Borel isomorphism of $Y_k$ with $T(Y_k)$ for each $k\in\N$. Since $J_{\mu}$ is a Jacobian on $\widetilde{Y}$ for $T$ with respect to $\mu$, by the definition of $\{\tilde{\nu}_k\}_{k\in\N}$, we have
	\begin{equation}		\label{eq:bypropertyofJacobian}
		\tilde{\nu}_k(A)=\mu(T(A))=\int_{A}\!J_{\mu}\,\mathrm{d}\mu\quad\text{ for each }k\in\N\text{ and each Borel }A\subseteq\widetilde{Y}_k.
	\end{equation}
	Moreover, by Proposition~\ref{prop:transfer operator}~(i), we obtain that
	\begin{equation*}
		\mu(A)
        =\mu\bigl(T^{-1}(T(A))\cap Y_k\bigr)
        =\int_{T(A)}\!\Phi_k\,\mathrm{d}\mu
        =\int_A\! \Phi_k\circ T \,\mathrm{d}(\mu\circ T)
        =\int_A\!\Psi\,\mathrm{d}\tilde{\nu}_k
    \end{equation*}
	for each $k\in\N$ and each Borel $A\subseteq\widetilde{Y}_k$. Thus, by (\ref{eq:bypropertyofJacobian}) and the change-of-variable formula, we obtain $J_{\mu}(x)\cdot\Psi(x)=1$ for $\mu$-a.e.\ $x\in\bigcup_{k\in\N}\widetilde{Y}_k$. Since $\mu\bigl(\bigcup_{k\in\N}\widetilde{Y}_k\bigr)=\mu\bigl(\widetilde{Y}\bigr)=1$, we conclude that $J_{\mu}(x)\cdot\Psi(x)=1$ for $\mu$-a.e.\ $x\in X$.
	
	\smallskip

	(iii) By Definition~\ref{def:admissible}~(iii)~and~(iv), we have that $Y=\bigcup_{k\in\N}Y_k$ and $T$ is injective on $Y_k$ for each $k\in\N$. Since $\Psi$ is nonnegative and Borel, the expression for $\cL_{\mu}(u)(x)$ given in (\ref{eq:transfer operator}) is the sum of countably (possibly infinitely) many nonnegative terms for each $u\in L^+(X)$ and each $x\in X$. Hence, $\cL_{\mu}(u)\in L^+(X)$ for each $u\in L^+(X)$, and $\cL_{\mu}\colon L^+(X)\mapping L^+(X)$ is monotone.

	We now establish (\ref{eq:invariance of transfer operator}). By (\ref{eq:transfer operator}), we have that
	\begin{equation}\label{eq:expressionofl1d}
		\cL_{\mu}(\mathbbm{1}_A)(x)=\sum_{y\in T^{-1}(x)\cap Y}\mathbbm{1}_A(y)\Psi(y)=\sum_{y\in T^{-1}(x)\cap A\cap Y}\Psi(y)\quad\text{ for each }A\in\cB(X).
	\end{equation}
	Then by (\ref{eq:expressionofl1d}), (\ref{eq:def of psi}), $Y=\bigcup_{k\in\N}Y_k$, (\ref{eq:inverseJacobian}), the monotone convergence theorem, and $\mu\in\cM(X,T)\cap\cP(X;Y)$, we obtain that
	\begin{equation*}
		\begin{aligned}
			\int_A\!\cL_{\mu}(\mathbbm{1})\,\mathrm{d}\mu
			&=\int_A  \sum_{y\in T^{-1}(x)\cap Y}\Psi(y) \,\mathrm{d}\mu(x)
			=\int_A  \sum_{k\in\N:T^{-1}(x)\cap Y_k\neq\emptyset}\Phi_k \,\mathrm{d}\mu
            =\int_A  \sum_{k\in\N}\Phi_k \,\mathrm{d}\mu\\
			&=\sum_{k\in\N}\int_A\!\Phi_k\,\mathrm{d}\mu
            =\sum_{k\in\N}\mu\bigl(T^{-1}(A)\cap Y_k\bigr)=\mu\bigl(T^{-1}(A)\cap Y\bigr)=\mu\bigl(T^{-1}(A)\bigr)=\mu(A)
		\end{aligned}
	\end{equation*}
	for each $A\in\cB(X)$. 
	This establishes (\ref{eq:invariance of transfer operator}).

	To establish (\ref{eq:operatoronL^+X}), consider arbitrary $u,\,v\in L^+(X)$ and $c\geq 0$. By (\ref{eq:transfer operator}), we have $\cL_{\mu}(u+v)(x)=\cL_{\mu}(u)(x)+\cL_{\mu}(v)(x)$ for each $x\in X$. Moreover, $\cL_{\mu}(cu)(x)=c\cL_{\mu}(u)(x)$ for each $x\in X$ and each nonnegative $c\in\Q$. By the monotonicity of $\cL_{\mu}$, it follows from the monotone convergence theorem that $\cL_{\mu}(cu)(x)=c\cL_{\mu}(u)(x)$ for each $c\geq 0$. This establishes (\ref{eq:operatoronL^+X}).
	
	Finally, we establish (\ref{eq:integration of transfer opertor}).
	Consider an arbitrary $D\in\cB(X)$. Fix $k\in\N$ and write $D_k\=D\cap Y_k$. By (\ref{eq:expressionofl1d}), $\cL_{\mu}\bigl(\mathbbm{1}_{D_k}\bigr)(x)=0$ for each $x\notin T(D_k)$. Now consider an arbitrary $x\in T(D_k)$. Note that $T$ is injective on $Y_k$ by Definition~\ref{def:admissible}~(iii). Then the set $T^{-1}(x)\cap D_k$ contains a unique element, which we denote by $y_k(x)$. 
	Using (\ref{eq:expressionofl1d}) and (\ref{eq:def of psi}), we deduce that $\cL_{\mu}\bigl(\mathbbm{1}_{D_k}\bigr)(x)=\Psi(y_k(x))=\Phi_k(x)$. 
	Integrating this and applying Proposition~\ref{prop:transfer operator}~(i) yields
	\begin{equation*}
    	\Functional{\mu}{\cL_{\mu}\bigl(\mathbbm{1}_{D_k}\bigr)} = \int_{T(D_k)}\!\Phi_k\,\mathrm{d}\mu = \mu\bigl(T^{-1}(T(D_k))\cap Y_k\bigr) = \mu(D_k).
	\end{equation*}
	Thus it follows from (\ref{eq:expressionofl1d}) and Definition~\ref{def:admissible}~(iv) that $\Functional{\mu}{\cL_{\mu}(\mathbbm{1}_D)}=\sum_{k\in\N}\Functional{\mu}{\cL_{\mu}\bigl(\mathbbm{1}_{D_k}\bigr)}=\sum_{k\in\N}\mu(D_k)=\mu(D\cap Y)=\mu(D)$, where the last equality holds as $\mu\in\cP(X;Y)$.
	This establishes the identity for indicator functions. The general result~\eqref{eq:integration of transfer opertor} for all nonnegative measurable functions then follows from a standard argument using the monotone convergence theorem.
\end{proof}

The following theorem characterizes Jacobians in terms of an eigenfunction equation for the transfer operator. 
This characterization will be key in identifying equilibrium states.

\begin{theorem}\label{thm:description of Jacobian}
	Let $(X,\,\rho,\,T,\,Y,\,\{Y_k\}_{k\in\N},\,\mu)$ be admissible. Assume that $J_{\mu}$ is given as in Proposition~\ref{prop:transfer operator}, and $J$ is a positive Borel function on $X$. Then $J$ is a Jacobian on $Y$ for $T$ with respect to $\mu$ if and only if			\begin{align}
		\sum_{y\in T^{-1}(x)\cap Y}\frac{1}{J(y)}&=1\quad\text{ for } \mu\text{-a.e.\ } x\in X,\label{eq:1 is the eigen function of normalized Ruelle operator}\\
		\log(J)\in L^1(\mu),\quad&\text{ and }\quad \functional{\mu}{\log(J)}=\Functional{\mu}{\log\bigl(J_{\mu}\bigr)}.\label{eq:integrationoflogJ}
	\end{align}
\end{theorem}

\begin{proof}
	Let $\{\Phi_k\}_{k\in\N},\,\Psi$, and $\cL_{\mu}$ be as in Proposition~\ref{prop:transfer operator}. Write $\widetilde{Y}\=Y\smallsetminus\Psi^{-1}(0)$.

	We first prove the forward implication.
    Assume that $J$ is a Jacobian on $Y$. We show that $J$ satisfies (\ref{eq:1 is the eigen function of normalized Ruelle operator}) and (\ref{eq:integrationoflogJ}). By Proposition~\ref{prop:transfer operator}~(ii), $J_{\mu}$ is a Jacobian on $\widetilde{Y}$. By Proposition~\ref{prop:rokhlin'sformula}, $J(x)=J_{\mu}(x)\geq 1$ for $\mu$-a.e.\ $x\in X$. Thus, $\log(J),\,\log\bigl(J_{\mu}\bigr)\in L^1(\mu)$ and $\functional{\mu}{\log(J)}=\Functional{\mu}{\log\bigl(J_{\mu}\bigr)}$. This establishes (\ref{eq:integrationoflogJ}).
	
	We now establish (\ref{eq:1 is the eigen function of normalized Ruelle operator}). By Proposition~\ref{prop:transfer operator}~(ii), there exists a Borel set $H_1\subseteq\widetilde{Y}$ such that $\mu(H_1)=1$ and $J(x)\cdot\Psi(x)=1$ for each $x\in H_1$. For each $x\notin T(Y\smallsetminus H_1)$, we have $T^{-1}(x)\cap Y\subseteq H_1$, and hence $J(y)\cdot\Psi(y)=1$ for each $y\in T^{-1}(x)\cap Y$. Since $J$ is a Jacobian on $Y$, it follows from $\mu(Y\smallsetminus H_1)=0$ that $\mu(T(Y\smallsetminus H_1))=0$. Thus, for $\mu$-a.e.\ $x\in X$, we have $J(y)\cdot\Psi(y)=1$ for each $y\in T^{-1}(x)\cap Y$. Combined with (\ref{eq:invariance of transfer operator}), this implies that $\sum_{y\in T^{-1}(x)\cap Y}\frac{1}{J(y)}=\sum_{y\in T^{-1}(x)\cap Y}\Psi(y)=\cL_{\mu}(\mathbbm{1}_X)(x)=1$ for $\mu$-a.e.\ $x\in X$.

	We now establish the backward implication. Assume that $J$ satisfies (\ref{eq:1 is the eigen function of normalized Ruelle operator}) and (\ref{eq:integrationoflogJ}). We show that $J$ is a Jacobian on $Y$. By Proposition~\ref{prop:transfer operator}~(ii), there exists a Borel set $H_2\subseteq\widetilde{Y}$ such that $\mu(H_2)=1$ and $J_{\mu}(x)\cdot\Psi(x)=1$ for each $x\in H_2$. For each $x\notin T\bigl(\widetilde{Y}\smallsetminus H_2\bigr)$, we have $T^{-1}(x)\cap\widetilde{Y}\subseteq H_2$, and hence $J_{\mu}(y)\cdot\Psi(y)=1$ for each $y\in T^{-1}(x)\cap\widetilde{Y}$. Since $J_{\mu}$ is a Jacobian on $\widetilde{Y}$, it follows from $\mu\bigl(\widetilde{Y}\smallsetminus H_2\bigr)=0$ that $\mu\bigl(T\bigl(\widetilde{Y}\smallsetminus H_2\bigr)\bigr)=0$. Thus, we obtain that
	\begin{equation}\label{eq:psiisinverseofJ}
		\text{for }\mu\text{-a.e.\ }x\in X,\quad J_{\mu}(y)\cdot\Psi(y)=1\quad\text{ for each }y\in T^{-1}(x)\cap\widetilde{Y}.
	\end{equation}
	Since $J$ is a positive Borel function and $J_{\mu}$ is a nonnegative Borel function, we have $J_{\mu}\big/J\in L^+(X)$. By (\ref{eq:transfer operator}), $\widetilde{Y}=Y\smallsetminus\Psi^{-1}(0)$, (\ref{eq:psiisinverseofJ}), and (\ref{eq:1 is the eigen function of normalized Ruelle operator}), we obtain that for $\mu$-a.e.\ $x\in X$,
	\begin{equation}\label{eq:1 is the eigen function of normalized Ruelle operator''}
		\begin{aligned}
			\cL_{\mu}\bigl(J_{\mu}\big/J\bigr)(x)
            &=\sum_{y\in T^{-1}(x)\cap Y}\frac{J_{\mu}(y)\Psi(y)}{J(y)}
            =\sum_{y\in T^{-1}(x)\cap\widetilde{Y}}\frac{J_{\mu}(y)\Psi(y)}{J(y)}\\
			&=\sum_{y\in T^{-1}(x)\cap\widetilde{Y}}\frac{1}{J(y)}
            \leq\sum_{y\in T^{-1}(x)\cap Y}\frac{1}{J(y)}
            =1.
		\end{aligned}
	\end{equation}
	By (\ref{eq:integration of transfer opertor})~and~(\ref{eq:integrationoflogJ}), we obtain that
	\begin{equation}		\label{eq:inequality become equality}
		1
		= \functional{\mu}{\mathbbm{1}_X}
		\geq\functional[\big]{\mu}{\cL_{\mu}\bigl(J_{\mu}\big/J\bigr)}
		= \functional[\big]{\mu}{J_{\mu}\big/J}
		\geq 1 - \functional{\mu}{\log(J)} + \functional[\big]{\mu}{\log\bigl(J_{\mu}\bigr)}
		= 1.
	\end{equation}
	The last inequality holds since $x \geq 1 + \log(x)$ for each $x>0$, with equality if and only if $x=1$. 
	Thus, all inequalities in (\ref{eq:inequality become equality}) are equalities. 
	Hence, $J_{\mu}(x)=J(x)$ and $\cL_{\mu}\bigl(J_{\mu} \big/ J\bigr)(x)=1$ for $\mu$-a.e.\ $x\in X$. 
	This implies that the inequality in (\ref{eq:1 is the eigen function of normalized Ruelle operator''}) is an equality.
	Since $J$ is a positive function, we have $T^{-1}(x)\cap\widetilde{Y}=T^{-1}(x)\cap Y$ for $\mu$-a.e.\ $x\in X$. For each $x\in T \bigl(Y\smallsetminus\widetilde{Y}\bigr)$, we have $T^{-1}(x)\cap\widetilde{Y}\neq T^{-1}(x)\cap Y$. This implies $\mu\bigl(T\bigl(Y\smallsetminus\widetilde{Y}\bigr)\bigr)=0$. Since $J_{\mu}$ is a Jacobian on $\widetilde{Y}$ and $J(x)=J_{\mu}(x)$ for $\mu$-a.e.\ $x\in X$, we conclude that $J$ is a Jacobian on $\widetilde{Y}$. By $\mu\bigl(\widetilde{Y}\bigr)=1$, we obtain $\mu(T(A))=\mu\bigl(T\bigl(A\cap \widetilde{Y}\bigr)\bigr)+\mu\bigl(T\bigl(A\smallsetminus\widetilde{Y}\bigr)\bigr)=\int_{A\cap\widetilde{Y}}\!J\,\mathrm{d}\mu+0=\int_{A}\!J\,\mathrm{d}\mu$ for each admissible set $A\subseteq Y$. Therefore, $J$ is a Jacobian on $Y$.
\end{proof}

Recall the definition of $\cM(X,T;Y,J)$ from (\ref{eq:defofcMforJonsubset}). For each triple $(X,\,\rho,\,T)$ satisfying properties~(i)~and~(ii) in Definition~\ref{def:admissible}, each Borel $Y\subseteq X$ such that there exists a unique Jacobian $J_{\mu}$ (cf.\ Proposition~\ref{prop:rokhlin'sformula}) for $T$ with respect to $\mu$ for each $\mu\in\cM(X,T)\cap\cP(X,Y)$, and each Borel function $\phi\colon X\mapping\R$, we define 
\begin{equation}\label{eq:defofspecialequilibriumstate}
    \cE_0(T,\phi;Y)\coloneqq\bigl\{\mu\in\cE(T,\phi)\cap\cP(X;Y):h_{\mu}(T)=\Functional{\mu}{\log\bigl(J_{\mu}\bigr)}\bigr\},
\end{equation}

Indeed, for a Borel subset $Y\subseteq X$ satisfying that there exists $\{Y_k\}_{k\in\N}$ such that properties~(iii)~and~(iv) in Definition~\ref{def:admissible} are satisfied, the existence and uniqueness  of the Jacobians follow from Propositions~\ref{prop:transfer operator}~(ii) and~\ref{prop:rokhlin'sformula}. Moreover, for $Y$ with a weaker assumption, the uniqueness and existence of Jacobians still hold (cf.\ Lemma~\ref{lem:admissibleweaker}).

Combining Theorem~\ref{thm:description of Jacobian} with the definition of equilibrium states yields the following result, which provides a verifiable criterion for identifying equilibrium states.

\begin{theorem}\label{thm:prescribed Jacobian of equilibrium state}		
	Let $(X,\,\rho,\,T,\,Y,\,\{Y_k\}_{k\in\N},\,\mu)$ be admissible, and $\phi\colon X\mapping\R$ be a Borel function with $\functional{\mu}{\phi}\in\R$. Suppose that $J$ is a positive Borel function on $X$ which satisfies the following properties:
	\begin{enumerate}
		\smallskip
		\item There exists a bounded Borel function $h\colon X\rightarrow\R$ such that for each $x \in Y$,
		\begin{equation*}
			J(x)=\exp(P(T,\potential)-\potential(x)+h(T(x))-h(x)).
		\end{equation*}
		\item $\sum\limits_{y\in T^{-1}(x)\cap Y}\frac{1}{J(y)}=1$ for each $x\in\bigcap_{i\in\N_0}T^i(Y)$.
	\end{enumerate}
	Then $\mu\in\cM(X,T;Y,J)$ if and only if $\mu\in\cE_0(T,\phi;Y)$.
\end{theorem}

\begin{proof}
    Since $h$ is a bounded Borel function and $\mu\in\cM(X,T)$, we have $\functional{\mu}{h}=\functional{\mu}{h\circ T}\in(-\infty,+\infty)$. Since $J$ satisfies property~(i), $\mu\in\cP(X;Y)$, and $\functional{\mu}{\phi}<+\infty$, we obtain that
	\begin{equation}\label{eq:commonpart}
		\functional{\mu}{\log(J)}=P(T,\potential)-\functional{\mu}{\potential}+\functional{\mu}{h\circ T}-\functional{\mu}{h}=P(T,\potential)-\functional{\mu}{\potential}.
	\end{equation}
    By Proposition~\ref{prop:transfer operator}~(ii), there exists a Borel function $J_{\mu}\colon X\rightarrow\R^+$ that is a Jacobian for $T$ with respect to $\mu$.

	We first establish the forward implication. Consider $\mu\in\cM(X,T;Y,J)$. 
	Since $\mu\in\cP(X;Y)$, it follows from (\ref{eq:defofcMforJonsubset}) that $J_{\mu}(x)\geq J(x)$ for $\mu$-a.e.\ $x\in X$.
	By (\ref{eq:commonpart}), (\ref{eq:Variational Principle for pressure}), and Proposition~\ref{prop:rokhlin'sformula}, we have
	\begin{equation}\label{eq:JmuequalJ}
		\functional{\mu}{\log(J)}=P(T,\phi)-\functional{\mu}{\phi}\geq h_{\mu}(T)\geq\Functional{\mu}{\log\bigl(J_{\mu}\bigr)}.
	\end{equation}
	Since $J_{\mu}(x)\geq J(x)$ for $\mu$-a.e.\ $x\in X$, the inequalities in (\ref{eq:JmuequalJ}) must be equalities, which implies $\mu\in\cE(T,\phi)$ and $h_{\mu}(T)=\Functional{\mu}{\log\bigl(J_{\mu}\bigr)}$. 
	This implies $\mu\in\cE_0(T,\phi;Y)$ (recall \eqref{eq:defofspecialequilibriumstate}).

	We now establish the backward implication. Suppose $\mu\in\cE(T,\phi)$ and $h_{\mu}(T)=\Functional{\mu}{\log\bigl(J_{\mu}\bigr)}$. By (\ref{eq:Variational Principle for pressure}) and (\ref{eq:commonpart}), we obtain $\Functional{\mu}{\log\bigl(J_{\mu}\bigr)}=h_{\mu}(T) = P(T,\potential) - \functional{\mu}{\potential}=\functional{\mu}{\log(J)}$. Since $\mu\in\cM(X,T)\cap\cP(X;Y)$, we have $1\geq\mu\bigl(T^i(Y)\bigr)\geq\mu(Y)=1$ for each $i\in\N_0$. Thus $\mu\bigl(\bigcap_{i\in\N_0}T^i(Y)\bigr)=1$. Since $J$ satisfies property~(ii), by Theorem~\ref{thm:description of Jacobian}, $J$ is a Jacobian on $Y$ for $T$ with respect to $\mu$. By (\ref{eq:defofcMforJonsubset}), we conclude that $\mu\in\cM(X,T;Y,J)$.
\end{proof}

Finally, we state the following lemma to weaken property~(iii) in Definition~\ref{def:admissible}.

\begin{lemma}\label{lem:admissibleweaker}
    Let $(X,\,\rho,\,T,\,Y,\,\{Y_k\}_{k\in\N},\mu)$ satisfy properties~(i), (ii), (iv), and~(v) of Definition~\ref{def:admissible}. Assume that $Y_k$ is admissible for $T$ for each $k\in\N$. Then there exists a sequence $\{Y_k^{\prime}\}_{k\in\N}$ of Borel subsets such that $(X,\,\rho,\,T,\,Y,\,\{Y_k^{\prime}\}_{k\in\n},\,\mu)$ is an admissible sextuple. 
\end{lemma}

\begin{proof}
    Set $Y_k^{\prime}\=Y_k\smallsetminus\bigcup_{n=1}^{k-1}Y_n$ for each $k\in\N$. Then we have $\bigcup_{k\in\N}Y_k^{\prime}=\bigcup_{k\in\N}Y_k=Y$. By Definition~\ref{def:Jacobian}, it follows from $Y_k^{\prime}\subseteq Y_k$ that $Y_k^{\prime}$ is admissible for $T$ for each $k\in\N$. Hence, by definition, $\{Y_k^{\prime}\}_{k\in\N}$ is a sequence of pairwise disjoint admissible sets for $T$, which implies that property~(iii) of Definition~\ref{def:admissible} is satisfied for the sextuple $(X,\,\rho,\,T,\,Y,\,\{Y_k^{\prime}\}_{k\in\n},\,\mu)$. Therefore, by Definition~\ref{def:admissible}, $(X,\,\rho,\,T,\,Y,\,\{Y_k^{\prime}\}_{k\in\n},\,\mu)$ is admissible.
\end{proof}

\subsection{Proof of Theorem~\ref{theorem B}}    \label{subsec:proof of theorem B}

Theorem~\ref{theorem B} follows immediately from the following theorem.

\begin{theorem}\label{thm:recursivelycompactofJacobian}
	Let $(X,\,\rho,\,\cS,\,\{X_n\}_{n\in\N},\,\{T_n\}_{n\in\N})$ be a uniformly computable system, and $Y_n$ be an open subset of $X_n$ for each $n\in\N$. Assume that there exist two recursively enumerable sets $K$ and $L$ with $L\subseteq\N\times K$, and a sequence $\{Y_{n,k}\}_{(n,k)\in L}$ of uniformly lower semi-computable open sets in $(X,\,\rho,\,\cS)$ such that $Y_{n,k}$ is admissible for $T_n$, and $Y_n=\bigcup_{(n,k)\in L_n}Y_{n,k}$, where $L_n\coloneqq\{(n,k)\in L:k\in K\}$ for each $n\in\N$. 
    
    Suppose $\{J_n\}_{n\in\N}$ is a sequence of uniformly lower semi-computable functions $J_n\colon X \rightarrow[0,+\infty)$ with respect to $\{Y_n\}_{n\in\N}$ such that $J_n$ is nonnegative on $Y_n$ and Borel for each $n\in\N$.
	Then $\{\cM(X,T_n;Y_n,J_n)\}_{n\in\N}$ is uniformly recursively compact in $(\cP(X),\,W_{\rho},\,\cQ_{\cS})$.
\end{theorem}

\begin{proof}[\bf Proof of Theorem~\ref{theorem B}]
    By Theorem~\ref{thm:recursivelycompactofJacobian}, $\{\cM(X,T_n;Y_n,J_n)\}_{n\in\N}$ is uniformly recursively compact in $(\cP(X),\,W_{\rho},\,\cQ_{\cS})$. By Proposition~\ref{prop:operatorsonrecursivelycompactsets}~(i), since $\{\cK_n\}_{n\in\N}$ is uniformly recursively compact, so is $\{\cM(X,T_n;Y_n,J_n)\cap\cK_n\}_{n\in\N}$. Note that $\cM(X,T_n;Y_n,J_n)\cap\cK_n=\{\mu_n\}$ for each $n\in\N$ by (\ref{eq:propertyofKn}). Therefore, by Proposition~\ref{thecomplementofrecursivelycompactset}~(i), $\{\mu_n\}_{n\in\N}$ is uniformly computable in $(\cP(X),\,W_{\rho},\,\cQ_{\cS})$.
\end{proof}

We devote the rest of this subsection to a proof of Theorem~\ref{thm:recursivelycompactofJacobian}.
Our proof proceeds by constructing a sequence of uniformly lower semi-computable open sets and expressing the set $\cM(X,T_n;Y_n,J_n)$ as the complement of their union. 
The uniform recursive compactness then follows from established properties of recursively compact spaces.

\begin{proof}[\bf Proof of Theorem~\ref{thm:recursivelycompactofJacobian}]
	By Proposition~\ref{prop:test function}, there exists a sequence $\{\tau_s\}_{s\in\N}$ of uniformly computable functions $\tau_s\colon X\mapping\R$ such that $\{\tau_s:s\in\N\}$ is dense in $C(X)$. By the computability of the absolute value function, the sequence $\bigl\{\tau_s^+\bigr\}_{s\in\N}$ is a sequence of uniformly computable functions. Since $\{Y_{n,k}\}_{(n,k)\in L}$ is uniformly lower semi-computable open, by Proposition~\ref{prop:characterfunctionislowersemicomputable}, there exists a sequence $\{h_{m,n,k}\}_{(m,(n,k))\in\N\times L}$ of uniformly computable functions $h_{m,n,k}\colon X\rightarrow\R$ such that for each $(n,k)\in L$, the following properties are satisfied:
	\begin{enumerate}
		\smallskip

		\item For each $x\in X$, $\{h_{m,n,k}(x)\}_{m\in\N}$ is nondecreasing and $h_{m,n,k}(x)\to\mathbbm{1}_{Y_{n,k}}(x)$ as $m\to+\infty$.

		\smallskip

		\item For each $m\in\N$, $h_{m,n,k}(x)\geq 0$ for each $x\in X$ and $h_{m,n,k}(x)=0$ for each $x\notin Y_{n,k}$.
	\end{enumerate}
	We define for all $m,s\in\N,\,x\in X$, and $(n,k)\in L$,
	\begin{align}
		W_{m,s}^{n,k}(x)&\=\begin{cases}
		    \sup\bigl\{\tau^+_s(y)\cdot h_{m,n,k}(y):y\in T_n^{-1}(x)\bigr\}&\textrm{if }x\in T_n(Y_{n,k});\\
            0&\textrm{otherwise},
		\end{cases}\label{eq:def of W}\\ 
        V_{m,s}^{n,k}(x)&\coloneqq J_{n}(x)\cdot\tau^+_s(x)\cdot h_{m,n,k}(x),\qquad\qquad\qquad\,\,\, \text{and}\label{eq:def of V}\\
        \Psi_{m,s}^{n,k}&\coloneqq\bigl\{\mu\in\PPP(X):\Functional{\mu}{W_{m,s}^{n,k}-V_{m,s}^{n,k}}<0\bigr\}\label{eq:def of Psi}.
	\end{align}
	We first establish the following claim.

	\smallskip
	\emph{Claim~1.}
	$\bigl\{\Psi_{m,s}^{n,k}\big\}_{(m,s,(n,k))\in\N^2\times L}$ is uniformly lower semi-computable open in $(\cP(X),\,W_{\rho},\,\cQ_{\cS})$.
	\smallskip

	\emph{Proof of Claim~1.} Let $\cS_{\Q}=\{q_v\}_{v\in\N}$ (see Subsection~\ref{sub:Computable Analysis}). First, we show that $\bigl\{W_{m,s}^{n,k}\bigr\}_{(m,s,(n,k))\in\N^2\times L}$ is a sequence of uniformly upper semi-computable functions. Denote $Q_v\=(-\infty,q_v)$ for each $v\in\N$. Indeed, it is not hard to derive from (\ref{eq:def of W}) and property~(ii) of $\{h_{m,n,k}\}_{(m,(n,k))\in\N\times L}$ that 
    \begin{align}\label{eq:formulaforW}
        \bigl(W_{m,s}^{n,k}\bigr)^{-1}(Q_v^c)&=\begin{cases}
		    X&\textrm{if }q_v\leq 0;\\
            T_n\bigl(\bigl(\tau_s^+\cdot h_{m,n,k}\bigr)^{-1}(Q_v^c)\bigr)&\textrm{otherwise}
		\end{cases}
    \end{align}
    for all $m,\,s,\,v\in\N$, and $(n,k)\in L$.  
    
    By the uniform computability of $\bigl\{\tau_s^+\bigr\}_{s\in\N}$ and $\{h_{m,n,k}\}_{(m,(n,k))\in\N\times L}$, Proposition~\ref{prop:computabilityoffunctionspreservedunderoperators} implies that $\bigl\{\tau_s^+\cdot h_{m,n,k}\bigr\}_{(m,s,(n,k))\in\N^2\times L}$ is a sequence of uniformly computable functions. Since $\{Q_v\}_{v\in\N}$ is uniformly lower semi-computable open in $\bigl(\R,\,d_{\R},\,\cS_{\Q}\bigr)$, by Proposition~\ref{prop:computable function}, the sequence $\bigl\{\bigl(\tau_s^+\cdot h_{m,n,k}\bigr)^{-1}(Q_v)\bigr\}_{(m,s,v,(n,k))\in\N^3\times L}$ is uniformly lower semi-computable open in $(X,\,\rho,\,\cS)$. Note that $\bigl(\tau_s^+\cdot h_{m,n,k}\bigr)^{-1}(Q_v^c)=\bigl(\bigl(\tau_s^+\cdot h_{m,n,k}\bigr)^{-1}(Q_v)\bigr)^c$ for all $m,\,s,\,v\in\N$, and $(n,k)\in L$. Then since $X$ is recursively compact in $(X,\,
    \rho,\,\cS)$, by Proposition~\ref{thecomplementofrecursivelycompactset}~(iii), $\bigl\{\bigl(\tau_s^+\cdot h_{m,n,k}\bigr)^{-1}(Q_v^c)\bigr\}_{(m,s,v,(n,k))\in\N^3\times L}$ is uniformly recursively compact. Note that by property~(ii) of $\{h_{m,n,k}\}_{(m,(n,k))\in\N\times L}$ and the hypotheses of Theorem~\ref{thm:recursivelycompactofJacobian}, we have that $\bigl(\tau_s^+\cdot h_{m,n,k}\bigr)^{-1}(Q_v^c)\subseteq Y_{n,k}\subseteq X_n$ for all $m,\,s,\,v\in\N$, and $(n,k)\in L$. Hence, since $\{T_n\}_{n\in\N}$ is a sequence of uniformly computable functions with respect to $\{X_n\}_{n\in\N}$, by Proposition~\ref{thecomplementofrecursivelycompactset}~(v), we obtain that $\bigl\{T_n\bigl(\bigl(\tau_s^+\cdot h_{m,n,k}\bigr)^{-1}(Q_v^c)\bigr)\bigr\}_{(m,s,v,(n,k))\in\N^3\times L}$ is uniformly recursively compact. Since $X$ is recursively compact, and $\cS_{\Q}=\{q_v\}_{v\in\N}$, by (\ref{eq:formulaforW}), $\bigl\{\bigl(W_{m,s}^{n,k}\bigr)^{-1}(Q_v^c)\bigr\}_{(m,s,v,(n,k))\in\N^3\times L}$ is uniformly recursively compact. Thus, since $\bigl(W_{m,s}^{n,k}\bigr)^{-1}(Q_v^c)=\bigl(\bigl(W_{m,s}^{n,k}\bigr)^{-1}(Q_v)\bigr)^c$ for all $m,\,s,\,v\in\N$, and $(n,k)\in L$, by Proposition~\ref{thecomplementofrecursivelycompactset}~(ii), the sequence $\bigl\{\bigl(W_{m,s}^{n,k}\bigr)^{-1}(Q_v)\bigr\}_{(m,s,v,(n,k))\in\N^3\times L}$ is uniformly lower semi-computable open. Hence, it follows from the implication from (ii) to (i) in Proposition~\ref{prop:semicomputability} that $\bigl\{W_{m,s}^{n,k}\bigr\}_{(m,s,(n,k))\in\N^2\times L}$ is a sequence of uniformly upper semi-computable functions. 
	 
	We next show that $\bigl\{V_{m,s}^{n,k}\bigr\}_{(m,s,(n,k))\in\N^2\times L}$ is a sequence of uniformly lower semi-computable functions. Denote $R_v\coloneqq (q_v,+\infty)$ for each $v\in\N$. Since $\bigl\{\tau_s^+\bigr\}_{s\in\N}$ and $\{h_{m,n,k}\}_{(m,(n,k))\in\N\times L}$ are sequences of uniformly computable nonnegative functions, and $\{J_n\}_{n\in\N}$ is a sequence of uniformly lower semi-computable functions with respect to $\{Y_n\}_{n\in\N}$, by (\ref{eq:def of V}), Definition~\ref{def:domain of lower or upper computability}, and Proposition~\ref{prop:computabilityoffunctionspreservedunderoperators}, $\bigl\{V_{m,s}^{n,k}\bigr\}_{(m,s,(n,k))\in\N^2\times L}$ is a sequence of uniformly lower semi-computable functions with respect to $\{Y_n\}_{n\in\N}$. Hence, by Proposition~\ref{prop:semicomputability}, there exists a sequence $\bigl\{Y_{m,s,v}^{n,k}\bigr\}_{(m,s,v,(n,k))\in\N^3\times L}$ of uniformly lower semi-computable open sets such that $\bigl(V_{m,s}^{n,k}\bigr)^{-1}(R_v)\cap Y_n=Y_{m,s,v}^{n,k}\cap Y_n$ for all $m,\,s,\,v\in\N$ and $(n,k)\in L$. Indeed, it is not hard to derive from (\ref{eq:def of V}) and property~(ii) of $\{h_{m,n,k}\}_{(m,(n,k))\in\N\times L}$ that 
    \begin{align}\label{eq:formulaforV}
        \bigl(V_{m,s}^{n,k}\bigr)^{-1}(R_v)&=\begin{cases}
		    X&\textrm{if }q_v<0;\\
            Y_{m,s,v}^{n,k}\cap Y_n&\textrm{otherwise}
		\end{cases}
        \quad\text{ for all }m,\,s,\,v\in\N,\text{ and }(n,k)\in L.
    \end{align}
    
    Since $\{Y_n\}_{n\in\N}$ and $\bigl\{Y_{m,s,v}^{n,k}\bigr\}_{(m,s,v,(n,k))\in\N^3\times L}$ are both uniformly lower semi-computable open in $(X,\,\rho,\,\cS)$, by Proposition~\ref{thecomplementofrecursivelycompactset}~(ii),~(iii), and Proposition~\ref{prop:operatorsonrecursivelycompactsets}~(ii), it follows from the recursive compactness of $X$ that $\bigl\{Y_{m,s,v}^{n,k}\cap Y_n\bigr\}_{(m,s,v,(n,k))\in\N^3\times L}$ is uniformly lower semi-computable open. Since $\cS_{\Q}=\{q_v\}_{v\in\N}$, by (\ref{eq:formulaforV}), $\bigl\{\bigl(V_{m,s}^{n,k}\bigr)^{-1}(R_v)\bigr\}_{(m,s,v,(n,k))\in\N^3\times L}$ is uniformly lower semi-computable open.
    Thus, by Proposition~\ref{prop:semicomputability}, we conclude that $\bigl\{V_{m,s}^{n,k}\bigr\}_{(m,s,(n,k))\in\N^2\times L}$ is a sequence of uniformly lower semi-computable functions.
    
    Since $\bigl\{W_{m,s}^{n,k}\bigr\}_{(m,s,(n,k))\in\N^2\times L}$ (resp.\ $\bigl\{V_{m,s}^{n,k}\bigr\}_{(m,s,(n,k))\in\N^2\times L}$) is a sequence of uniformly upper (resp.\ lower) semi-computable functions, by Definition~\ref{def:domain of lower or upper computability}, $\bigl\{W_{m,s}^{n,k}-V_{m,s}^{n,k}\bigr\}_{(m,s,(n,k))\in\N^2\times L}$ is a sequence of uniformly upper semi-computable functions. Thus, by Corollary~\ref{cor:semicomputability of integral operator}, (\ref{eq:def of Psi}), and Proposition~\ref{prop:semicomputability}, we conclude that $\bigl\{\Psi_{m,s}^{n,k} \bigr\}_{(m,s,n,k)\in\N^2\times L}$ is uniformly lower semi-computable open in $(\cP(X),\,W_{\rho},\,\cQ_{\cS})$. This establishes Claim~1.
	\smallskip

	We now characterize the set $\cM(X,T_n;Y_n,J_n)$ (recall \eqref{eq:defofcMforJonsubset}).

	\smallskip
	\emph{Claim~2.}
	$\cM(X,T_n;Y_n,J_n)=\cP(X)\smallsetminus\bigcup_{(m,s)\in\N^2}\bigcup_{(n,k)\in L_n}\Psi_{m,s}^{n,k}$ for each $n\in\N$.
	\smallskip

	\emph{Proof of Claim~2.} We fix an arbitrary $n\in\N$ and establish our claim by showing that these two sets are mutually inclusive for $n$.
    
	Suppose $\mu\in \cP(X)\smallsetminus\bigcup_{(m,s)\in\N^2}\bigcup_{(n,k)\in L_n}\Psi_{m,s}^{n,k}$.
	By (\ref{eq:def of Psi}), we have $\Functional{\mu}{W_{m,s}^{n,k}}\geq\Functional{\mu}{V_{m,s}^{n,k}}$ for all $m,\,s\in\N$, and $(n,k)\in L_n$. 
    Consider an arbitrary $(n,k)\in L_n$.
    Since $Y_{n,k}$ is admissible for $T_n$, by \cite[Corollary~15.2]{kechris2012classical}, $T_n$ is a Borel isomorphism of $Y_{n,k}$ with $T_n(Y_{n,k})$.  
    It follows from (\ref{eq:def of W}) and (\ref{eq:def of V}) that
	\begin{equation}		\label{eq:integration with varphik inequality}
		\int_{T_n(Y_{n,k})}\!\bigl(\tau_s^+\cdot h_{m,n,k}\bigr)\circ\bigl(T_n|_{Y_{n,k}} \bigr)^{-1}\,\mathrm{d}\mu
		\geq\Functional{\mu}{J_n\cdot\tau_s^+\cdot h_{m,n,k}}\quad\text{ for all }m,\,s\in\N.
	\end{equation}
	Letting $m\to+\infty$ in (\ref{eq:integration with varphik inequality}) and applying the monotone convergence theorem, we derive from properties~(i)~and~(ii) of $\{h_{m,n,k}\}_{(m,(n,k))\in\N\times L}$ that
	\begin{equation}\label{eq:integration with varphik}
		\int_{T_n(Y_{n,k})}\!\tau_s^+\circ \bigl(T_n|_{Y_{n,k}}\bigr)^{-1} \,\mathrm{d}\mu
		\geq\int_{Y_{n,k}}\!\bigl(J_n\cdot \tau_s^+\bigr)\,\mathrm{d}\mu\quad\text{ for each }s\in\N.
	\end{equation}
    For each $A\in\cB(X)$, we define $\mu_{1,n,k}(A)\=\mu(T_n(A\cap Y_{n,k}))$ and $\mu_{2,n,k}(A) \= \int_{A\cap Y_{n,k}}\!J_n\,\mathrm{d}\mu$. Then $\mu_{1,n,k}$ and $\mu_{2,n,k}$ are finite Borel measures on $X$. By the change-of-variable formula and (\ref{eq:integration with varphik}), we have $\Functional{\mu_{1,n,k}}{\tau_s^+}\geq\Functional{\mu_{2,n,k}}{\tau_s^+}$ for each $s\in\N$. By Proposition~\ref{prop:test function}, for each $A\in\cB(X)$, $\mu(T_n(A\cap Y_{n,k}))\geq\int_{A\cap Y_{n,k}}\!J_n\,\mathrm{d}\mu$.
    Thus, since $Y_n=\bigcup_{(n,k)\in L_n}Y_{n,k}$, we obtain that $\mu(T_n(A))\geq\int_A\!J_n\,\mathrm{d}\mu$ for each admissible set $A\subseteq Y_n$ for $T_n$. By (\ref{eq:defofcMforJonsubset}), this implies that $\mu\in\cM(X,T_n;Y_n,J_n)$.

	Conversely, consider an arbitrary $\mu\in M(X,T_n;Y_n,J_n)$. Note that $T_n$ is a Borel isomorphism of $Y_{n,k}$ with $T_n(Y_{n,k})$ for each $(n,k)\in L$. Then by (\ref{eq:def of W}), (\ref{eq:def of V}), the change-of-variable formula, and property~(ii) of $\{h_{m,n,k}\}_{(m,(n,k))\in\N\times L}$, we obtain that
	\begin{equation*}
		\begin{aligned}
        \int\! W_{m,s}^{n,k}\,\mathrm{d}\mu
        &\geq \int_{T_n(Y_{n,k})}\!\bigl(\tau_s^+\cdot h_{m,n,k}\bigr) \circ\bigl(T_n|_{Y_{n,k}} \bigr)^{-1}\,\mathrm{d}\mu
		= \int_{Y_{n,k}}\!\bigl(\tau_s^+\cdot h_{m,n,k}\bigr)\,\mathrm{d}\bigl(\mu\circ\bigl(T_n|_{Y_{n,k}}\bigr)\bigr)\\
		&\geq\int_{Y_{n,k}}\!\bigl(J_n\cdot\tau_s^+\cdot h_{m,n,k}\bigr)\,\mathrm{d}\mu=\int\!\bigl(J_n\cdot\tau_s^+\cdot h_{m,n,k}\bigr)\,\mathrm{d}\mu=\int\!V_{m,s}^{n,k}\,\mathrm{d}\mu.
		\end{aligned}
	\end{equation*}
	for all $m,\,s\in\N$ and $(n,k)\in L_n$. 
	By (\ref{eq:def of Psi}), this implies that $\mu\notin\Psi_{m,s}^{n,k}$, establishing Claim~2.

	\smallskip
	
	By Proposition~\ref{prop:lower semi-computable open} and Claim~1, $\bigl\{\bigcup_{(m,s)\in\N^2}\bigcup_{(n,k)\in L_n}\Psi_{m,s}^{n,k}\bigr\}_{n\in\N}$ is uniformly lower semi-computable open in $(\cP(X),\,W_{\rho},\,\cQ_{\cS})$. Since $X$ is recursively compact in $(X,\,\rho,\,\cS)$, by Proposition~\ref{prop:recursively compactness of measure space}, we have that $\cP(X)$ is recursively compact in $(\cP(X),\,W_{\rho},\,\cQ_{\cS})$. By Proposition~\ref{thecomplementofrecursivelycompactset}~(iii) and Claim~2, we conclude that $\{\cM(X,T_n;Y_n,J_n)\}_{n\in\N}$ is uniformly recursively compact in $(\cP(X),\,W_{\rho},\,\cQ_{\cS})$.
\end{proof}

\subsection{Applications of Theorem~\ref{theorem B}}
\label{sub:applications}

In this subsection, we demonstrate how to apply Theorem~\ref{theorem B} to different systems by constructing appropriate sequences $\{\cK_n\}_{n\in\N}$, with the main applications being Theorems~\ref{thm:additionalassumptionone}~and~\ref{thm:additionalassumptiontwo}.


 We state the following hypotheses under which we establish our results.

\begin{asss}
    \quad
    \begin{enumerate}
        \smallskip
        \item Let $(X,\,\rho,\,\cS,\,\{X_n\}_{n\in\N},\,\{T_n\}_{n\in\N})$ be a uniformly computable system. Let $\{Y_n\}_{n\in\N}$ be a sequence of open sets and $\{J_n\}_{n\in\N}$ be a sequence of Borel functions that satisfy the hypotheses of Theorem~\ref{thm:recursivelycompactofJacobian}.
        \smallskip
        \item Let $\phi_n\colon X\mapping\R$ be a bounded Borel function for each $n\in\N$.
        \smallskip
        \item For each $n\in\N$, properties~(i)~and~(ii) in Theorem~\ref{thm:prescribed Jacobian of equilibrium state} are satisfied in the case where $T\coloneqq T_n$, $Y\coloneqq Y_n$, $J\coloneqq J_n$, and $\phi\coloneqq\phi_n$. 
        \smallskip
        \item $\{C_n\}_{n\in\N}$ is a sequence of uniformly recursively compact subsets in $(X,\,\rho,\,\cS)$ such that $\card{C_n\smallsetminus Y_n}<+\infty$ and $\cE_0(T_n,\phi_n;Y_n)\cap\cP(X;C_n\cap Y_n)=\{\mu_n\}$ for each $n\in\N$.\footnote{For the definition of the set $\cE_0(T_n,\phi_n;Y_n)$, see (\ref{eq:defofspecialequilibriumstate}); its existence and uniqueness are discussed immediately thereafter.}
        \smallskip
        \item $H$ and $I$ are nonempty recursively enumerable sets with $H\subseteq\N\times I$. Moreover, $\{p_{n,i}\}_{(n,i)\in H}$ is a sequence of uniformly computable points such that $\{p_{n,i}:(n,i)\in H_n\}$ is the union of the set $C_n\smallsetminus X_n$ and the set $B_n$ of periodic points of $T_n$ in $C_n\cap(X_n\smallsetminus Y_n)$ whose orbits for $T_n$ are contained in $C_n\cap X_n$, where $H_n\coloneqq\{(n,i)\in H:i\in I\}$ for each $n\in\N$.
    \end{enumerate}
\end{asss}

We need the following lemma, whose proof is verbatim the same as that of \cite[Lemma~6.3~(iii)]{shi2024entropy}.

\begin{lemma}\label{lem:positivesupportonapoint}
	Let $(X,\rho)$ be a compact metric space, $T\colon X\rightarrow X$ be a Borel-measurable transformation, and $\mu\in\cM(X,T)$. Suppose $x \in X$ satisfies $\mu(\{x\})>0$. Then $x$ is a periodic point of $T$. If we assume in addition that $\mu$ is ergodic, then $\mu=\frac{1}{n}\sum_{i=0}^{n-1}\delta_{T^i(x)}$, where $n$ is the period of $x$.
\end{lemma}

We now consider dynamical systems satisfying the hypotheses of Theorem~\ref{theorem B} and the Assumptions. 
Before detailing the technical statements, we emphasize their applicability to expanding Thurston maps and rational maps. 
Specifically, Theorem~\ref{thm:additionalassumptionone} can be applied to establish the computability of equilibrium states for expanding Thurston maps with \holder continuous potentials. As a corollary, Theorem~\ref{thm:additionalassumptiontwo} directly implies the computability of equilibrium states for rational maps with hyperbolic \holder potentials.

However, a full treatment of general potentials requires verifying the computability of the associated Jacobians, a task that relies on independent techniques such as the cone method introduced in \cite[Subsection~4.1]{binder2025computability}. 
To demonstrate the efficacy of our approach while avoiding the extensive technicalities associated with these estimates, we restrict our explicit application in this article to expanding Thurston maps with the zero potential (see Theorem~\ref{Application B}).

\begin{theorem}\label{thm:additionalassumptionone}
    Under the Assumptions, assume, in addition, that the following statement is true:  
    
    There exists a sequence $\{U_{n,i,k}\}_{((n,i),k)\in H\times\N}$ of uniformly lower semi-computable open sets in $(X,\,\rho,\,\cS)$ such that $p_{n,i}\in U_{n,i,k}$ for all $(n,i)\in H$ and $k\in\n$. Moreover, there exists a sequence $\{t_{n,i,k}\}_{((n,i),k)\in H\times\N}$ of uniformly computable real numbers such that for each $(n,i)\in H$, 
        \begin{equation}\label{eq:additionalassumptionone}
        \inf_{k\in\N}t_{n,i,k}=0,\quad\text{ and}\quad\mu_n(U_{n,i,k})\leq t_{n,i,k}\quad\text{ for each }k\in\N.
        \end{equation}
    Then $\{\mu_n\}_{n\in\N}$ is uniformly computable in $(\cP(X),\,W_{\rho},\,\cQ_{\cS})$.
\end{theorem}

\begin{proof}
    Define for each $n\in\N$,
    \begin{equation}\label{eq:defofKn}
    \cK_n\coloneqq\{\mu\in\cP(X;C_n)\cap\cM(X,T_n;X_n):\mu(U_{n,i,k})\leq t_{n,i,k}\text{ for all }k\in\N\text{ and }(n,i)\in H_n\}.
    \end{equation}
    
    First, we establish the uniformly recursive compactness of $\{\cK_n\}_{n\in\N}$ as follows.

    \smallskip
    \emph{Claim~1.} $\{\cK_n\}_{n\in\N}$ is uniformly recursively compact in $(\cP(X),\,W_{\rho},\,\cQ_{\cS})$.
    \smallskip

    \emph{Proof of Claim~1.}
    Since $(X,\,\rho,\,\cS,\,\{X_n\}_{n\in\N},\,\{T_n\}_{n\in\N})$ satisfy the Assumptions, in $(X,\,\rho,\,\cS)$, $X$ is recursively compact, $\{X_n\}_{n\in\N}$ is uniformly lower semi-computable open, and $\{T_n\}_{n\in\N}$ is a sequence of uniformly computable functions with respect to $\{X_n\}_{n\in\N}$. Thus, by Proposition~\ref{prop:recursivecompactofsetinvariant}, $\{\cM(X,T_n;X_n)\}_{n\in\N}$ is uniformly recursively compact in $(\cP(X),\,W_{\rho},\,\cQ_{\cS})$.
    
    Moreover, since in $(X,\,\rho,\,\cS)$, $X$ is recursively compact, and $\{C_n\}_{n\in\N}$ is uniformly recursively compact, by Proposition~\ref{thecomplementofrecursivelycompactset}~(ii), $\{X\smallsetminus C_n\}_{n\in\N}$ is uniformly lower semi-computable open. Note that by definition, $\cP(X;C_n)=\{\mu\in\cP(X):\mu(X\smallsetminus C_n)\leq 0\}$ for each $n\in\N$. By Corollary~\ref{cor:seminormrecursivelycompact}, $\{\cP(X;C_n)\}_{n\in\N}$ is uniformly recursively compact in $(\cP(X),\,W_{\rho},\,\cQ_{\cS})$.

    Since in $(X,\,\rho,\,\cS)$, $X$ is recursively compact, $\{U_{n,i,k}\}_{((n,i),k)\in H\times\N}$ is uniformly lower semi-computable open, and $\{t_{n,i,k}\}_{((n,i),k)\in H\times\N}$ is a sequence of uniformly computable real numbers, by Corollary~\ref{cor:seminormrecursivelycompact}, $\{\{\mu\in\cP(X):\mu(U_{n,i,k})\leq t_{n,i,k}\}\}_{((n,i),k)\in H\times\N}$ is uniformly recursively compact in $(\cP(X),\,W_{\rho},\,\cQ_{\cS})$. Combined with the uniformly recursive compactness of $\{\cP(X;C_n)\}_{n\in\N}$ and $\{\cM(X,T_n;X_n)\}_{n\in\N}$, by Proposition~\ref{prop:operatorsonrecursivelycompactsets} and (\ref{eq:defofKn}), this implies that $\{\cK_n\}_{n\in\n}$ is uniformly recursively compact.

    Next, we establish the following claim.

    \smallskip
    \emph{Claim~2.} $\cM(X,T_n;Y_n,J_n)\cap\cK_n=\{\mu_n\}$ for each $n\in\N$.
    \smallskip

    \emph{Proof of Claim~2.}
    We consider an arbitrary $n\in\N$ and prove our claim for $n$.
    
    First, we show that 
    \begin{equation}\label{eq:lemmaforclaim2}
    \cM(X,T_n;Y_n,J_n)\cap\cM(X,T_n;X_n)\cap\PPP(X;C_n\cap Y_n)=\{\mu_n\}.
    \end{equation} 
    Since $\{Y_n\}_{n\in\N}$ satisfies the hypotheses of Theorem~\ref{theorem B}, there exist two recursively enumerable sets $K,\,L$ with $L\subseteq\N\times K$, and a sequence $\{Y_{n,k}\}_{(n,k)\in L}$ of open sets such that $Y_{n,k}$ is admissible for $T_n$ for each $(n,k)\in L$. Since $L$ is a recursively enumerable set, there exists a function $f\colon\N\mapping L_n$ with $f(\N)=L_n$. Define $Y_m^{\prime}\coloneqq Y_{f(m)}\smallsetminus\bigcup_{k=1}^{m-1}Y_{f(k)}$ for each $m\in\N$. Then by Definition~\ref{def:admissible}, $(X,\,\rho,\,T_n,\,Y_n,\,\{Y^{\prime}_{m}\}_{m\in\N},\,\mu)$ is admissible. Note that $\phi_n$ is a bounded Borel function by item~(ii) in the Assumptions. Then we have $\functional{\mu}{\phi_n}\in\R$ for each $\mu\in\cP(X)$ and each $n\in\N$. Hence, by item~(iii) in the Assumptions and Theorem~\ref{thm:prescribed Jacobian of equilibrium state}, we obtain that $\cM(X,T_n;Y_n,J_n)\cap\cM(X,T_n)\cap\PPP(X;C_n\cap Y_n)=\cE_0(T_n,\phi_n;Y_n)\cap\PPP(X;C_n\cap Y_n)$.
    Note that $Y_n\subseteq X_n$ by item~(ii) in the Assumptions. Then by (\ref{eq:defofcMonsubset}), we have $\cM(X,T_n;X_n)\cap\cP(X;C_n\cap Y_n)=\cM(X,T_n)\cap\cP(X;C_n\cap Y_n)$. Thus, by item~(iv) in the Assumptions, we obtain that $\cM(X,T_n;Y_n,J_n)\cap\cM(X,T_n;X_n)\cap\PPP(X;C_n\cap Y_n)=\cM(X,T_n;Y_n,J_n)\cap\cM(X,T_n)\cap\PPP(X;C_n\cap Y_n)=\cE_0(T_n,\phi_n;Y_n)\cap\PPP(X;C_n\cap Y_n)=\{\mu_n\}$.

    Now we prove our claim by showing that these two sets are mutually inclusive. Indeed, by (\ref{eq:additionalassumptionone}) and (\ref{eq:defofKn}), $\mu_n\in\cK_{n}$. Hence, by (\ref{eq:lemmaforclaim2}), we obtain that $\mu_n\in\cM(X,T_n;Y_n,J_n)\cap\cK_n$.   
    
    Finally, we show that $\cM(X,T_n;Y_n,J_n)\cap\cK_n\subseteq\{\mu_n\}$. We argue this for $n$ by contradiction and consider a measure $\mu\in\cM(X,T_n;Y_n,J_n)\cap\cK_n$ with $\mu\neq\mu_n$. Then it is not hard to derive from (\ref{eq:lemmaforclaim2}) that $\mu\in\cP(X;C_n)\smallsetminus\PPP(X;C_n\cap Y_n)$. Note that $\card{C_n\smallsetminus Y_n}<+\infty$ by item~(iv) in the Assumptions. Thus, since $\mu\in\cP(X;C_n)\smallsetminus\PPP(X;C_n\cap Y_n)$, we have $\mu(\{x_0\})>0$ for some $x_0\in C_n\smallsetminus Y_n$.
    Then we show that $\mu\in\cP(X;X_n)\cap \cM(X,T_n)$. By the assumptions in Theorem~\ref{thm:additionalassumptionone}, $p_{n,i}\in U_{n,i,k}$ for all $(n,i)\in H_n$ and $k\in\n$. Hence, by $\mu\in\cK_n$, it follows from (\ref{eq:defofKn}) that $0\leq \mu(\{p_{n,i}\})\leq \mu(U_{n,i,k})\leq t_{n,i,k}$ for all $(n,i)\in H_n$ and $k\in\n$. 
    Note that $\inf_{k\in\N}t_{n,i,k}=0$ for each $(n,i)\in H_n$ by (\ref{eq:additionalassumptionone}). 
    Then we have $\mu(\{p_{n,i}\})=0$ for each $(n,i)\in H_n$. 
    Hence, by item~(v) in the Assumptions, we obtain that $\mu(C_n\smallsetminus X_n)=\mu(B_n)=0$. 
    Note that $\mu\in\cP(X;C_n)\cap\cM(X,T_n;X_n)$ by $\mu\in\cK_n$ and (\ref{eq:defofKn}). 
    Then by (\ref{eq:defofcMonsubset}), we have $\mu\in\cP(X;X_n)\cap\cM(X,T_n;X_n)=\cP(X;X_n)\cap\cM(X,T_n)$. Hence, by Lemma~\ref{lem:positivesupportonapoint}, it follows from $\mu(\{x_0\})>0$ that $x_0$ is a periodic point of $T_n$ in $C_n\cap(X_n\smallsetminus Y_n)$. Moreover, by $\mu\in\cP(X;C_n\cap X_n)$, the orbit of $x_0$ for $T_n$ is contained in $C_n\cap X_n$. 
    Thus, $x_0\in B_n$ by the definition of $B_n$. 
    This implies that $\mu(B_n)\geq\mu(\{x_0\})>0$, which contradicts with $\mu(B_n)=0$.
    Hence, we obtain that $\cM(X,T_n;Y_n,J_n)\cap\cK_n\subseteq\{\mu_n\}$. This completes the proof of Claim~2.
    \smallskip
    
    By Claims~1~and~2, $\{\cK_n\}_{n\in\N}$ is uniformly recursively compact and satisfies (\ref{eq:propertyofKn}). 
    Hence, it follows from Theorem~\ref{theorem B} that $\{\mu_n\}_{n\in\N}$ is uniformly computable in $(\cP(X),\,W_{\rho},\,\cQ_{\cS})$.
\end{proof}

We now consider dynamical systems satisfying the hypotheses of Theorem~\ref{theorem B} and the Assumptions whose maps $T_n$ are uniformly contracting near all periodic orbits in $B_n$ for $T_n$. 

\begin{theorem}
\label{thm:additionalassumptiontwo}
    Under the Assumptions, assume, in addition, that the following statements are true:
    \begin{enumerate}
        \smallskip
        \item $C_n\subseteq X_n$ for each $n\in\N$.
        \smallskip
        \item There exists a computable function $m\colon H\mapping\N$ and two sequences $\{r_{n,i}\}_{(n,i)\in H}$ and $\{\lambda_{n,i}\}_{(n,i)\in H}$ of uniformly computable real numbers such that for each $(n,i)\in H$, $m(n,i)$ is the period of $p_{n,i}$ for $T_n$, $r_{n,i}>0$, $0<\lambda_{n,i}<1$, and 
        \begin{equation*}
            \rho\bigl(p_{n,i},T_n^{m(n,i)}(q)\bigr)\leq\lambda_{n,i}\cdot\rho(p_{n,i},q)\quad\text{ for each }q\in B(p_{n,i},r_{n,i}).
        \end{equation*}
    \end{enumerate}
    Then $\{\mu_n\}_{n\in\N}$ is a sequence of uniformly computable measures.
\end{theorem}

\begin{proof}
    We establish that $\mu_n(B(p_{n,i},r_{n,i}))=0$ for each $(n,i)\in H$. We consider an arbitrary pair $(n,i)\in H$ and define $B(l)\=T_n^{lm(n,i)}(B(p_{n,i},r_{n,i}))$ for each $l\in\N_0$. 
    
    By item~(iv) in the Assumptions, it follows from $C_n\subseteq X_n$ that $p_{n,i}\in B_n\subseteq C_n\cap (X_n\smallsetminus Y_n)=C_n\smallsetminus Y_n$. Note that $\mu_n\in\PPP(X;C_n\cap Y_n)$ by item~(iv) in the Assumptions. Then we have $\mu_n(\{p_{n,i}\})=0$. Moreover, by the definition of $\{B(l)\}_{l\in\N_0}$, it follows from $\mu_n\in\cE_0(T_n,\phi_n;Y_n)\subseteq\cM(X,T_n)$ that $\mu_n(B(0))\leq\mu_n(B(l))$ for each $l\in\N_0$. 
    Since $\lambda_{n,i}\in(0,1)$, we have $B(l)\subseteq B\bigl(p_{n,i},r_{n,i}\lambda_{n,i}^l\bigr)\subseteq B(0)$ for each $l\in\N_0$. Thus, we have $\mu_n(B(0))=\mu_n\bigl(B\bigl(p_{n,i},r_{n,i}\lambda_{n,i}^l\bigr)\bigr)=\mu_n(B(l))$, hence, $\mu_n\bigl(B(0)\smallsetminus B\bigl(p_{n,i},r_{n,i}\lambda_{n,i}^l\bigr)\bigr) = 0$. Letting $l\rightarrow+\infty$, we obtain that $\mu_n(B(0)\smallsetminus\{p_{n,i}\})=0$. Combined with $\mu_n(\{p_{n,i}\})=0$, this proves that $\mu_n(B(p_{n,i},r_{n,i}))=0$.

    By Proposition~\ref{prop:lowersemicomputableraiusball}, it follows from the uniform computability of $\{r_{n,i}\}_{(n,i)\in H}$ and $\{p_{n,i}\}_{(n,i)\in H}$ that $\{B(p_{n,i},r_{n,i})\}_{(n,i)\in H}$ is a sequence of uniformly lower semi-computable open sets in $(X,\,\rho,\,\cS)$.
	
    Therefore, by the above result, the hypotheses of Theorem~\ref{thm:additionalassumptionone} are satisfied with $U_{n,i,k}=B(p_{n,i},r_{n,i})$ and $t_{n,i,k}=0$ for each $(n,i)\in H$ and each $k\in\N$. Hence, it follows from Theorem~\ref{thm:additionalassumptionone} that $\{\mu_n\}_{n\in\N}$ is uniformly computable in $(\cP(X),\,W_{\rho},\,\cQ_{\cS})$. 
\end{proof}

\section{Computability of equilibrium states for expanding Thurston maps}
\label{sec:Computable Analysis on thermodynamic formalism}

In this section, we consider a class of nonuniformly expanding maps on the topological $2$-sphere known as \emph{expanding Thurston maps}. There has been active research on both thermodynamic formalism for expanding Thurston maps (see e.g.\ \cite{bonk2010expanding,bonk2017expanding, haissinsky2009coarse, li2018equilibrium,li2015weak,li2017ergodic,shi2024entropy}) and algorithmic aspects of these maps (see e.g.\ \cite{selinger2015constructive,rafi2020centralizers}).

We first review the definition of expanding Thurston maps, along with some key concepts and results.
Then in Subsection~\ref{subsec:Misiurewicz--Thurston rational maps}, we focus on Misiurewicz--Thurston rational maps and apply Approach~I (see Theorem~\ref{theorem A}) to establish Theorem~\ref{Application A}. 
Finally, in Subsection~\ref{subsec:The measures of maximal entropy for computable expanding Thurston maps}, we study computable expanding Thurston maps whose critical points are computable. By Theorem~\ref{thm:upper semicontinuous iff no periodic critical points}, their measure-theoretic entropy functions may not be upper semicontinuous, which prevents us from using Approach~I (see Theorem~\ref{theorem A}) to demonstrate the computability of their corresponding equilibrium states. Instead, we apply Approach~II (see Theorem~\ref{theorem B}) to prove Theorem~\ref{Application B}.

\subsection{Expanding Thurston maps}
\label{sub:Expanding Thurston_maps}

We review some key concepts and results concerning expanding Thurston maps.
For a more thorough treatment of the subject, we refer the reader to the monographs \cite{bonk2017expanding,li2017ergodic}.

Let $S^2$ denote an oriented topological $2$-sphere and $f \colon S^2 \rightarrow S^2$ be a branched covering map. 
We denote by $\deg_f(x)$ the local degree of $f$ at $x \in S^2$.
The \emph{degree} of $f$ is $\deg{f} = \sum_{x\in f^{-1}(y)} \deg_{f}(x)$ for $y\in S^2$ and is independent of $y$. 

A point $x\in S^2$ is a \emph{critical point} of $f$ if $\deg_f(x) \geqslant 2$. 
The set of critical points of $f$ is denoted by $\crit{f}$. 
A point $y\in S^2$ is a \emph{postcritical point} of $f$ if $y = f^n(x)$ for some $x \in \crit{f}$ and $n\in \n$. 
The set of postcritical points of $f$ is denoted by $\post{f}$. 
If $\card{\post{f}}<+\infty$, then $f$ is said to be \emph{postcritically-finite}.

\begin{definition}[Thurston maps]
	A Thurston map is a branched covering map $f \colon S^2 \rightarrow S^2$ with $\deg f \geqslant 2$ and $\card{\post{f}}< +\infty$.
\end{definition}

We can now define expanding Thurston maps.

\begin{definition}[Expanding Thurston maps]     \label{def:expanding_Thurston_maps}
	A Thurston map $f \colon S^2 \rightarrow S^2$ is called \emph{expanding} if there exists a metric $d$ on $S^2$ that induces the standard topology on $S^2$ and a Jordan curve $\mathcal{C} \subseteq S^2$ containing $\post{f}$ such that
	\begin{equation*}    \label{eq:definition of expansion}
		\lim_{n \to +\infty} \sup \bigl\{ \diamn{d}{X} \describe X \text{ is a connected component of } f^{-n}\bigl(S^2 \smallsetminus \mathcal{C} \bigr)  \bigr\} = 0.
	\end{equation*}
\end{definition}

For an expanding Thurston map $f$, we can fix a particular metric $d_v$ on $S^2$ called a \emph{visual metric for $f$}. 
Such a metric induces the standard topology on $S^2$ (\cite[Proposition~8.3]{bonk2017expanding}). 
For the existence of such a metric, see \cite[Chapter~8]{bonk2017expanding}. 
For a visual metric $d_v$ for $f$, there exists a unique constant $\Lambda > 1$ called the \emph{expansion factor} of $d_v$ (see \cite[Chapter~8]{bonk2017expanding} for more details).

We summarize the existence and the uniqueness of equilibrium states for expanding Thurston maps in the following theorem, which is part of \cite[Theorem~1.1]{li2018equilibrium}.

\begin{theorem}[Li \cite{li2018equilibrium}]     \label{thm:properties of equilibrium state}
	Let $f \colon S^2 \rightarrow S^2$ be an expanding Thurston map, and $d_v$ be a visual metric on $S^2$ for $f$. Assume that $\phi\in C^{0, \holderexp}\bigl(S^2,d_v\bigr)$ is a real-valued \holder continuous function with an exponent $\holderexp \in (0,1]$. 
	Then there exists a unique equilibrium state $\mu_{\phi}$ for $f$ and $\phi$.
\end{theorem}

The main tool used in \cite{li2018equilibrium} to develop the thermodynamic formalism for expanding Thurston maps is the Ruelle operator.
We recall the definition of the Ruelle operator below and refer the reader to \cite[Chapter~3.3]{li2017ergodic} for a detailed discussion.

Let $f \colon S^{2} \rightarrow S^2$ be an expanding Thurston map and $\phi\in C\bigl(S^{2}\bigr)$ be a real-valued continuous function.
The \emph{Ruelle operator} $\ruelleopt[\phi]$ (associated with $f$ and $\phi$) acting on $C\bigl(S^{2}\bigr)$ is given by
\begin{equation}    \label{eq:def:Ruelle operator}
	\ruelleopt[\phi](u)(x)  \=  \sum_{y \in f^{-1}(x)} \deg_{f}(y) u(y) \myexp{\phi(y)}
\end{equation}
for each $u \in C\bigl(S^{2}\bigr)$.
Note that $\ruelleopt[\phi]$ is a well-defined and continuous operator on $C\bigl(S^{2} \bigr)$.

Recall that the measure-theoretic entropy function of a continuous map $T \colon X \rightarrow X$ defined on a compact metrizable topological space $X$ is the function $\mu \mapsto h_{\mu}(T)$ defined on the space $\cM(X,T)$ equipped with the weak$^*$ topology.

The following result regarding the upper semicontinuity of the measure-theoretic entropy function is established in \cite[Theorem~1.1]{shi2024entropy}, extending \cite[Corollary~1.3]{li2015weak}.

\begin{theorem}[Li \& Shi \cite{shi2024entropy}]    \label{thm:upper semicontinuous iff no periodic critical points}
	Let $f \colon S^2 \rightarrow S^2$ be an expanding Thurston map. 
	Then the measure-theoretic entropy function of $f$ is upper semicontinuous if and only if $f$ has no periodic critical points.
\end{theorem}

\subsection{Misiurewicz--Thurston rational maps}
\label{subsec:Misiurewicz--Thurston rational maps}

We apply Theorem~\ref{theorem A} to show the computability of equilibrium states for Misiurewicz--Thurston rational maps.

A \defn{Misiurewicz--Thurston rational map} is a postcritically-finite rational map on the Riemann sphere $\ccx$ without periodic critical points.
We remark that a postcritically-finite rational map is an expanding Thurston map (in the sense of Definition~\ref{def:expanding_Thurston_maps}) if and only if it has no periodic critical point (see \cite[Proposition~2.3]{bonk2017expanding}). To discuss the computability of rational maps, we give a computable structure of $\ccx$ as follows.

\begin{prop}		\label{prop:hatCiscompactmetricspace}
    Let $\cS\bigl(\ccx\bigr) \= \{s_n\}_{n\in\N}$ be an enumeration of $\Q\bigl(\ccx\bigr)\=\{a+b\mathbf{i}:a,\,b\in\Q\}$ such that there exists an algorithm that for each $n\in\N$, upon input $n$, outputs $p_1,\,q_1,\,r_1,\,p_2,\,q_2,\,r_2\in\N$ with $s_n=(-1)^{r_1} \frac{p_1}{q_1}+(-1)^{r_2}\frac{p_2}{q_2}\mathbf{i}$. Then $\bigl(\widehat{\cx},\,\sigma,\,\cS\bigl(\widehat{\cx}\bigr)\bigr)$ is a computable metric space in which $\widehat{\cx}$ is recursively compact, where $\sigma$ is the chordal metric on $\widehat{\cx}$.  
\end{prop}

The above result immediately follows from Definitions~\ref{def:computablemetricspace},~\ref{def:recursivelyprecompact}, and Proposition~\ref{prop:relationbetweencompactandprecompact}. We call such an enumeration given in Proposition~\ref{prop:hatCiscompactmetricspace} an \emph{effective enumeration of $\Q\bigl(\ccx\bigr)$}.

The following result states that there exists an algorithm that computes all zeros of a computable polynomial (cf.\ \cite[Proposition~3.2]{braverman2009computability}; see also \cite[Theorem~3.38]{Qiandu2025Lecture}).
 
\begin{lemma}		\label{lem:rootfindingalgorithm}
	Let $\cS\bigl(\ccx\bigr)=\{s_n\}_{n\in\N}$ be an effective enumeration of $\Q\bigl(\ccx\bigr) \= \{a+b\mathbf{i} : a,\,b\in\Q\}$. Then there exists an algorithm that satisfies the following property:
	
	For all $k,\,l\in\N$, and complex polynomial $p$ of degree $k$, this algorithm outputs a sequence $\{q_i\}_{i=1}^k$ of integers such that there exists an enumeration $\{x_i\}_{i=1}^k$ of all the zeros of $p$ (counting with multiplicity) satisfying that $\sigma\bigl(s_{q_i},x_i\bigr)<2^{-l}$ for each integer $1\leq i\leq k$, after we input the following data into this algorithm:
	\begin{enumerate}
		\smallskip

        \item the integers $k$ and $l$,

        \smallskip
		
		\item an algorithm computing all the coefficients of the polynomial $p$.
	\end{enumerate}
\end{lemma}

In this subsection, we investigate computable Misiurewicz--Thurston rational maps, i.e., Misiurewicz--Thurston rational maps $f=h_1/h_2$ for which the coefficients of the polynomials $h_1,\,h_2$ are all computable. Here $h_1,\,h_2$ are two polynomials without common roots. We now design an algorithm to compute the Ruelle operator $\cL_{\potential}$ (recall \eqref{eq:def:Ruelle operator}).

\begin{prop}\label{prop:computability of Ruelle operators}
	There exists an algorithm that satisfies the following property:
    
    For all $n,\,m\in\N$, $\phi,\,u\in C \bigl( \ccx,\sigma \bigr)$, a Misiurewicz--Thurston rational map $f=h_1/h_2$, and $x\in\ccx\smallsetminus\bigl\{f^i(\infty):1\leq i\leq m\bigr\}$, the algorithm outputs a rational $2^{-n}$-approximation of $\cL^m_{\potential}(u)(x)$, given the following input data:
	\begin{enumerate}
        \smallskip
		
		\item an algorithm computing the function $\potential:\ccx\rightarrow\R$,

        \smallskip

        \item an algorithm computing the function $u:\ccx\rightarrow\R$,

        \smallskip

        \item an algorithm computing all the coefficients of the polynomials $h_1$ and $h_2$.
		
		\smallskip
		
		\item an oracle $\tau:\N\rightarrow\N$ for the point $x$,
		
		\smallskip
		
		\item the integers $n$ and $m$.
	\end{enumerate}
\end{prop}

\begin{proof}
	Since we can compute the function $\cL_{\potential}^m(u)$ by iterating the operator $\cL_{\potential}$ on the function $u$, by (\ref{eq:def:Ruelle operator}), it suffices to establish the algorithm in the case where $m\=1$.
    
    First, we consider the preimage of $x\in\ccx\smallsetminus\{f(\infty)\}$ with respect to $f$ and define the polynomial $g(z)\=h_1(z)-xh_2(z)$ for each $z\in\ccx$. Since $x\neq f(\infty)$, we have $\deg(g)=\max\{\deg(h_1),\,\deg(h_2)\}=\deg(f)$ and can use data~(iii)~and~(iv) to compute all the coefficients of $g$. By Lemma~\ref{lem:rootfindingalgorithm}, we can compute all the zeros $\{y_i\}_{i=1}^k$ of $g$ (counting with multiplicity). 
    Since $x \neq f(\infty)$, the zeros $\{y_i\}_{i=1}^k$ of $g$ correspond precisely to the preimage of $x$ with respect to $f$ (counting with multiplicity). Thus, by (\ref{eq:def:Ruelle operator}), we obtain that
	\begin{equation*}
		\cL_{\potential}(u)(x)
        =\sum_{y\in f^{-1}(x)}\deg_{f}(y)u(y)\exp(\potential(y))=\sum_{i=1}^ku(y_i)\exp(\phi(y_i)).
	\end{equation*}
    Note that the exponential function $\exp\colon\R\mapping\R$ is computable by Example~\ref{exp:computablefunctionexample}. Then we can use data~(i)~and~(ii) to compute the value of $\cL_{\potential}(u)(x)$.
\end{proof}

We now prove the computability of the topological pressure $P(f,\potential)$.

\begin{prop}\label{prop:computabilityoftopologypressure}
	Let $f\colon\ccx\rightarrow\ccx$ be a computable Misiurewicz--Thurston rational map, $d_v$ be a visual metric on $\ccx$ for $f$, and $\alpha\in (0,1]$. Then there exists an algorithm that for all $n\in\N$ and $\potential\in C^{0,\holderexp}\bigl( \ccx, d_v\bigr)$, outputs a rational $2^{-n}$-approximation of $P(f,\potential)$, after inputting the following data:
    \begin{enumerate}
		\smallskip

		\item an algorithm $\Phi$ computing the function $\potential$,

		\smallskip

        \item a rational number $R$ with $\abs{\potential}_{\alpha,d_{v}}\leq R$,

        \smallskip

		\item the integer $n$.
    \end{enumerate}
\end{prop}

\begin{proof}
    First, we design an algorithm $\cM_C(\Phi,R,n)$ as follows for each given $C\in\Q$.
    We begin with computing $N\in\N$ with $N>2^{n+1}CR$. Let $\cS\bigl(\ccx\bigr)=\{s_n\}_{n\in\N}$ be an effective enumeration of $\Q\bigl(\ccx\bigr)$. Then by computing the distance between ideal points and points in $\bigl\{f^i(\infty):1\leq i\leq N\bigr\}$, we can find $m\in\N$ such that $s_m\in\ccx\smallsetminus\bigl\{f^i(\infty):1\leq i\leq N\bigr\}$. Note that the logarithm function $\log\colon \R^+\mapping\R$ is computable by Example~\ref{exp:computablefunctionexample}. By employing the algorithm $\Phi$ and the algorithm in Proposition~\ref{prop:computability of Ruelle operators}, we can compute and output the value $v$ such that 
    \begin{equation*}
    |v-w|<2^{-n-1},\quad\text{ where }w \=  N^{-1}\log\bigl(\cL_{\potential}^{N}\bigl(\mathbbm{1}_{\ccx}\bigr)(s_m)\bigr).
	\end{equation*}
	
    By \cite[Lemma~5.15]{li2018equilibrium}, there exists $C_0\in\Q$ such that
	\begin{equation}		\label{eq:inequalityforiterationofruelleoperator}
		\abs[\big]{\log \parentheses[\big]{ \normopt^{n}\bigl(\indicator{\ccx}\bigr)(x) }} \leqslant C_0 \abs{\potential}_{\alpha,d_v}\quad\text{ for all }x\in\ccx,\,n\in\N_0,\text{ and }\potential\in C^{0,\alpha}\bigl( \ccx ,\sigma \bigr),
	\end{equation}
    where $\overline{\phi}(x) \= \phi(x)- P(f,\phi)$. 
    
    Finally, we demonstrate that the algorithm $\cM_{C_0}(\Phi,R,n)$ outputs $v$ with $\abs{v-P(f,\potential)}<2^{-n}$.
    Indeed, by the definition of $\cM_{C_0}(\Phi,R,n)$ and (\ref{eq:inequalityforiterationofruelleoperator}), we have
	\begin{equation*}
		\begin{aligned}
			\abs{ w-P(f,\potential) }
            &=\abs[\big]{ N^{-1}\log\bigl(e^{-NP(f,\potential)}\cL_{\potential}^{N}\bigl(\mathbbm{1}_{\ccx}\bigr)(x_0)\bigr)}<N^{-1}\abs[\big]{\log \parentheses[\big]{ \normopt^{N}\bigl(\indicator{\ccx}\bigr)(x_0)}}
            \leq N^{-1}C_0R<2^{-n-1}.
		\end{aligned}
	\end{equation*}
    Hence, by $\abs{v-w}<2^{-n-1}$, this implies that $\abs{v-P(f,\potential)}<2^{-n}$.
\end{proof}

Now we can apply Theorem~\ref{theorem A} to prove Theorem~\ref{Application A}.

\begin{proof}[\bf Proof of Theorem~\ref{Application A}]
    By \cite[Corollary~1, p.~379]{ljubichEntropyPropertiesRational1983}, the measure-theoretic entropy function $\nu\mapsto h_{\nu}(f)$ is upper semicontinuous. 
    By the hypothesis of Theorem~\ref{Application A}, we have $\cE(f,\phi_n)=\{\mu_n\}$ for each $n\in\N$. 
    Hence, to prove Theorem~\ref{Application A}, by Theorem~\ref{theorem A}, it suffices to verify that the sequence $\{\phi_n\}_{n\in\N}$ of functions $\potential_n\colon\ccx\rightarrow\R$ satisfies properties~(i)~and~(ii) in Theorem~\ref{theorem A} in the case where $T_n\=f$ for each $n\in\N$.
	
    We first apply Proposition~\ref{prop:computabilityoftopologypressure} to show the uniform computability of $\{P(f,\potential_n)\}_{n\in\N}$. Since $f$ is a Misiurewicz--Thurston rational map, by \cite[Lemma~18.10]{bonk2017expanding}, the identity map $\id{\ccx}\colon\bigl(\ccx,d_v\bigr)\rightarrow\bigl(\ccx,\sigma\bigr)$ is a quasisymmetric homeomorphism. Note that $\bigl(\ccx,\sigma\bigr)$ is a uniformly perfect space. By \cite[Corollary~11.5]{heinonen2001lectures}, the identity map $\id{\ccx}$ is H\"older continuous on bounded sets. Hence, by the compactness of $\bigl(\ccx,\sigma\bigr)$, there exist two constants $\beta\in(0,1]$ and $C\in\Q^+$ satisfying that
    \begin{equation}		\label{eq:relationbetweensigmametricandvisualmetric}
		\sigma(x,y)\leq Cd_v(x,y)^{\beta}\quad\text{ for all }x,y\in\ccx.
	\end{equation} 
    Thus, for each $n\in\N$, by $\potential_n\in C^{0,\alpha}\bigl(\ccx,\sigma\bigr)$ and $\abs{\potential_n}_{\alpha,\sigma}\leq Q_n$, we obtain that $\potential_n\in C^{0,\alpha\beta}\bigl(\ccx,d_v\bigr)$ and $\abs{\potential_n}_{\alpha\beta,d_v}\leq C^{\alpha}\abs{\potential_n}_{\alpha,\sigma}\leq C^{\alpha}Q_n$. Hence, by Proposition~\ref{prop:computabilityoftopologypressure}, $\{P(f,\potential_n)\}_{n\in\N}$ is a sequence of uniformly computable real numbers, i.e., $\{\potential_n\}_{n\in\N}$ satisfies property~(ii) in Theorem~\ref{theorem A}.
	
    Let $\cS\bigl(\ccx\bigr)=\{s_i\}_{i\in\N}$ be an effective enumeration of $\Q\bigl(\ccx\bigr)$. We define for all $i\in\N$ and $x\in\ccx$, $f_0(x) \=  1$ and $f_i(x) \= \sigma(x,s_i)$. Denote by $D$ the set of rational linear combinations of finitely many functions in $\bigl\{\prod_{j=1}^mf_{i_j} :  m\in\N\text{ and }i_j\in\N_0\text{ for each integer }1\leq j\leq m \bigr\}$. By the Stone--Weierstrass theorem, $D$ is dense in $C\bigl(\ccx \bigr)$. Since $\N^*$ is a recursively enumerable set, there exists an enumeration $\{\psi_k\}_{k\in\N}$ of $D$ such that there exists an algorithm that, for each $k\in\N$, on input $k$, outputs the following expression of $\psi_k$:
	\begin{equation}		\label{eq:expressionforpsij}
		\psi_k=\sum_{t=1}^{l_k}q_{k,t}\prod_{j=1}^{m_{k,t}}f_{i_{j,k,t}}, \quad\text{ where }l_k,\,m_{k,t},\,i_{j,k,t}\in\N_0\text{ and }q_{k,t}\in\Q\text{ for all } j,\,k,\,t\in\N.
	\end{equation}
    Thus, since $\{f_i\}_{i\in\N_0}$ is a sequence of uniformly computable functions, so is the sequence $\{\psi_k\}_{k\in\N}$.
	
	\smallskip
	\emph{Claim.} $\{P(f,\psi_k)\}_{k\in\N}$ is a sequence of uniformly computable real numbers.
	\smallskip
	
	\emph{Proof of the claim.} By the definition of the chordal metric $\sigma$ and the function $f_i$, we have $f_i(x)=\sigma(x,s_i)\leq 2$ for all $i\in\N_0$ and $x\in\ccx$. Thus, since $f_i\in C^{0,1}\bigl(\ccx,\sigma\bigr)$ and $\abs{f_i}_{1,\sigma}=1$ for each $i\in\N_0$, we obtain that for all $k,\,t\in\N$ and $x,\,y\in\ccx$,
	\begin{equation*}
		\begin{aligned}
			\abs[\Bigg]{\prod_{j=1}^{m_{k,t}}f_{i_{j,k,t}}(x)-\prod_{j=1}^{m_{k,t}}f_{i_{j,k,t}}(y)}
			&=\abs[\Bigg]{\sum_{j=1}^{m_{k,t}}\Biggl(\prod_{p=1}^{j-1}f_{i_{p,k,t}}(x)\cdot\prod_{q=j+1}^{m_{k,t}}f_{i_{q,k,t}}(y)\cdot\bigl(f_{i_{j,k,t}}(x)-f_{i_{j,k,t}}(y)\bigr)\Biggr)}\\
			&\leq\sum_{j=1}^{m_{k,t}}\Biggl(\prod_{p=1}^{j-1}\abs[\big]{f_{i_{p,k,t}}(x)}\cdot\prod_{q=j+1}^{m_{k,t}}\abs[\big]{f_{i_{q,k,t}}(y)}\cdot\abs[\big]{f_{i_{j,k,t}}(x)-f_{i_{j,k,t}}(y)}\Biggr)\\
			&\leq 2^{m_{k,t}-1}m_{k,t}\cdot\sigma(x,y).
		\end{aligned}
	\end{equation*}
	By (\ref{eq:expressionforpsij}), $\abs{\psi_k}_{1,\sigma}\leq\sum_{t=1}^{l_k} 2^{m_{k,t}-1}m_{k,t}\abs{q_{k,t}} $ for each $k\in\N$. Hence, by Definition~\ref{def:Algorithm about computable functions}, the function $F\colon\N\rightarrow\R$ given by $F(k) \= \sum_{t=1}^{l_k} 2^{m_{k,t}-1}m_{k,t}\abs{q_{k,t}}$ for each $k\in\N$ is a computable function satisfying $F(k)\geq\abs{\psi_k}_{1,\sigma}$ for each $k\in\N$. Thus, by (\ref{eq:relationbetweensigmametricandvisualmetric}), we have $\psi_k\in C^{0,\beta}\bigl(\ccx,d_v\bigr)$ and $\abs{\psi_k}_{\beta,d_v}\leq C\abs{\psi_k}_{1,\sigma}\leq C F(k)$ for each $k\in\N$. Note that $C\in\Q$ and $F$ is a computable function. It follows from Proposition~\ref{prop:computabilityoftopologypressure} that $\{P(f,\psi_k)\}_{k\in\N}$ is a sequence of uniformly computable real numbers. This completes the proof of the claim.
    \smallskip
	
    By the above claim, the sequence $\{\potential_n\}_{n\in\N}$ satisfies property~(i) in Theorem~\ref{theorem A}. Therefore, by Theorem~\ref{theorem A}, $\{\mu_n\}_{n\in\N}$ is a sequence of uniformly computable measures.
\end{proof}

\subsection{Computable expanding Thurston maps}
\label{subsec:The measures of maximal entropy for computable expanding Thurston maps}

In this subsection, we focus on computable expanding Thurston maps on $\ccx$, i.e., expanding Thurston maps which are computable functions from the computable metric space $\bigl(\ccx,\,\sigma,\,\cS\bigl(\ccx\bigr)\bigr)$ to itself, and investigate the computability of measures of maximal entropy for computable expanding Thurston maps with some additional computability assumptions.

Before the proof of Theorem~\ref{Application B}, we establish the following result.

\begin{lemma}\label{lem:lowersemicomputableopennessofconnectedcomponents}
    Let $(X,\,\rho,\,\cS)$ be a computable metric space in which $X$ is recursively compact. 
    Suppose all balls in $(X,\rho)$ are connected. 
    Assume that $I$ is a recursively enumerable set and that $\{U_i\}_{i\in I}$ is uniformly lower semi-computable open. Then there exists a recursively enumerable set $E\subseteq\N\times I$ and a sequence $\{V_{n,i}\}_{(n,i)\in E}$ of uniformly lower semi-computable open sets that are connected such that $U_i=\bigcup_{(n,i)\in E_i}V_{n,i}$, where $E_i\coloneqq\{(n,i)\in E:n\in\N\}$ for each $i\in I$.
\end{lemma}

\begin{proof}
    Without loss of generality, we assume that $\bigcup_{i\in I}U_i\neq\emptyset$. Write $\cS = \{s_n\}_{n\in\N}$. Since $\{U_i\}_{i\in I}$ is uniformly lower semi-computable open, by Proposition~\ref{prop:semidecideinopenset}, there exists a nonempty recursively enumerable set $E\subseteq\N\times I$ such that $\{s_n:(n,i)\in E_i\}=\{s_n:n\in\N\}\cap U_i$, where $E_i=\{(n,i)\in E:n\in\N\}$ for each $i\in I$. We define $V_{n,i} \=  B\bigl(s_n,\rho\bigl(s_n,U_i^c\bigr)\bigr)$ for each $(n,i)\in E$. Since all open balls in $(X,\rho)$ are connected, $V_{n,i}$ is connected for each $(n,i)\in E$. 

    Now we fix $i\in I$ and check that $U_i=\bigcup_{(n,i)\in E_i}V_{n,i}$. Indeed, for each $(n,i)\in E_i$, since $B\bigl(s_n,\rho\bigl(s_n,U_i^c\bigr)\bigr)$ is connected and $s_n\in U_i$, we have $B\bigl(s_n,\rho\bigl(s_n,U_i^c\bigr)\bigr)\subseteq U_i$. By the definition of $\{V_{n,i}\}_{(n,i)\in E}$, this implies that $\bigcup_{(n,i)\in E}V_{n,i}\subseteq U_i$.
    Then we establish $U_i\subseteq\bigcup_{(n,i)\in E_i}V_{n,i}$. We consider an arbitrary $x_0\in U_i$, and show that $x_0\in V_{n_0,i}$ for some $n_0\in\N$ with $(n_0,i)\in E_i$. Since $\{s_n:n\in\N\}$ is dense in $X$ and $\{s_n:(n,i)\in E_i\}=\{s_n:n\in\N\}\cap U_i$, there exists $n_0\in\N$ with $(n_0,i)\in E_i$ such that $\rho(s_{n_0},x_0)<\rho\bigl(x_0,U_i^c\bigr)/2$. 
    Thus, we have $x_0\in B\bigl(s_{n_0},\rho\bigl(x_0,U_i^c\bigr)\big/2\bigr)\subseteq B\bigl(s_{n_0},\rho\bigl(x_{0},U_i^c\bigr)-\rho\bigl(x_0,s_{n_0}\bigr)\bigr)\subseteq B\bigl(s_{n_0},\rho\bigl(s_{n_0},U_i^c\bigr)\bigr)=V_{n_0,i}$.
    Therefore, we obtain that $U_i=\bigcup_{(n,i)\in E_i}V_{n,i}$ for each $i\in I$. 
    
    Finally, we show that $\{V_{n,i}\}_{(n,i)\in E}$ is uniformly lower semi-computable open. Indeed, since $X$ is recursively compact, by Proposition~\ref{thecomplementofrecursivelycompactset}~(iii), $\bigl\{U_i^c\bigr\}_{i\in I}$ is uniformly recursively compact. Note that $E$ is a nonempty recursively enumerable set. Then by Proposition~\ref{thecomplementofrecursivelycompactset}~(iv), the sequence $\bigl\{\rho\bigl(s_{n},U_i^c\bigr)\bigr\}_{(n,i)\in E}$ is uniformly lower semi-computable. Note that for each pair $(n,i)\in E$, $V_{n,i} =  B\bigl(s_n,\rho\bigl(s_n,U_i^c\bigr)\bigr)$ by definition. Then by Proposition~\ref{prop:lowersemicomputableraiusball}, $\{V_{n,i}\}_{(n,i)\in E}$ is uniformly lower semi-computable open. 
\end{proof}

To prove Theorem~\ref{Application B}, we need to recall some notions.
Let $f \colon \ccx \rightarrow \ccx$ be an expanding Thurston map and $\mathcal{C} \subseteq \ccx$ be a Jordan curve containing $\post{f}$.
For $n \in \n_0$, we define the set of $n$-tiles as 
\begin{equation*}
    \Tile{n} \coloneq \bigl\{ \overline{A} :  A \text{ is a connected component of } \ccx \smallsetminus f^{-n}(\mathcal{C}) \bigr\}.
\end{equation*}
For $n \in \n_0$ and $v \in f^{-n}(\post{f})$, we define the \emph{closed $n$-flower of $v$} as
\[ 
    \overline{W}^n(v) \define \bigcup \set{X \describe X \in \Tile{n}, \,  v \in X }.
\]

\begin{proof}[\bf Proof of Theorem~\ref{Application B}]
    Denote by $\mu_0$ the unique measure of maximal entropy of $f$ and by $\phi_0$ the constant function which equals $0$ on $\ccx$. 
    Recall that by Proposition~\ref{prop:hatCiscompactmetricspace}, $\bigl(\ccx,\,\sigma,\,\cS\bigl(\ccx\bigr)\bigr)$ is a computable metric space in which $\ccx$ is recursively compact. Denote by $\mathrm{Per}(g)$ the set of periodic points of a map $g\colon \ccx\rightarrow \ccx$. 
    
    Now we show that there exists $k\in\N$ such that $f^{km}(x)=x$ for all $m\in\N$ and $x\in\mathrm{Per}\bigl(f^{km}\bigr)\cap\mathrm{post}\bigl(f^{km}\bigr)$. 
    Indeed, since $\card{\mathrm{Per}(f)\cap\post{f}}<+\infty$, there exists $k\in\N$ such that for all $x\in\mathrm{Per}(f)\cap\post{f}$, we have $f^{k}(x)=x$.
    Thus, since $\mathrm{Per}(f^{n})\cap\mathrm{post}(f^{n})\subseteq\mathrm{Per}(f)\cap\post{f}$ for each $n\in\N$, we have $f^{km}(x)=x$ for all $m\in\N$ and $x\in\mathrm{Per}\bigl(f^{km}\bigr)\cap\post{f^{km}}$.

    Furthermore, by \cite[Theorem~15.1]{bonk2017expanding} we can find $m\in\N$ such that there exists an $f^m$-invariant Jordan curve $\mathcal{C} \subseteq S^{2}$ with $\mathrm{post}\bigl(f^{km}\bigr) = \post{f} \subseteq \mathcal{C}$. 
    We set $F \=  f^{km}$ and fix such a Jordan curve $\mathcal{C}$. 
    Then by the previous discussion, we obtain that $F(\mathcal{C}) \subseteq \mathcal{C}$ and all points in $\mathrm{Per}(F)\cap\post{F}$ are fixed points of $F$.

    \smallskip
    \emph{Claim~1.} Item~(i) in the Assumptions in Subsection~\ref{sub:applications} holds in the case where $(X,\,\rho,\,\cS)\=\bigl(\ccx,\,\sigma,\,\cS\bigl(\ccx\bigr)\bigr)$, 
    $X_n\=\ccx$, $T_n\=F$, $Y_n\=\ccx\smallsetminus F^{-1}(\post{F})$, and $J_n\=\deg{F}\cdot\mathbbm{1}_{\ccx}$ for each $n\in\N$. 

    \smallskip
    
    \emph{Proof of Claim~1.} Indeed, we have previously shown that $\bigl(\ccx,\,\sigma,\,\cS\bigl(\ccx\bigr)\bigr)$ is a computable metric space. 
    Since $f$ is computable, by Definition~\ref{def:Algorithm about computable functions}, $\{T_n\}_{n\in\N}$, with $T_n=F$ for all $n\in\N$, is a sequence of uniformly computable functions. By \cite[Lemma~6.5]{bonk2017expanding}, $F$ is an expanding Thurston map with $\post{F}=\post{f}$. Note that all critical points of $f$ are computable by hypothesis. Then by Definition~\ref{def:Algorithm about computable functions}, all postcritical points of $f$ are computable. Hence, $\post{F}$ is a finite set of computable points. Moreover, $\{J_n\}_{n\in\N}$, with $J_n=\deg{F}\cdot\mathbbm{1}_{\ccx}$ for all $n\in\N$, is a sequence of uniformly computable functions which are nonnegative and Borel.
	
    We begin by constructing a recursively enumerable set $K$ and a sequence $\{Y_{n,k}\}_{(n,k)\in L}$ of uniformly lower semi-computable open sets, where $L=\N\times K$. 
    Now we write $\cS\bigl(\ccx\bigr)=\{s_i\}_{i\in\N}$. Since $\post{F}$ is a finite set of computable points, by Proposition~\ref{thecomplementofrecursivelycompactset}~(i),~(ii),~and~(v), $\ccx\smallsetminus\post{F}$ is a nonempty lower semi-computable open set. By Proposition~\ref{prop:semidecideinopenset}, there exists a recursively enumerable set $I\subseteq\N$ such that $\Q\bigl(\ccx\bigr)\smallsetminus\post{F}=\{s_i:i\in I\}$. We define $r_i\=  \sigma(s_i,\post{F}) > 0$ for each $i\in I$. Thus, by Definition~\ref{def:computability of real}, $\{r_i\}_{i\in I}$ is uniformly computable. Consequently, by Proposition~\ref{prop:lowersemicomputableraiusball}, $\{D_i\}_{i\in I}$ is uniformly lower semi-computable open, where $D_i \= B(s_i,r_i)$ for each $i\in I$. As $F$ is computable, $\bigl\{F^{-1}(D_i)\bigr\}_{i\in I}$ is also uniformly lower semi-computable. By Lemma~\ref{lem:lowersemicomputableopennessofconnectedcomponents}, there exists a recursively enumerable set $K\subseteq\N\times I$ and a sequence $\{V_k\}_{k\in K}$ of uniformly lower semi-computable open sets that are connected such that $\bigcup_{k\in K_i}V_k=F^{-1}(D_i)$ for each $i\in I$. 
    
    Define $Y_{n,k}\=V_k$ for all $n\in\N$ and $k\in K$. Then we check that $Y_{n,k}$ is admissible for $T_n$, and $Y_{n,k}\subseteq X_n$ for each $(n,k)\in L$. By definitions, it suffices to show that $\bigcup_{k\in K}V_k=\ccx\smallsetminus F^{-1}(\post{F})$ and $V_k$ is admissible for $F$ for each $k\in\N$.
    
    First, we verify that $\bigcup_{k\in K}V_k=\ccx\smallsetminus F^{-1}(\post{F})$. 
    Indeed, by $\{s_i:i\in I\}=\Q\bigl(\ccx\bigr)\smallsetminus\post{F}$, $\card{\post{F}}<+\infty$, and the definition of $\{r_i\}_{i\in I}$, we obtain that $\bigcup_{i\in I}D_i=\ccx\smallsetminus\post{F}$. Hence, we have $\bigcup_{k\in K}V_k=\bigcup_{i\in I}\bigcup_{k\in K_i}V_k=\bigcup_{i\in I}F^{-1}(D_i)=F^{-1} \bigl( \bigcup_{i\in I} D_i \bigr)=\ccx\smallsetminus F^{-1}(\post{F})$.
    
    Next, we verify that $V_k$ is admissible for $F$ for each $k\in K$. Now we consider an arbitrary $k=(n,i)\in K$. Since $V_k$ is an open set and $F\colon\ccx\mapping\ccx$ is a Borel-measurable function, by \cite[Corollary~15.2]{kechris2012classical}, it suffices to show that $F$ is injective on $V_k$. Since $F$ is an expanding Thurston map, by \cite[Lemma~A.11]{bonk2017expanding}, we obtain that the restriction $F|_{\ccx\smallsetminus F^{-1}(\post{F})}\colon\ccx\smallsetminus F^{-1}(\post{F})\rightarrow\ccx\smallsetminus\post{F}$ is a covering map. 
    
    Now we demonstrate that the restriction $F|_{F^{-1}(D_i)}$ is a covering map onto $D_i$. We fix an arbitrary $x_0\in D_i$. Since $F$ is a finite map, $F^{-1}(x_0)$ is a finite set, say $\{z_n:n\in[1,b]\cap\N\}$. Note that $D_i\subseteq\ccx\smallsetminus\post{F}$ and $D_i$ is path-connected and locally path-connected. Then by \cite[Lemma~A.6~(ii)]{bonk2017expanding}, for each integer $1\leq n\leq b$, there exists a continuous map $h_n\colon D_i\rightarrow \ccx\smallsetminus F^{-1}(\post{F})$ such that $h_n(x_0)=z_n$ and $F\circ h_n=\mathrm{id}_{D_i}$.
	
    In order to show that for each pair of distinct integers $1\leq l<l^{\prime}\leq b$,    \begin{equation}\label{eq:homeomorphism1}
        h_l(D_i)\cap h_{l^{\prime}}(D_i)=\emptyset,
    \end{equation}
    we argue by contradiction and assume that (\ref{eq:homeomorphism1}) is not true for some pair of distinct integers $1\leq l<l^{\prime}\leq b$. Then there exist $z,\,z^{\prime}\in D_i$ with $h_l(z)=h_{l^{\prime}}(z^{\prime})$. Hence, $z=F(h_l(z))=F(h_{l^{\prime}}(z^{\prime}))=z^{\prime}$. Combined with \cite[Lemma~A.6~(i)]{bonk2017expanding}, this implies that $h_l=h_{l^{\prime}}$, which contradicts $l< l^{\prime}$. Thus, (\ref{eq:homeomorphism1}) is true for each pair of distinct integers $1\leq l<l^{\prime}\leq b$.
    
    Then we show that 
    \begin{equation}\label{eq:homeomorphism2}
        \bigcup_{n=1}^bh_n(D_i)=F^{-1}(D_i).
    \end{equation} 
    Indeed, since $F\circ h_n=\mathrm{id}_{D_i}$ for all integer $1\leq n\leq b$ and $x\in D_i$, we have $h_n(D_i)\subseteq F^{-1}(D_i)$ for each integer $1\leq n\leq b$. Hence, it suffices to show that $F^{-1}(D_i)\subseteq\bigcup_{n=1}^bh_n(D_i)$. We consider an arbitrary $y^{\prime}\in F^{-1}(D_i)$. Then by \cite[Lemma~A.6~(ii)]{bonk2017expanding}, there exists $h\colon D_i\rightarrow\ccx\smallsetminus F^{-1}(\post{F})$ such that $h(F(y^{\prime}))=y^{\prime}$ and $F\circ h=\mathrm{id}_{D_i}$. Note that $F^{-1}(x_0)=\{z_n:n\in[1,b]\cap\N\}$ and $F(h(x_0))=x_0$. Then there exists an integer $n_0\in [1,b]$ satisfying that $h_{n_0}(x_0)=z_{n_0}=h(x_0)$. Combined with \cite[Lemma~A.6~(i)]{bonk2017expanding}, this implies that $h_{n_0}=h$. Thus, we have $y^{\prime}=h(F(y^{\prime}))\subseteq h(D_i)=h_{n_0}(D_i)\subseteq\bigcup_{n=1}^bh_n(D_i)$. Hence, we complete the proof of (\ref{eq:homeomorphism2}).
    
    Since $h_n$ is continuous for each integer $1\leq n\leq b$, by (\ref{eq:homeomorphism1}) and (\ref{eq:homeomorphism2}), we have shown that the restriction $F|_{F^{-1}(D_i)}$ is a covering map onto $D_i$.
    
    Recall that $D_i$ is an open ball in $\ccx$. Hence, since $\ccx$ is locally path-connected, all connected components of $F^{-1}(D_i)$ are always path-connected. Thus, since $D_i$ is a simply connected open set, for each connected component $V$ of $F^{-1}(D_i)$, the restriction $F|_{V}$ is a homeomorphism onto $D_i$. Note that $V_k$ is connected and $V_k\subseteq F^{-1}(D_i)$. Then there exists a connected component $V$ of $F^{-1}(D_i)$ with $V\supseteq V_k$. Thus, $F$ is injective on $V_{k}$.
    
    So far we have shown Claim~1.



    \smallskip
    
    We can enumerate $\mathrm{Per}(F)\cap\post{F}$ as a sequence $\{y_i\}_{i=1}^N$ of uniformly computable (mutually distinct) points, since it is a finite set of computable points. 
    
    We now construct a sequence $\{\mathbb{B}_i\}_{i=1}^N$ of uniformly lower semi-computable open sets satisfying that $\{y_i\}\subseteq \mathbb{B}_i\subseteq \cflower{1}{y_i}$ for each integer $1\leq i\leq N$. 
    Fix an arbitrary integer $1\leq i\leq N$. By \cite[Lemma~5.28~(i)]{bonk2017expanding}, there exists $R_i\in\Q^+$ such that $B(y_i,R_i)\subseteq\cflower{1}{y_i}$.
    Since $y_i$ is a computable point, by Proposition~\ref{prop:lowersemicomputableraiusball}, $\mathbb{B}_i \=  B(y_i,R_i)$ is a lower semi-computable open set with $\{y_i\}\subseteq \mathbb{B}_i\subseteq \cflower{1}{y_i}$. Hence, by Definition~\ref{def:lower semicomputability and uniform version of open set}, $\{\mathbb{B}_i\}_{i=1}^N$ is uniformly lower semi-computable open.
    
    \smallskip
    \emph{Claim~2.} Items~(ii)~to~(v) in the Assumptions in Subsection~\ref{sub:applications} hold in the case where $(X,\,\rho,\,\cS)\=\bigl(\ccx,\,\sigma,\,\cS\bigl(\ccx\bigr)\bigr)$, $X_n=C_n\=\ccx,\,T_n\=F,\,Y_n\=\ccx\smallsetminus F^{-1}(\post{F}),\,\phi_n\=\phi_0,\,J_n\=\deg{F}\cdot\mathbbm{1}_{\ccx}$, $B_n\=\mathrm{Per}(F)\cap\post{F}$ for each $n\in\N$, $I\=\N\cap [1,N]$, $H\=\N\times I$, and $p_{n,i}\=y_i$ for each $(n,i)\in H$.
    \smallskip

    \emph{Proof of Claim~2.}
    Now we verify items~(ii)~to~(v) in the Assumptions in Subsection~\ref{sub:applications}.
    
    Item~(ii): Since $\phi_0$ is a constant function, this assumption holds.
    
    Item~(iv): Note that $\ccx$ is recursively compact in $(X,\,\rho,\,\cS)$ by Proposition~\ref{prop:hatCiscompactmetricspace}. Hence, by definitions, it suffices to show that $\card[\big]{F^{-1}(\post{F})}<+\infty$ and $\cE_0\bigl(F,\phi_0;\ccx\smallsetminus F^{-1}(\post{F})\bigr)\cap\cP\bigl(X;\ccx\smallsetminus F^{-1}(\post{F})\bigr)=\{\mu_0\}$. By Theorem~\ref{thm:properties of equilibrium state} and \cite[Corollary~17.3]{bonk2017expanding}, we have $\cE(F,\phi_0)=\cE(f,\phi_0)=\{\mu_0\}$. 
    Moreover, it follows from \cite[Corollary~7.4]{li2018equilibrium} that $\mu_0\bigl(F^{-1}(\post{F})\bigr)=\mu_0(\post{F})=0$. Combined with (19.5) in \cite{bonk2017expanding} and (\ref{eq:defofspecialequilibriumstate}), this implies that $\cE_0\bigl(F,\phi_0;\ccx\smallsetminus F^{-1}(\post{F})\bigr)=\{\mu_0\}$. 
    
    Item~(iii): Since $F$ is an expanding Thurston map, by \cite[Corollary~17.2]{bonk2017expanding}, we have $h_{\mathrm{top}}(F)=\log(\deg{F})$. Suppose $h=\phi_0$. Then property~(i) in Theorem~\ref{thm:prescribed Jacobian of equilibrium state} is satisfied in the case where $T\=F,\,Y\=\ccx\smallsetminus F^{-1}(\post{F}),\,J\=\deg{F}\cdot\mathbbm{1}_{\ccx}$, and $\phi\=\phi_0$.
    Moreover, since $F$ is an expanding Thurston map, we have $\card[\big]{F^{-1}(x)\smallsetminus F^{-1}(\post{F})}=\deg{F}$ for each $x\in F\bigl(\ccx\smallsetminus F^{-1}(\post{F})\bigr)$. Hence, property~(ii) in Theorem~\ref{thm:prescribed Jacobian of equilibrium state} is satisfied in the case where $T\=F,\,Y\=\ccx\smallsetminus F^{-1}(\post{F})$, and $J\=\deg{F}\cdot\mathbbm{1}_{\ccx}$. This verifies Item~(iii) in the Assumptions.
    
    Item~(v): Now we demonstrate that $\mathrm{Per}(F)\cap\post{F}$ is the set of periodic points of $F$ in $F^{-1}(\post{F})$, namely, $\mathrm{Per}(F)\cap\post{F}=\mathrm{Per}(F)\cap F^{-1}(\post{F})$. Indeed, this immediately follows from the property that $F(x)=x$ for each $x\in\mathrm{Per}(F)\cap\post{F}$.
    Since $\{y_i\}_{i\in I}$ is uniformly computable, $\{p_{n,i}\}_{(n,i)\in H}$, with $p_{n,i}=y_i$ for all $(n,i)\in H$, is also uniformly computable. Moreover, by the definition of $\{y_i\}_{i=1}^N$, we have that $\{p_{n,i}:(n,i)\in H_n\}=\{y_i:i\in I\}=\mathrm{Per}(F)\cap\post{F}$. 
    Therefore, Claim~2 follows.

    \smallskip
    \emph{Claim~3.} The additional statement in Theorem~\ref{thm:additionalassumptionone} is true in the case where $(X,\,\rho,\,\cS)\=\bigl(\ccx,\,\sigma,\,\cS\bigl(\ccx\bigr)\bigr),\,I\=\N\cap[1,N],\,H\=\N\times I,\,\mu_n\=\mu_0$, and $p_{n,i}\=y_i$ for each $(n,i)\in H$.
    \smallskip

    \emph{Proof of Claim~3.} Recall that $\{\mathbb{B}_i\}_{i=1}^N$ is uniformly lower semi-computable open and $F$ is a computable function. Define $U_{n,i,1}\=\mathbb{B}_i$ and $U_{n,i,k+1}\=F^{-1}(U_{n,i,k})\cap \mathbb{B}_i$ recursively for each $n,\,k\in\N$, and each $i\in I$. Then by Corollary~\ref{cor:preimageisloweropen}, we have $\{U_{n,i,k}\}_{((n,i),k)\in H\times\N}$ is uniformly lower semi-computable open. Recall that $y_i$ is a fixed point of $F$ and $y_i\in\mathbb{B}_i$ for each $i\in I$. Then we obtain that $y_i\in U_{n,i,k}$ for all $n,\,k\in\N$, and $i\in I$.
    
    Next, we set $t_{n,i,k} \=(\deg_{F}(y_i)/\deg{F})^{k-1}$ for each $n,\,k\in\N$, and each $i\in I$. Then by \cite[Lemma~4.27]{li2017ergodic}, $\deg_{F}(y_i)<\deg{F}$ for each $i\in I$. Thus, for all $n\in\N$ and $i\in I$, we obtain that $\inf_{k\in\N}t_{n,i,k}=\inf_{k\in\N}(\deg_{F}(y_i)/\deg{F})^{k}=0$. Moreover, since $\deg_{F}(y_i)/\deg{F}\in\Q$ for each $i\in I$, we have that $\{t_{n,i,k}\}_{((n,i),k)\in H\times\N}$ is a sequence of uniformly computable real numbers. 
    
    Finally, we consider an arbitrary pair of $n\in\N$ and $i\in I$ and prove that $\mu_0(U_{n,i,k})\leq t_{n,i,k}$ for each $k\in\N$. Recall that $y_i \in \post{F}$ and $F(y_i) = y_i$. 
    For each $m\in\N$, it follows from \cite[Lemma~5.28~(ii)]{bonk2017expanding} that $\card[\big]{\cflower{m}{y_i}} = 2  (\deg_{F}(y_i))^{m}$. 
    In particular, \cite[Proposition~5.16]{bonk2017expanding} implies that $F \bigl(\cflower{m+1}{y_i}\bigr) = \cflower{m}{y_i}$ for each $m\in\N$.
    Employing \cite[Proposition~17.12]{bonk2017expanding} and \cite[Proposition~7.1]{li2018equilibrium} (recall that the Jordan curve $\mathcal{C}$ is $F$-invariant), we deduce that $\mu_0\bigl(\cflower{m+1}{y_i}\bigr) = \mu_0\bigl( \cflower{m}{y_i} \bigr) \cdot(\deg_{F}(y_i) / \deg{F})$ for each $ m\in\N$. 
    Then it follows from the induction that $\mu_0\bigl(\cflower{m+1}{y_i}\bigr)=(\deg_{F}(y_i)/\deg{F})^{m}\mu_0\bigl(\cflower{1}{y_i}\bigr)\leq(\deg_{F}(y_i)/\deg{F})^{m}$ for each $m\in\N$. 
    Since $\mathbb{B}_i\subseteq \cflower{1}{y_i}$, by the definition of $\{U_{n,i,k}\}_{k\in\N}$, $U_{n,i,k}\subseteq \cflower{k}{y_i}$ for each $k\in\N$. Hence, $\mu_0(U_{n,i,k})\leq\mu_0\bigl(\cflower{k}{y_i}\bigr)\leq(\deg_{F}(y_i)/\deg{F})^{k-1}=t_{n,i,k}$ for each $k\in\N$.
    Therefore, we have shown Claim~3.

    \smallskip

    We are now in a position to complete the proof of Theorem~\ref{Application B}.
    By Claims~1, 2, and~3, it follows from Theorem~\ref{thm:additionalassumptionone} that the constant sequence $\{\mu_n\}_{n\in\N}$, defined by $\mu_n=\mu_0$ for all $n\in\N$, is a sequence of uniformly computable measures. 
    Thus, $\mu_0$ is computable.
\end{proof}

As concrete applications of Theorem~\ref{Application B}, we analyze two expanding Thurston maps $g_1$ and $g_2$ with periodic critical points. 
We verify that both maps satisfy the hypotheses of Theorem~\ref{Application B}, thereby establishing the computability of their measures of maximal entropy $\mu_1$ and $\mu_2$ (Corollary~\ref{cor:computabilityofmeasureofmaximalentropyforbarycentricg1}). 
Crucially, since both maps have periodic critical points, their measure-theoretic entropy functions fail to be upper semicontinuous (Theorem~\ref{thm:upper semicontinuous iff no periodic critical points}), precluding a direct application of Theorem~\ref{theorem A}.

Our first example $g_1$ is derived from the barycentric subdivision rule as defined in \cite[Example~12.21]{bonk2017expanding}.
We construct a sphere $S_{\triangle}$ by gluing two equilateral (Euclidean) triangles along their boundaries (see Figure~\ref{fig:example:Thurston map from barycentric subdivision}).
These two triangles serve as $0$-tiles. 
Subdividing each $0$-tile via bisectors yields six smaller triangles, producing twelve $1$-tiles.
Recursively, $n$-tiles are generated by analogously subdividing $(n-1)$-tiles via bisectors, and are Euclidean triangles.
The map $g_1 \colon S_{\triangle} \mapping S_{\triangle}$ is defined as a piecewise linear map on $S_{\triangle}$ in the following way: the orientation-preserving branched covering map $g_1$ is affine on each $1$-tile and maps the $1$-tile linearly onto a $0$-tile in a way such that $g_1$ is continuous on $S_{\triangle}$.
Then $g_1$ has a fixed critical point, and the shared boundary of $0$-tiles forming a $g_1$-invariant Jordan curve containing the postcritical set $\{A,\,B,\,C\}$ of $g_1$. 
It is not difficult to see that the diameters of $n$-tiles decay to zero as $n$ tends to $+\infty$, confirming that $g_1$ is an expanding Thurston map.

\begin{figure}[H]
	\vspace*{.2cm}
	\centering
	\begin{overpic}
		[width=13cm, tics=20]{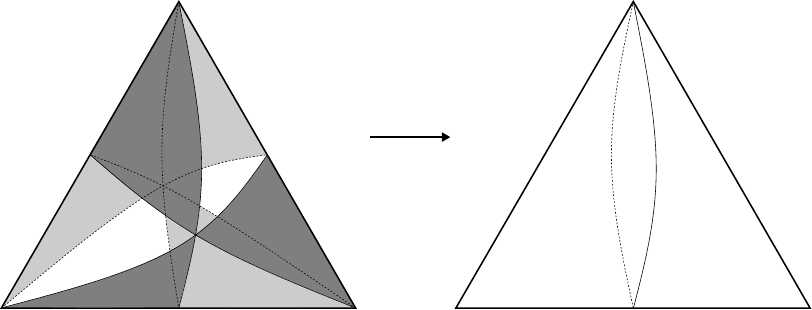}
		\put(50,23){$g_1$}
		\put(94,20){$S_{\triangle}$}
		
		\put(21,38.7){$A$}
		\put(-2.5,-2.5){$B$}
		\put(44,-2.5){$C$}
		\put(21,-2.5){$D$}
		\put(33.8,18.5){$E$}
		\put(8.6,18.4){$F$}
		\put(23.8,6){$G$}
		\put(17.4,16.7){$H$}

		\put(77,39){$A$}
		\put(53.5,-2.5){$B$}
		\put(100,-2.5){$C$}
		
	\end{overpic}
	\vspace*{.2cm}
	\caption{An expanding Thurston map from the barycentric subdivision rule.}
	\label{fig:example:Thurston map from barycentric subdivision}
\end{figure}
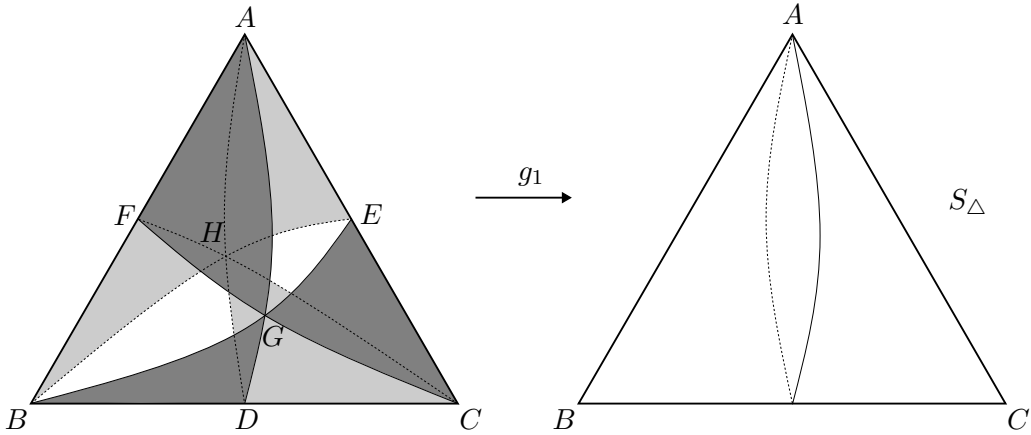

To further demonstrate the scope of Theorem~\ref{Application B} while illustrating the complexity of expanding Thurston maps, we construct a more intricate example $g_2$. 
The map $g_2$ acts on the same polyhedral sphere $S_{\triangle}$ with the postcritical set $\{A,\,B,\,C\}$.
While $g_1$ subdivides each $0$-tile into six congruent smaller triangles, $g_2$ implements a finer subdivision rule: each $0$-tile splits into eight triangles following the structure illustrated in Figure~\ref{fig:example:Thurston map with periodic critical points}.
Similarly, the orientation-preserving branched covering map $g_2 \colon S_{\triangle} \mapping S_{\triangle}$ is defined as a piecewise linear map on $S_{\triangle}$ that maps each $1$-tile onto a $0$-tile linearly.
The map $g_2$ also has a fixed critical point.
Recursive iteration of this subdivision generates $n$-tiles whose diameters decrease to zero as $n$ tends to $+\infty$, showing that $g_2$ is an expanding Thurston map.

Now we apply Theorem~\ref{Application B} to show that their measures of maximal entropy $\mu_1$ and $\mu_2$ are computable. 

We first define the topology of these two equilateral triangles by embedding them into $\R^2$. 
Let $d_{\triangle}$ denote the geodesic metric on $S_{\triangle}$, and let $\cS(S_{\triangle})$ be an effective enumeration of the set $\Q(S_{\triangle})$ of points in $S_{\triangle}$ with rational coordinates. 
Moreover, all marked points in Figure~\ref{fig:example:Thurston map from barycentric subdivision}, specifically $A,\,B,\,C,\,D,\,E,\,F,\,G,\,H$, are computable (recall Definition~\ref{def:computability and uniform computability of point}). 
Since $g_1$ acts as an affine map on each $1$-tile, we may compute $g_1(x)$ for each point $x$ within a $1$-tile by applying the corresponding matrix transformation. 
By Definition~\ref{def:Algorithm about computable functions}, this ensures that $g_1\colon S^2\rightarrow S^2$ is computable. 

We next construct a homeomorphism $\tau\colon S^2\rightarrow\ccx$ which satisfies that $\tau$ and $\tau^{-1}$ are both computable between computable metric spaces $(S_{\triangle},\,d_{\triangle},\,\cS(S_{\triangle}))$ and $\bigl(\ccx,\,\sigma,\,\cS\bigl(\ccx\bigr)\bigr)$.
Here we identify $\ccx$ with the unit sphere in $\R^3$ via stereographic projection, so that the chordal metric $\sigma$ on $\ccx$ is the restriction of the Euclidean metric in $\R^3$. 
Indeed, we consider the map $\tau$ which maps the equator of $S_{\triangle}$ (the shared boundary of the equilateral triangles) to the equator of $\ccx$. 
More precisely, $\tau$ maps the points $A,\,B,\,C$ to the three equally spaced points along the equator of $\ccx$ and maps the points $H,\,G$ to the south pole and the north pole of $\ccx$, respectively. 
For the rest part of $S_{\triangle}$, the map $\tau$ can be defined by subdivision rule. It is not hard to see that $\tau$ and $\tau^{-1}$ are both computable.

Thus, in the computable metric space $\bigl(\ccx,\,\sigma,\,\cS\bigl(\ccx\bigr)\bigr)$, the map $\widetilde{g}_1 \= \tau\circ g_1\circ\tau^{-1}\colon\ccx\rightarrow\ccx$ is computable and all points in $\crit{\widetilde{g}_1}=\tau(\crit{g_1})$ are computable. 
So far we have verified all conditions in Theorem~\ref{Application B} in the case where $f\=\widetilde{g}_1$. Then by Theorem~\ref{Application B}, the measure of maximal entropy $\widetilde{\mu}_1$ for the map $\widetilde{g}_1$ is computable. Note that $\mu_1(A)=\widetilde{\mu}_1(\tau(A))$ for each $A\in\cB(S_{\triangle})$. 
Then the computability of $\tau^{-1}$ and $\widetilde{\mu}_1$ implies that $\mu_1$ is computable.

Similarly, we can verify all conditions in Theorem~\ref{Application B} in the case where $f\=\widetilde{g}_2$. 
Therefore, by a proof that is verbatim the same as the above, except for replacing $g_1,\,\widetilde{g}_1,\,\mu_1,\,\widetilde{\mu}_1$ with $g_2,\,\widetilde{g}_2,\,\mu_2,\,\widetilde{\mu}_2$, it follows that $\mu_2$ is computable.

\begin{figure}[H]
	\vspace*{.2cm}
	\centering
	\begin{overpic}
		[width=13cm, tics=20]{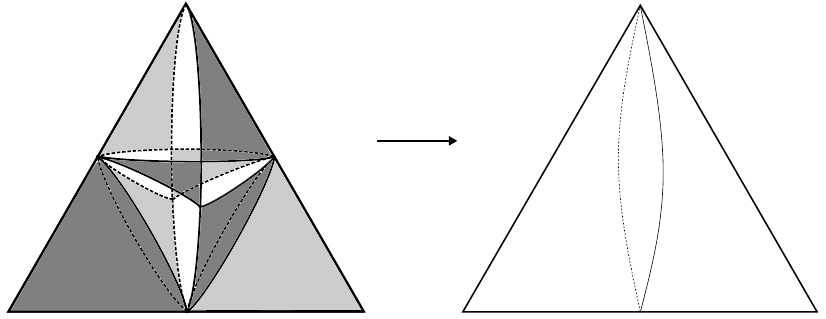}
		\put(50,23){$g_2$}
		\put(94,20){$S_{\triangle}$}
		
		\put(21.5,39){$A$}
		\put(-2,-2){$B$}
		\put(44.5,-2){$C$}
		\put(21.5,-2){$D$}
		\put(34,19.2){$E$}
		\put(8.9,19.2){$F$}
		\put(25.2,19.7){$G$}
		\put(17.8,21.2){$H$}
		\put(25,11.2){$I$}
		\put(18.4,12){$J$}
		
		\put(77,39){$A$}
		\put(53.5,-2.5){$B$}
		\put(100,-2.5){$C$}
		
	\end{overpic}
	\vspace*{.2cm}
	\caption{Expanding Thurston map $g_2$.}
	\label{fig:example:Thurston map with periodic critical points}
\end{figure}
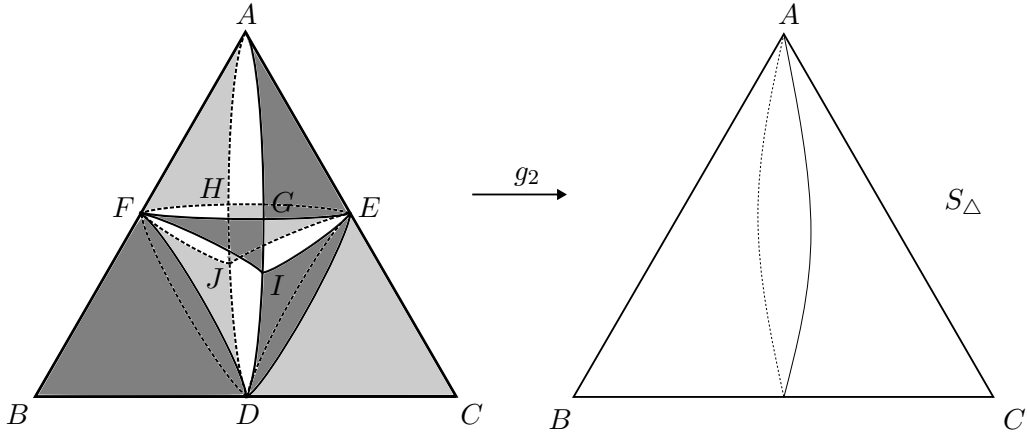


\begin{cor}\label{cor:computabilityofmeasureofmaximalentropyforbarycentricg1}
     For each $i\in\{1,\,2\}$, the measure of maximal entropy of the expanding Thurston map $g_i$ is computable.
\end{cor}


\begin{thebibliography}{9999999999}

\bibitem[BBRY11]{binder2011computability}
        \textsc{Binder,~I.}, \textsc{Braverman,~M.}, \textsc{Rojas,~C.}, and \textsc{Yampolsky,~M.},
		Computability of Brolin--Lyubich measure.
		{\it Comm.\ Math.\ Phys.} {\bf 308} (2011), 743--771.

\bibitem[BHLZ25]{binder2025computability}
        \textsc{Binder,~I.}, \textsc{He},~Qiandu, \textsc{Li},~Zhiqiang, and \textsc{Zhang},~Yiwei,
		On computability of equilibrium states.
		{\it Int.\ Math.\ Res.\ Not.\ IMRN} {\bf 2025} (2025), article no.~rnaf055.

\bibitem[BHLS25]{binder2025nonuniformly}
        \textsc{Binder,~I.}, \textsc{He},~Qiandu, \textsc{Li},~Zhiqiang, and \textsc{Siqueira,~J.}, 
        Computability of equilibrium states in hyperbolic systems and beyond. Preprint, \url{https://qiandu-he.github.io/Qiandu-He/BHLS.pdf}, 2025.

\bibitem[BRY14]{binder2014computable}
        \textsc{Binder,~I.}, \textsc{Rojas,~C.}, and \textsc{Yampolsky,~M.},
		Computable Carath\'eodory theory.
		{\it Adv.\ Math.} {\bf 265} (2014), 280--312. 
        
\bibitem[BM10]{bonk2010expanding}
        \textsc{Bonk,~M.} and \textsc{Meyer,~D.},
		{\it Expanding Thurston Maps},
		Preprint, (\url{arXiv:1009.3647v1}), 2010.

\bibitem[BM17]{bonk2017expanding}
        \textsc{Bonk,~M.} and \textsc{Meyer,~D.},
		{\it Expanding Thurston Maps}, volume~225 of 
		{\it Mathematical Surveys and Monographs}, Amer.\ Math.\ Soc., Providence, RI, 2017.

\bibitem[BBY12]{sylvain2012thurston}
        \textsc{Bonnot,~S.}, \textsc{Braverman,~M.}, and \textsc{Yampolsky,~M.},
		Thurston equivalence to a rational map is decidable.
		{\it Mosc.\ Math.\ J.} {\bf 12} (2012), 747--763, 884.

\bibitem[Bo75]{bowen1975equilibrium}
        \textsc{Bowen,~R.},
		{\it Equilibrium States and the Ergodic Theory of Anosov Diffeomorphisms},
        volume~470 of {\it Lecture Notes in Math.}, Springer, Berlin, 1975.

\bibitem[BY06]{braverman2006noncomputable}
        \textsc{Braverman,~M.} and \textsc{Yampolsky,~M.},
		Non-computable Julia sets.
		{\it J.\ Amer.\ Math.\ Soc.} {\bf 19} (2006), 551--578.

\bibitem[BY09]{braverman2009computability}
        \textsc{Braverman,~M.} and \textsc{Yampolsky,~M.},
		{\it Computability of Julia Sets},
	    volume~23 of {\it Algorithms and Computation in Mathematics}, Springer, Berlin, 2009.

\bibitem[Bra23]{braverman2022communication}
        \textsc{Braverman,~M.},
		Communication and information complexity. In {\it Proc.\ Internat.\ Congr.\ Math.} (2022), Volume~I, EMS Press, Berlin, 2023, pp.~284--320.

\bibitem[Bri94]{Bridges1994computability}
        \textsc{Bridges,~D.S.},
		{\it Computability: A Mathematical Sketchbook},
		Springer, New York, 1994.

\bibitem[Bro65]{brolin1965invariant}
        \textsc{Brolin,~H.},
		Invariant sets under iteration of rational functions.
		{\it Ark.\ Mat.} {\bf 6} (1965), 103--144.

\bibitem[Ca94]{cannon1994combinatorial}
        \textsc{Cannon,~J.W.},
		The combinatorial Riemann mapping theorem.
		{\it Acta Math.} {\bf 173} (1994), 155--234.

\bibitem[Co12]{coudène2016ergodic}
        \textsc{Coud\`ene,~Y.},
		{\it Th\'eorie Ergodique et Syst\`emes Dynamiques},
        EDP Sciences, Les Ulis \& CNRS \'Editions, Paris, 2012.

\bibitem[DeT17]{demers2017equilibrium}
        \textsc{Demers,~M.F.} and \textsc{Todd,~M.},
		Equilibrium states, pressure and escape for multimodal maps with holes.
		{\it Israel J.\ Math.} {\bf 221} (2017), 367--424.

\bibitem[DU91]{Denker1991ergodic}
        \textsc{Denker,~M.} and \textsc{Urba\'nski,~M.},
		Ergodic theory of equilibrium states for rational maps.
		{\it Nonlinearity} {\bf 4} (1991), 103--134.

\bibitem[DoT23]{dobbs2023free}
        \textsc{Dobbs,~N.} and \textsc{Todd,~M.},
		Free energy and equilibrium states for families of interval maps.
		{\it Mem.\ Amer.\ Math.\ Soc.} {\bf 286} (2023), v+103.

\bibitem[Do68]{dobruschin1968description}
        \textsc{Dobru\v{s}in,~R.L.},
		Description of a random field by means of conditional probabilities and conditions for its regularity.
		{\it Teor.\ Verojatnost.\ i Primenen.} {\bf 13} (1968), 201--229.

\bibitem[DY21]{dudko2021computational}
        \textsc{Dudko,~A.} and \textsc{Yampolsky,~M.},
		On computational complexity of Cremer Julia sets.
		{\it Fund.\ Math.} {\bf 252} (2021), 343--353.


\bibitem[GHR11]{galatolo2011dynamics}
        \textsc{Galatolo,~S.}, \textsc{Hoyrup,~M.}, and \textsc{Rojas,~C.},
		Dynamics and abstract computability: computing invariant measures.
		{\it Discrete Contin.\ Dyn.\ Syst.} {\bf 29} (2011), 193--212.

\bibitem[HP09]{haissinsky2009coarse}
        \textsc{Ha\"issinsky,~P.} and \textsc{Pilgrim,~K.M.},
		Coarse expanding conformal dynamics.
		{\it Ast\'erisque} {\bf 325} (2009), viii+139.

\bibitem[HT03]{hawkins2003maximal}
        \textsc{Hawkins,~J.} and \textsc{Taylor,~M.},
		Maximal entropy measure for rational maps and a random iteration algorithm for Julia sets.
		{\it Intl.\ J.\ of Bifurcation and Chaos} {\bf 13} (2003), 1442--1447.

\bibitem[He25]{Qiandu2025Lecture}
        \textsc{He},~Qiandu,
		Notes on computability and ergodic theory.
		Preprint, \url{https://qiandu-he.github.io/Qiandu-He/notes.pdf}, 2025.

\bibitem[Hei01]{heinonen2001lectures}
        \textsc{Heinonen,~J.},
		{\it Lectures on Analysis on Metric Spaces}, 
        Springer, New York, 2001.

\bibitem[HR09]{hoyrup2009computability}
        \textsc{Hoyrup,~M.} and \textsc{Rojas,~C.},
		Computability of probability measures and Martin-L\"of randomness over metric spaces.
		{\it Inform.\ and Comput.} {\bf 207} (2009), 830--847.

\bibitem[HS94]{hubbard49spider}
        \textsc{Hubbard,~J.H.} and \textsc{Schleicher,~D.},
		The spider algorithm.
		In {\it Complex Dynamical Systems}, 155--180, volume~49 of {\it Proc.\ Sympos.\ Appl.\ Math.}, 
        Amer.\ Math.\ Soc., Providence, RI, 1994.

\bibitem[IT10]{iommi2010natural}
        \textsc{Iommi,~G.} and \textsc{Todd,~M.},
		Natural equilibrium states for multimodal maps.
		{\it Comm.\ Math.\ Phys.} {\bf 300} (2010), 65--94.

\bibitem[Ke95]{kechris2012classical}
        \textsc{Kechris,~A.S.},
		{\it Classical Descriptive Set Theory},
		Springer, New York, 1995.

\bibitem[Li15]{li2015weak}
        \textsc{Li},~Zhiqiang,
		Weak expansion properties and large deviation principles for expanding Thurston maps.
		{\it Adv.\ Math.} {\bf 285} (2015), 515--567.

\bibitem[Li17]{li2017ergodic}
        \textsc{Li},~Zhiqiang,
		{\it Ergodic Theory of Expanding Thurston Maps},
		volume~4 of {\it Atlantis Stud.\ Dyn.\ Syst.}, Atlantis Press, 2017.

\bibitem[Li18]{li2018equilibrium}
        \textsc{Li},~Zhiqiang,
		Equilibrium states for expanding Thurston maps.
		{\it Comm.\ Math.\ Phys.} {\bf 357} (2018), 811--872.

\bibitem[LS24a]{shi2024uniqueness}
        \textsc{Li},~Zhiqiang and \textsc{Shi},~Xianghui,
		Thermodynamic formalism for subsystems of expanding Thurston maps II.
		Preprint, \url{arXiv:2404.07247}, 2024.

\bibitem[LS24b]{shi2024entropy}
        \textsc{Li},~Zhiqiang and \textsc{Shi},~Xianghui,
		Entropy density and large deviation principles without upper semi-continuity of entropy.
		Preprint, \url{arXiv:2406.01712}, 2024.

\bibitem[LSZ25]{shi2025thermodynamic}
        \textsc{Li},~Zhiqiang, \textsc{Shi},~Xianghui, and \textsc{Zhang},~Yiwei,
		Thermodynamic formalism for subsystems of expanding Thurston maps and large deviations asymptotics.
		{\it Proc.\ Lond.\ Math.\ Soc.\ (3)} {\bf 130} (2025), article no.~e70019.

\bibitem[LZ25]{li2025ground}
        \textsc{Li},~Zhiqiang and \textsc{Zhang},~Yiwei,
		Ground states and periodic orbits for expanding Thurston maps.
		{\it Math.\ Ann.} {\bf 391} (2025), 3913--3985.

\bibitem[Ly82]{lyubich1982maximum}
        \textsc{Lyubich,~M.Yu.},
		The maximum-entropy measure of a rational endomorphism of the Riemann sphere.
		{\it Funct.\ Anal.\ Appl.} {\bf 16} (1982), 309--311.

\bibitem[Ly83]{ljubichEntropyPropertiesRational1983}
        \textsc{Lyubich,~M.Yu.},
		Entropy properties of rational endomorphisms of the Riemann sphere.
		{\it Ergodic Theory Dynam.\ Systems} {\bf 3} (1983), 351--385.

\bibitem[PU10]{przytycki2010conformal}
        \textsc{Przytycki,~F.} and \textsc{Urba\'nski,~M.},
		{\it Conformal Fractals: Ergodic Theory Methods},
		volume~371 of {\it London Math.\ Soc.\ Lecture Note Ser.}, Cambridge Univ.\ Press, Cambridge, 2010.

\bibitem[RSY20]{rafi2020centralizers}
        \textsc{Rafi,~K.}, \textsc{Selinger,~N.}, and \textsc{Yampolsky,~M.},
		Centralizers in mapping class groups and decidability of Thurston equivalence.
		{\it Arnold Math.\ J.} {\bf 6} (2020), 271--290.

\bibitem[Ro49]{rohlin1949fundamental}
        \textsc{Rohlin,~V.A.},
		On the fundamental ideas of measure theory.
		{\it Mat.\ Sbornik N.S.} {\bf 25} (1949), 107--150.

\bibitem[RY21a]{rojas2021computable}
        \textsc{Rojas,~C.} and \textsc{Yampolsky,~M.},
		Computable geometric complex analysis and complex dynamics.
		In {\it Handbook of Computability and Complexity in Analysis}, 143--172, volume of {\it Theory Appl.\ Comput.}, Springer, Cham, 2021.

\bibitem[RY21b]{rojas2021real}
        \textsc{Rojas,~C.} and \textsc{Yampolsky,~M.},
		Real quadratic Julia sets can have arbitrarily high complexity.
		{\it Found.\ Comput.\ Math.} {\bf 21} (2021), 59--69.

\bibitem[Sa99]{sarig1999thermodynamic}
        \textsc{Sarig,~O.M.},
		Thermodynamic formalism for countable Markov shifts.
		{\it Ergodic Theory Dynam.\ Systems} {\bf 19} (1999), 1565--1593.

\bibitem[SY15]{selinger2015constructive}
        \textsc{Selinger,~N.} and \textsc{Yampolsky,~M.},
		Constructive geometrization of Thurston maps and decidability of Thurston equivalence.
		{\it Arnold Math.\ J.} {\bf 1} (2015), 361--402.

\bibitem[Si72]{sinai1972gibbs}
        \textsc{Sinai,~Ya.G.},
		Gibbs measures in ergodic theory.
		{\it Russian Mathematical Surveys} {\bf 27} (1972), 21--69.

\bibitem[Vi09]{villani2009optimal}
        \textsc{Villani,~C.},
		{\it Optimal Transport},
		volume~338 of {\it Grundlehren der mathematischen Wissenschaften}, Springer, Berlin, 2009.

\bibitem[Wa82]{walters1982introduction}
        \textsc{Walters,~P.},
		{\it An Introduction to Ergodic Theory},
		Springer, New York, 1982.

\bibitem[Wa92]{walters1992differentiability}
        \textsc{Walters,~P.},
		Differentiability properties of the pressure of a continuous transformation on a compact metric space.
		{\it J.\ London Math.\ Soc.\ (2)} {\bf 46} (1992), 471--481.

\bibitem[We00]{weihrauch2000computable}
        \textsc{Weihrauch,~K.},
		{\it Computable Analysis},
		volume of {\it Texts in Theoretical Computer Science. An EATCS Series}, Springer, Berlin, 2000.

\end{thebibliography}
\end{document}